\documentclass{amsart}
\usepackage{tensor}
\usepackage{braket}
\usepackage{booktabs}
\usepackage{mysymbols}
\usepackage{bm}
\title{Asymptotics of ACH-Einstein metrics}
\author{Yoshihiko Matsumoto}
\thanks{Partially supported by Grant-in-Aid for JSPS Fellows (22-6494)}
\address{Graduate School of Mathematical Sciences, The University of Tokyo}
\email{yoshim@ms.u-tokyo.ac.jp}

\theoremstyle{plain}
\newtheorem{thm}{Theorem}[section]
\newtheorem{prop}[thm]{Proposition}
\newtheorem{lem}[thm]{Lemma}
\newtheorem{cor}[thm]{Corollary}

\theoremstyle{definition}
\newtheorem{dfn}[thm]{Definition}
\theoremstyle{remark}
\newtheorem{rem}[thm]{Remark}
\newtheorem{ex}[thm]{Example}
\newtheorem*{IndexRule}{Rule for the index notation}

\numberwithin{equation}{section}
\makeatletter
	
	\@addtoreset{table}{section}
\makeatother

\newcommand{\ThetaTangentX}{\tensor*[^{[\Theta]}]{TX}{}}
\newcommand{\ThetaTangentXFiber}{\tensor*[^{[\Theta]}]{T_pX}{}}

\newcommand{\ThetaCotangentX}{\tensor*[^{[\Theta]}]{T^*X}{}}
\newcommand{\CpxThetaCotangentX}{\tensor*[^{[\Theta]}]{T^*_\bbC X}{}}
\newcommand{\interior}[1]{\mathring{#1}}
\newcommand{\conjalpha}{{\conj\alpha}}
\newcommand{\conjbeta}{{\conj\beta}}
\newcommand{\conjgamma}{{\conj\gamma}}

\newcommand{\conjsigma}{{\conj\sigma}}
\newcommand{\conjtau}{{\conj\tau}}
\newcommand{\euler}{\rho\partial_\rho}

\DeclareMathOperator{\Fl}{Fl}
\newcommand{\sublaplacian}{\Delta_b}

\begin{document}

\begin{abstract}
	We study the boundary asymptotics of ACH metrics which are formally Einstein.
	In terms of the partially integrable almost CR structure induced on the boundary at infinity,
	existence and uniqueness of such formal asymptotic expansions are studied.
	It is shown that there always exist formal solutions to the Einstein equation if we allow logarithmic terms,
	and that a local CR-invariant tensor arises as the obstruction to the existence of a log-free solution.
	Some properties of this new CR invariant, the CR obstruction tensor, are discussed.
\end{abstract}

\maketitle

\tableofcontents

\section*{Introduction}

Asymptotically complex hyperbolic metrics, or ACH metrics for short, that we study here were
introduced by Epstein, Melrose and Mendoza in a study \cite{EpsteinMelroseMendoza} of the resolvent of
the Laplacian of complete K\"ahler metrics of the form $\partial\conj{\partial}\log(1/r)$ on
a bounded strictly pseudoconvex domain, where $r$ is a boundary defining function.
The purpose of this paper is to discuss the formal asymptotic expansion of ACH-Einstein metrics at the boundary
and geometry of their ``CR infinities.''
The integrability condition satisfied by those CR infinities are in general weaker than the classical one;
Tanno \cite{Tanno} called it the partial integrability condition.
This condition is also natural from the viewpoint of the theory of parabolic geometries
\cite{CapSchichl,CapSlovak}.

The relation between the boundary behavior of Cheng--Yau's complete K\"ahler--Einstein
metric \cite{ChengYau} on a bounded strictly pseudoconvex domain $\Omega\subset\bbC^{n+1}$
and the CR structure of the boundary is a classical object of CR geometry.
Following the pioneering work of Fefferman \cite{Fefferman} on the zero boundary value problem of
the complex Monge--Amp\`ere equation, to which construction of complete K\"ahler--Einstein metrics reduces,
Lee and Melrose \cite{LeeMelrose} proved that its solution admits an asymptotic expansion at $\bdry\Omega$
including logarithmic terms.
Graham \cite{Graham} showed that this expansion is determined by the local CR geometry of $\bdry\Omega$
up to the ambiguity of one real scalar-valued function on $\bdry\Omega$,
and identified the coefficient of the first logarithmic term as the only obstruction
to the existence of a log-free solution.

As for ACH metrics, a perturbation result on the existence and the boundary regularity of solutions of
the Einstein equation is obtained by Biquard \cite{Biquard},
and the existence of asymptotic expansions is studied by Biquard and Herzlich
\cite{BiquardHerzlichAsymptotics}.
Furthermore, in the 4-dimensional case, the asymptotics is investigated more closely by the same authors
\cite{BiquardHerzlichACHE} to obtain a Burns--Epstein type formula.
Our formal analysis applies to any dimension, and so it may serve as a tool to extend
\cite{BiquardHerzlichACHE} to the higher-dimensional case.

Another interesting way to see the study of ACH metrics is comparing it with that of
AH (asymptotically hyperbolic) metrics, which we recall briefly.
A Riemannian metric $g$ defined on the interior $\mathring{X}$
of an $(n+1)$-dimensional smooth manifold-with-boundary $X$ is said to be \emph{asymptotically hyperbolic} if
$\rho^2g$ smoothly and nondegenerately extends up to $\bdry X$, where $\rho$ is a boundary defining function,
and moreover if $\abs{d\rho}_{\rho^2g}=1$ at $\bdry X$.
An AH metric induces a conformal structure $[h]$ on $M=\bdry X$, namely,
the conformal class of the pullback of $\rho^2g$.
The $n$-dimensional conformal manifold $(M, [h])$ is called the \emph{conformal infinity} of
the AH metric $g$, or of the AH manifold $(X, g)$,
and its geometry is studied in connection with the boundary asymptotics of AH-Einstein metrics.

The odd- and even-dimensional conformal geometries have substantially different characters, and
a similarity to CR geometry occurs for the latter one.
Among the central objects studied in recent even-dimensional conformal geometry are
the Fefferman--Graham obstruction tensor $\tensor{\caO}{_i_j}$
\cite{FeffermanGrahamAsterisque,FeffermanGrahamAmbient} (for $n\ge 4$)
and Branson's $Q$-curvature \cite{Branson},
and our work brings their CR counterparts to light.
The conformal obstruction tensor $\tensor{\caO}{_i_j}$ arises as the obstruction to the existence of
a log-free formal solution of the Einstein equation for AH metrics.
In the same spirit we shall construct a CR version,
which will be denoted by $\tensor{\caO}{_\alpha_\beta}$.
On the other hand, our analysis of ACH-Einstein metrics is sufficient for doing the same
construction for CR geometries as the scattering-theoretic definition \cite{GrahamZworski} of the $Q$-curvature;
one has only to apply the discussion of Guillarmou and S\`a Barreto \cite{GuillarmouSaBarreto} to our metric.
This is already done for integrable CR manifolds by Hislop, Perry and Tang \cite{HislopPerryTang}
using the asymptotics of complete K\"ahler--Einstein metrics obtained in \cite{Fefferman},
but the ACH approach has an advantage that
it can extend the notion of CR $Q$-curvature naturally to the case of partially integrable almost CR manifolds.
On this line, just like in the conformal case \cite{GrahamHirachi},
one can show that the CR obstruction tensor $\tensor{\caO}{_\alpha_\beta}$ is
equal to the variation of the total CR $Q$-curvature with respect to modifications of partially integrable almost
CR structures preserving the contact distribution.
The details about the CR $Q$-curvature is discussed elsewhere.

By definition a $(2n+1)$-dimensional almost CR manifold $(M, T^{1,0})$, where $T^{1,0}$ is the CR holomorphic
tangent bundle, is \emph{partially integrable} if and only if
\begin{equation}
	[C^\infty(M,T^{1,0}),C^\infty(M,T^{1,0})]\subset C^\infty(M,T^{1,0}\oplus\conj{T^{1,0}}).
	\label{eq:PartialIntegrability}
\end{equation}
The sum $T^{1,0}\oplus\conj{T^{1,0}}$ is identical to the complexification $H_\bbC$ of a certain real subbundle
$H$ of $TM$.
(We may also describe an almost CR structure on $M$ as a pair $(H,J)$,
where $J\in\End H$, $J^2=-1$ and $T^{1,0}$ is the $i$-eigenbundle of $J$.)
The nonintegrability of $(M,T^{1,0})$ is measured by the \emph{Nijenhuis tensor}
$N\in C^\infty(M,H^*\otimes H^*\otimes H)$ defined by
\begin{equation*}
	N(X,Y):=\conj{\Pi^{1,0}}[\Pi^{1,0}X,\Pi^{1,0}Y]+\Pi^{1,0}[\conj{\Pi^{1,0}}X,\conj{\Pi^{1,0}}Y],
	\qquad\text{$X$, $Y\in C^\infty(M,H)$},
\end{equation*}
where $\Pi^{1,0}$ and $\conj{\Pi^{1,0}}$ are the projections
from $H_\bbC$ onto $T^{1,0}$ and $\conj{T^{1,0}}$, respectively.
We extend $N$ complex bilinearly (such extensions will be done in the sequel without notice).
Given a local frame $\set{Z_\alpha}$ of $T^{1,0}$,
we put $Z_\conjalpha=\conj{Z_\alpha}$ and write
$N(Z_\alpha,Z_\beta)=\tensor{N}{_\alpha_\beta^\conjgamma}Z_\conjgamma$.

A partially integrable almost CR manifold $(M,T^{1,0})$ is said to be \emph{nondegenerate} if
$H$ is a contact distribution.
In this case the conormal bundle $E\subset T^*M$ of $H$ is orientable,
and hence $E^\times:=E\setminus(\text{zero section})$ splits into two $\bbR^+$-bundles.
We fix one of them and call its sections \emph{pseudohermitian structures.}
A choice of a pseudohermitian structure $\theta$ defines the \emph{Levi form} $h$ by
\begin{equation}
	h(X,Y):=d\theta(X,JY),\qquad \text{$X$, $Y\in C^\infty(M,H)$}.
	\label{eq:LeviForm}
\end{equation}
Thanks to the nondegeneracy and the partial integrability, $h$ is a nondegenerate hermitian form.
For another pseudohermitian structure $\Hat{\theta}=e^{2u}\theta$, we have $\Hat{h}=e^{2u}h$.
In particular, the signature $(p,q)$ of the Levi form, $p+q=n$, is independent of the choice of $\theta$.
Once we fix a pseudohermitian structure,
$\tensor{h}{_\alpha_\conjbeta}=h(Z_\alpha,Z_\conjbeta)$
and its dual $\tensor{h}{^\alpha^\conjbeta}$ allows us to lower and raise indices of various tensors.

According to \cite{EpsteinMelroseMendoza}, ACH metrics are certain class of
$[\Theta]$-metrics\footnote{Originally, \cite{EpsteinMelroseMendoza} uses ``$\Theta$'' instead of ``$[\Theta]$.''
For example, they say ``$\Theta$-metrics'' rather than ``$[\Theta]$-metrics,''
``$\Theta$-tangent bundle'' than ``$[\Theta]$-tangent bundle,''
and write ``$\tensor*[^{\Theta}]{TX}{}$'' than ``$\ThetaTangentX$.''
However, we want to emphasize that they depend only on the conformal class of $\Theta$'s.}.
Although we shall detail the basics of ACH metrics in \S1, we include here a brief account of relevant definitions.
A \emph{$\Theta$-structure} on a smooth manifold-with-boundary $X$ is a smooth section of $T^*X|_{\bdry X}$
such that the 1-form $\iota^*\Theta$ on $\bdry X$ is nowhere vanishing,
where $\iota\colon\bdry X\hookrightarrow X$ is the inclusion map.
We call a conformal class $[\Theta]$ of $\Theta$-structures a \emph{$[\Theta]$-structure,}
and a pair $(X, [\Theta])$ a \emph{$[\Theta]$-manifold}.
Associated to it the \emph{$[\Theta]$-tangent bundle} $\ThetaTangentX$,
which is derived from the usual tangent bundle $TX$ by blowing up the zero section over $\bdry X$.
By definition \emph{$[\Theta]$-metrics} are nondegenerate fiber metrics of $\ThetaTangentX$.
Let $(X,[\Theta])$ be a $(2n+2)$-dimensional $[\Theta]$-manifold and suppose that
a nondegenerate partially integrable almost CR structure $T^{1,0}$ is given on $M=\bdry X$.
We assume that they are \emph{compatible} in the sense that $\iota^*[\Theta]$ agrees with the conformal class
$[\theta]$ of pseudohermitian structures of $(M,T^{1,0})$.
By the fact that $\ker\iota^*[\Theta]=H\subset TM$ is a contact distribution, there is a natural filtration
$K_2\subset K_1\subset\ThetaTangentX|_{\bdry X}$ by subbundles, where $K_1$ is of rank $2n+1$ and
$K_2$ of rank $1$.
Any $[\Theta]$-metric $g$ with some positivity condition induces an orthogonal decomposition
$\ThetaTangentX|_{\bdry X}=R\oplus K_1$, $K_1=K_2\oplus L$.
The bundle $L$ is identified with $H$ up to a conformal factor,
and thus we have another decomposition $L_\bbC=L^{1,0}\oplus\conj{L^{1,0}}$.
The definition of the notion of ACH metrics (with \emph{CR infinity} $(M,T^{1,0})$) is given in \S1
in terms of these ingredients.
By a \emph{distinguished local frame} $\set{W_\infty,W_0,W_\alpha,W_\conjalpha}$ for an ACH metric we mean
a local frame of $\ThetaTangentX$ near $\bdry X$ such that, if restricted to $\bdry X$,
$W_\infty$ generates $R$, $W_0$ generates $K_2$, and $W_1$, $\dotsc$, $W_n$ span $L^{1,0}$.
When we simply say ``ACH metrics,'' they are always smooth up to the boundary.

For example, if $(M,T^{1,0})$ is an arbitrary nondegenerate partially integrable almost CR manifold,
$X=M\times [0,\infty)_\rho$ carries a standard $[\Theta]$-structure which is compatible with $T^{1,0}$. Then
\begin{equation*}
	g_\theta=\dfrac{4}{\rho^2}d\rho^2+\dfrac{1}{\rho^4}\theta^2+\dfrac{1}{\rho^2}h
\end{equation*}
gives a standard model for ACH metrics with CR infinities $(M,T^{1,0})$,
where $\theta$ is any pseudohermitian structure on $(M,T^{1,0})$.
For arbitrary $[\Theta]$-manifold $(X,[\Theta])$ with $\bdry X=M$,
a $[\Theta]$-metric $g$ on $X$ is an ACH metric with CR infinity $(M,T^{1,0})$ if and only if,
when $g$ is seen as a Riemannian metric on $\interior{X}$, for some choice of $\theta$
and for some smooth $[\Theta]$-preserving diffeomorphism $\Phi$ between neighborhoods of the boundaries of
$M\times[0,\infty)$ and $X$, $\Phi^*g\sim g_\theta$ holds
in the sense that their difference is $O(\rho)$ with respect to $g_\theta$.
This is an immediate consequence of Proposition \ref{prop:NormalFormOfACHMetric}.
If $T$ is a vector field on $M$ which is transverse to $H$,
the set of vector fields $\set{\rho\partial_\rho,\rho^2 T,\rho Z_\alpha,\rho Z_\conjalpha}$ on $M\times(0,\infty)$
extends to a frame of $[\Theta]$-tangent bundle of $M\times[0,\infty)$ and gives a distinguished local frame for
$\Phi^*g$.
By pushing it forward, we have a distinguished local frame of $g$.

The first main theorem in this paper is on the existence of an approximate solution of the Einstein equation.
For any ACH metric $g$, its Ricci tensor is naturally defined as a symmetric 2-tensor over $\ThetaTangentX$.
We define the \emph{Einstein tensor} by $\Ein:=\Ric+\frac{1}{2}(n+2)g$.

\begin{thm}
	\label{thm:MainTheoremOnExistence}
	Let $(X,[\Theta])$ be a $(2n+2)$-dimensional $[\Theta]$-manifold
	and $T^{1,0}$ a compatible nondegenerate partially integrable almost CR structure on $M=\bdry X$.
	Then there exists an ACH metric $g$ with CR infinity $(M,T^{1,0})$ satisfying
	\begin{equation}
		\label{eq:ApproximateEinsteinACH_Intro}
		\begin{split}
			\tensor{\Ein}{_\infty_\infty}&=O(\rho^{2n+4}),\qquad
			\tensor{\Ein}{_\infty_0}=O(\rho^{2n+4}),\qquad
			\tensor{\Ein}{_\infty_\alpha}=O(\rho^{2n+3}),\\
			\tensor{\Ein}{_0_0}&=O(\rho^{2n+4}),\qquad
			\tensor{\Ein}{_0_\alpha}=O(\rho^{2n+3}),\\
			\tensor{\Ein}{_\alpha_{\conj{\beta}}}&=O(\rho^{2n+3}),\qquad
			\tensor{\Ein}{_\alpha_\beta}=O(\rho^{2n+2})
		\end{split}
	\end{equation}
	with respect to any distinguished local frame $\set{W_\infty,W_0,W_\alpha,W_\conjalpha}$ of
	$\ThetaTangentX$ near the boundary,
	where $\rho$ is any boundary defining function of $X$.
\end{thm}

Note that the condition \eqref{eq:ApproximateEinsteinACH_Intro} is independent of the choice of
a distinguished local frame and a boundary defining function.

Construction of better approximate solutions is obstructed by a tensor $\tensor{\caO}{_\alpha_\beta}$
on the boundary, which is called the \emph{CR obstruction tensor} and is defined as follows.
Let $g$ be any ACH metric satisfying \eqref{eq:ApproximateEinsteinACH_Intro} and
$\theta\in\iota^*[\Theta]$ a pseudohermitian structure on $\bdry X$.
Then there is a special boundary defining function $\rho$ for $\theta$, which satisfies
$\abs{d\rho/\rho}_g=1/2$ near $\bdry X$ and $\iota^*(\rho^4g)=\theta^2$.
We set
\begin{equation*}
	\tensor{\caO}{_\alpha_\beta}:=\left.\left(\rho^{-2n-2}\tensor{\Ein}{_\alpha_\beta}\right)\right|_{\bdry X}
\end{equation*}
in terms of the Einstein tensor of $g$.
This is well-defined, i.e., this does not depend on the choice of $g$,
and is a natural pseudohermitian invariant of $(M,T^{1,0},\theta)$.
We shall prove that for $\Hat{\theta}=e^{2u}\theta$ it holds that
\begin{equation}
	\label{eq:IntroTransformationLawOfObstructionTensor}
	\tensor{\Hat\caO}{_\alpha_\beta}=e^{-2nu}\tensor{\caO}{_\alpha_\beta}.
\end{equation}

Let $\zeta$ be the section of the CR canonical bundle $K=\bigwedge^{n+1}(\conj{T^{1,0}})^\perp$ of $M$
associated to $\theta$ in such a way that Farris' volume normalization condition \cite{Farris}
\begin{equation}
	\label{eq:FarrisSectionOfCanonicalBundle}
	\theta\wedge(d\theta)^{n}=i^{n^{2}}n!(-1)^{q}\theta\wedge(T\intprod\zeta)\wedge(T\intprod\conj{\zeta}),
\end{equation}
where the signature of the Levi form is $(p,q)$, is satisfied.
We define the density-weighted version of the CR obstruction tensor by
\begin{equation*}
	\tensor{\bm{\caO}}{_\alpha_\beta}:=\tensor{\caO}{_\alpha_\beta}\otimes\abs{\zeta}^{2n/(n+2)}
	\in\tensor{\caE}{_(_\alpha_\beta_)}(-n,-n).
\end{equation*}
Similarly we set $\tensor{\bm{A}}{_\alpha_\beta}=\tensor{A}{_\alpha_\beta}$,
$\tensor{\bm{N}}{_\alpha_\beta^\conjgamma}=\tensor{N}{_\alpha_\beta^\conjgamma}$,
where $A$ is the pseudohermitian torsion tensor associated to a choice of $\theta$, and
indices of such density-weighted tensors are lowered and raised using
$\tensor{\bm{h}}{_\alpha_\conjbeta}=\tensor{h}{_\alpha_\conjbeta}\otimes\abs{\zeta}^{-2/(n+2)}$
and its dual $\tensor{\bm{h}}{^\alpha^\conjbeta}$.
Then we have the following results, the first of which is just another expression of
\eqref{eq:IntroTransformationLawOfObstructionTensor}.

\begin{thm}
	\label{thm:OnObstructionTensor}
	(1) The density-weighted CR obstruction tensor $\tensor{\bm{\caO}}{_\alpha_\beta}$ is a CR invariant.

	(2) For an integrable CR manifold, $\tensor{\bm{\caO}}{_\alpha_\beta}$ vanishes.

	(3) Let $\tensor{\bm{D}}{^\alpha^\beta}$ be a differential operator
	$\tensor{\caE}{_(_\alpha_\beta_)}(-n,-n)\to\caE(-n-2,-n-2)$ defined by
	\begin{equation*}
		\tensor{\bm{D}}{^\alpha^\beta}
		=\tensor{\nabla}{^\alpha}\tensor{\nabla}{^\beta}-i\tensor{\bm{A}}{^\alpha^\beta}
		-\tensor{\bm{N}}{^\gamma^\alpha^\beta}\tensor{\nabla}{_\gamma}
		-\tensor{\bm{N}}{^\gamma^\alpha^\beta_,_\gamma}.
	\end{equation*}
	Then this is a CR-invariant operator and we have
	$\tensor{\bm{D}}{^\alpha^\beta}\tensor{\bm{\caO}}{_\alpha_\beta}
	-\tensor{\bm{D}}{^{\conj{\alpha}}^{\conj{\beta}}}\tensor{\bm{\caO}}{_{\conj{\alpha}}_{\conj{\beta}}}=0$.
\end{thm}

In spite of (2), there certainly is a partially integrable almost CR manifold for which
$\tensor{\bm{\caO}}{_\alpha_\beta}$ is nonzero as we will see in \S6.
This indicates the importance of studying partially integrable almost CR structures.

We shall also investigate how well the solution is improved if we introduce logarithmic terms to ACH metrics.
A function $f\in C^0(X)\cap C^\infty(\interior{X})$ is said to be an element of $\caA(X)$
if it admits an asymptotic expansion of the form
\begin{equation}
	f\sim \sum_{q=0}^\infty f^{(q)}(\log\rho)^q,\qquad f^{(q)}\in C^\infty(X)
	\label{eq:GeneralLogarithmicExpansion}
\end{equation}
for any boundary defining function $\rho$.
If $f\in\caA(X)$, then the Taylor expansions of $f^{(q)}$ at $\bdry X$ are uniquely determined.
A \emph{singular ACH metric} is a $[\Theta]$-metric $g$ with $\tensor{g}{_I_J}\in\caA(X)$ satisfying
the same condition for usual ACH metrics.
Then the components of its Ricci tensor also belong to $\caA(X)$, and hence so do those of the Einstein tensor.
For any $p\in\bdry X$, we say that $f\in\caA(X)$ vanishes to the infinite order at $p$
if and only if all the coefficients $f^{(q)}$ have the vanishing Taylor expansions at $p$.
A tensor over $\ThetaTangentX$ vanishes to the infinite order at $p$ if and only if all of its components
vanish to the infinite order at $p$.

\begin{thm}
	\label{thm:SingularSolutionToInfiniteOrder}
	Let $(X, [\Theta])$ and $(M, T^{1,0})$ be as in Theorem \ref{thm:MainTheoremOnExistence} and $p\in\bdry X$.
	Then there exists a singular ACH metric whose Einstein tensor vanishes to the infinite order at $p$.
	Furthermore, if $\tensor{\bm{\caO}}{_\alpha_\beta}(p)=0$,
	then there exists such an ACH metric with no logarithmic terms.
\end{thm}

When $M$ is the boundary of a bounded strictly pseudoconvex domain $\Omega$,
the result above has a somewhat peculiar implication.
Recall that the obstruction to the existence of a smooth solution to
the zero boundary value problem for the complex Monge--Amp\`ere equation on $\Omega$
is one scalar-valued function on $\bdry\Omega$.
On the other hand, by Theorem \ref{thm:OnObstructionTensor} (2),
the second assertion of Theorem \ref{thm:SingularSolutionToInfiniteOrder} applies to this case.
Our result says that in the ACH category, at any given point on $\bdry\Omega$,
we can always erase the logarithmic terms.
The author believes that there is some framework that can capture both
Graham's scalar-valued obstruction for integrable CR manifolds and the CR obstruction tensor at the same time.
This might be an interesting topic of further study.

Our result contradicts a work of Seshadri \cite{Seshadri}, which states that there are
a ``primary'' scalar-valued obstruction function and a ``secondary'' 1-tensor obstruction to
the existence of ACH-Einstein metrics without logarithmic terms.
Despite the fact that there is a difference in the definition of ACH metrics,
the conflict is not because of it.
The work \cite{Seshadri} contains some crucial calculation errors in \S4,
where the computation of the Ricci tensor is carried out.
Nevertheless, the influence of Seshadri's paper on our analysis is obvious;
if it were not for it, this work should have been much harder to complete.

The paper is organized as follows.
We first recall some basic facts about ACH metrics in \S\ref{sec:ThetaStructure}.
In \S\ref{sec:PseudohermitianGeometry} we quickly develop a theory of pseudohermitian geometry for
partially integrable almost CR manifolds.
After studying how the Ricci tensor depends on the metric in \S\ref{sec:RicciTensorAndSomeLowOrderTerms} and
\S\ref{sec:HigherOrderPerturbation}, we prove Theorem \ref{thm:MainTheoremOnExistence} and
Theorem \ref{thm:OnObstructionTensor} in \S\ref{sec:ApproximateSolution}.
In \S\ref{sec:FirstVariation}, we calculate the first variation of $\tensor{\caO}{_\alpha_\beta}$ with
respect to the modification of partially integrable almost CR structure on the sphere from the flat one
and verify that a generic small modification gives rise to nonvanishing CR obstruction tensor.
The last section \S\ref{sec:FormalSolution} is devoted to an investigation of singular ACH metrics
and the proof of Theorem \ref{thm:SingularSolutionToInfiniteOrder}.

In this paper the word ``smooth'' means infinite differentiability.
The Einstein summation convention is used throughout.
Parentheses surrounding indices indicate the symmetrization.
Our convention for the exterior product $\omega\wedge\eta$ of 1-forms is
$(\omega\wedge\eta)(X,Y)=\omega(X)\eta(Y)-\omega(Y)\eta(X)$,
while for the symmetric product $\omega\eta$ we observe
$(\omega\eta)(X,Y)=\tfrac{1}{2}(\omega(X)\eta(Y)+\omega(Y)\eta(X))$.

I would like to express my gratitude to Kengo Hirachi
for guidance to this interesting research area and continuous encouragement.
I also wish to thank Takao Akahori, Olivier Biquard, Charles Fefferman, Robin Graham, Colin Guillarmou,
Hiraku Nozawa, Rapha\"el Ponge and Neil Seshadri for helpful advice and discussions.


\section{$[\Theta]$-structures and ACH metrics}
\label{sec:ThetaStructure}

Let $X$ be a smooth manifold-with-boundary and $\Theta\in C^\infty(\bdry X,T^*X|_{\bdry X})$
such that $\iota^*\Theta$ is a nowhere vanishing $1$-form on $\bdry X$,
where $\iota\colon\bdry X\hookrightarrow X$ is the inclusion map.
Then a Lie subalgebra $\caV_{\Theta}$ of $C^\infty(X,TX)$ is defined as follows:
for any boundary defining function $\rho$, a vector field $V$ belongs to $\caV_\Theta$ if and only if
\begin{equation*}
	V\in\rho C^\infty(X,TX),\qquad
	\tilde\Theta(V)\in\rho^{2}C^\infty(X).
\end{equation*}
Here $\tilde\Theta\in C^\infty(X,T^*X)$ is any extension of $\Theta$.
It is clear that $\caV_\Theta$ depends only on the conformal class $[\Theta]$ of $\Theta$,
which we call a \emph{$[\Theta]$-structure} on $X$.
A pair $(X,[\Theta])$ is called a \emph{$[\Theta]$-manifold.}
We shall write $\caV_{[\Theta]}$ instead of $\caV_{\Theta}$.

Now consider a $(2n+2)$-dimensional $[\Theta]$-manifold $(X,[\Theta])$.
There is a canonical vector bundle $\ThetaTangentX$ of rank $2n+2$ over $X$, whose sections are
the elements of $\caV_{[\Theta]}$.
Over the interior of $X$ it is identified with the usual tangent bundle $TX$.
To illustrate the structure near $p\in\bdry X$, 
let $\set{N,T,Y_j}=\set{N, T, Y_1, \dots, Y_{2n}}$ be a local frame of $TX$ in a neighborhood of
$p$ dual to a certain coframe of the form $\set{d\rho,\tilde\Theta,\alpha^j}$,
where $\tilde\Theta$ is an extension of some $\Theta\in[\Theta]$.
Then any $V\in\caV_{[\Theta]}$ is, near $p$, expressed as
\begin{equation}
	V=a\rho N+b\rho^2T+c^{j}\rho Y_j,\qquad
	\text{$a$, $b$, $c^j\in C^\infty(X)$}.
	\label{eq:GeneralFormOfThetaVectorFieldAtBoundary}
\end{equation}
Hence $\set{\rho N,\rho^2 T,\rho Y_j}$ extends to a frame of $\ThetaTangentX$ near $p\in\bdry X$.
The dual frame of the bundle $\ThetaCotangentX:=(\ThetaTangentX)^*$ is
$\set{d\rho/\rho, \tilde\Theta/\rho^2, \alpha^j/\rho}$.
A fiber metric of $\ThetaTangentX$, which is not necessarily positive-definite, is called a
\emph{$[\Theta]$-metric.}

\begin{ex}
	\label{ex:CompleteKahlerMetric}
	Let $\Omega\subset\bbC^{n+1}$ be a bounded strictly pseudoconvex domain.
	Then the boundary $M=\bdry\Omega$ carries a natural strictly pseudoconvex CR structure.
	If $r\in C^\infty(\closure\Omega)$ is a boundary defining function and
	$\tilde\theta:=\tfrac{i}{2}(\partial r-\conj{\partial}r)$, then
	$\theta:=\iota^*\tilde\theta$ is a pseudohermitian structure on $M$,
	where $\iota\colon M\hookrightarrow\closure\Omega$ is the inclusion map.
	We consider the following complete K\"ahler metric $G$ on $\Omega$:
	\begin{equation*}
		G=4\sum_{j,k}\frac{\partial^2}{\partial z^j\partial\conj{z}^k}\left(\log\frac{1}{r}\right)dz^jd\conj{z}^k.
	\end{equation*}
	Let $\xi$ be the unique $(1,0)$ vector field satisfying
	$\xi\intprod\partial\conj{\partial}r=0\mod\conj\partial r$, $\partial r(\xi)=1$ and
	$\nu:=\re\xi$, $\tau:=2\im\xi$.
	Then $dr(\nu)=1$, $dr(\tau)=0$ and $\tilde\theta(\nu)=0$, $\tilde\theta(\tau)=1$.
	We set $\xi\intprod\partial\conj\partial r=\kappa\conj\partial r$, or
	$\kappa=\partial\conj\partial r(\xi,\conj\xi)$.
	Furthermore, we let $\xi_1$, $\dots$, $\xi_n$ be $(1,0)$ vector fields spanning
	$\ker\partial r\subset T^{1,0}\bbC^{n+1}$ near $M$ and $\xi_\conjalpha:=\conj{\xi_\alpha}$.
	The coframe $\set{dr,\tilde\theta,\tilde\theta^\alpha,\tilde\theta^\conjalpha}$ is defined
	as the dual of $\set{\nu,\tau,\xi_\alpha,\xi_\conjalpha}$.
	Then equations (1.3) and (1.4) of \cite{Graham} read, for some set of functions
	$\set{\tensor{\tilde{h}}{_\alpha_\conjbeta}}$,
	\begin{equation}
		G=4(1-r\kappa)\frac{\partial r\conj\partial r}{r^2}
		+4\tensor{\tilde{h}}{_\alpha_\conjbeta}\frac{\tilde\theta^\alpha\tilde\theta^\conjbeta}{r}
		=(1-r\kappa)\frac{dr^2+4\tilde\theta^2}{r^2}
		+4\tensor{\tilde{h}}{_\alpha_\conjbeta}\frac{\tilde\theta^\alpha\tilde\theta^\conjbeta}{r}
		\label{eq:CompleteKahlerMetric}
	\end{equation}
	and
	\begin{equation}
		\partial\conj\partial r=\kappa\partial r\wedge\conj\partial r
		-\tensor{\tilde{h}}{_\alpha_\conjbeta}\tilde\theta^\alpha\wedge\tilde\theta^\conjbeta
		=i\kappa dr\wedge\tilde\theta
		-\tensor{\tilde{h}}{_\alpha_\conjbeta}\tilde\theta^\alpha\wedge\tilde\theta^\conjbeta.
		\label{eq:ExtensionOfLeviForm}
	\end{equation}
	
	The \emph{square root} $X:=\closure\Omega_{1/2}$ of $\closure\Omega$ in the sense of
	\cite{EpsteinMelroseMendoza} is defined as in the following way.
	As a topological manifold, $X$ is identical with $\closure\Omega$.
	The smooth structure on $X$ is given in such a way that the identity maps $\interior{X}\to\Omega$,
	$\bdry X\to M$ are diffeomorphisms and $\rho:=\sqrt{r/2}$ is a boundary defining function for $X$.
	Let $i_{1/2}\colon X\to\overline{\Omega}$ be the identity map and
	$\tilde\Theta$, $\tilde\Theta^\alpha$, $\tilde\Theta^\conjalpha$ the pullbacks of
	$\tilde\theta$, $\tilde\theta^\alpha$, $\tilde\theta^\conjalpha$ by $i_{1/2}$.
	The conformal class $[\Theta]$ of $\Theta:=\tilde\Theta|_{\bdry X}$ is independent of the choice of $r$
	because that of $\tilde\theta|_{\bdry\Omega}$ is independent.
	Then $G$ lifts to the following metric on $\mathring{X}$:
	\begin{equation}
		g:=i_{1/2}^*G
		=4(1-2\rho^2\kappa)\frac{d\rho^2}{\rho^2}+(1-2\rho^2\kappa)\frac{\tilde{\Theta}^2}{\rho^4}
		+2\tensor{\tilde{h}}{_\alpha_\conjbeta}\frac{\tilde\Theta^\alpha}{\rho}\frac{\tilde\Theta^\conjbeta}{\rho}.
		\label{eq:CompleteKahlerMetricLiftedToX}
	\end{equation}
	Since $\set{d\rho/\rho,\tilde\Theta/\rho^2,\tilde\Theta^\alpha/\rho,\tilde\Theta^\conjalpha/\rho}$ is
	a local frame of $\CpxThetaCotangentX$, 
	this expression shows that $g$ extends to a positive-definite $[\Theta]$-metric on $(X,[\Theta])$.
\end{ex}

Let $F_p$, $p\in\bdry X$, be the set of vector fields of the form
\eqref{eq:GeneralFormOfThetaVectorFieldAtBoundary} with $a(p)=b(p)=c^j(p)=0$.
Then the fiber $\ThetaTangentXFiber$ is naturally identified with the quotient vector space $\caV_{[\Theta]}/F_p$.
Since $F_p$ is an ideal, $\ThetaTangentXFiber$ is a Lie algebra, which is called the \emph{tangent algebra} at $p$.
In the sequel we always further assume that
\begin{equation*}
	\text{$\ker\iota^*[\Theta]\subset T(\bdry X)$ is a contact distribution on $\bdry X$;}
\end{equation*}
then the derived series of $\ThetaTangentXFiber$ consists of the following subalgebras:
\begin{equation*}
	K_{1,p}:=\braket{\rho^2T,\rho Y_1,\dots,\rho Y_{2n}}/F_p,\qquad
	K_{2,p}:=\braket{\rho^2T}/F_p.
\end{equation*}
Collecting these subspaces we obtain the subbundles $K_1$ and $K_2$ of $\ThetaTangentX|_{\bdry X}$.

ACH metrics generalize the $[\Theta]$-metrics coming from complete K\"ahler metrics as illustrated
in Example \ref{ex:CompleteKahlerMetric}.
The characterizing features are completely described in terms of the boundary value of $g$.
Our first two assumptions are that
\begin{equation}
	\abs{\dfrac{d\rho}{\rho}}_g=\dfrac{1}{2}\qquad\text{at $\bdry X$}
	\label{eq:TransverseBoundaryValueAssumption}
\end{equation}
and
\begin{equation}
	\text{$g$ is positive-definite on $K_2$}.
	\label{eq:TangentialBoundaryValueAssumption}
\end{equation}
It is clear that \eqref{eq:TransverseBoundaryValueAssumption} is independent of the choice of a boundary
defining function $\rho$.
The condition \eqref{eq:TangentialBoundaryValueAssumption} implies that
if we pull $\rho^4g$, regarded as a section of $\Sym^2T^*X$, back to $\bdry X$ then it is equal to the square of
some contact form that belongs to $\iota^*[\Theta]$.
If there is a fixed $[\Theta]$-metric $g$ satisfying these two conditions,
then for any $p\in\bdry X$ there is a unique orthogonal decomposition
\begin{equation}
	\label{eq:OrthogonalDecompositionOfTangentAlgebra}
	\ThetaTangentXFiber=R_p\oplus K_{1,p},\qquad K_{1,p}=K_{2,p}\oplus L_p.
\end{equation}
The subbundle of $\ThetaTangentX|_{\bdry X}$ whose fiber at $p$ is $L_p$ is denoted by $L$.

Let $H\subset T(\bdry X)$ be the kernel of $\iota^*[\Theta]$.
Given a boundary defining function $\rho$, there is a vector-bundle isomorphism
\begin{equation}
	\lambda_\rho\colon H\to L,\qquad Y_p\mapsto \pi_p(\rho Y\text{ mod $F_p$}),
	\label{eq:LeviBundleIsomorphism}
\end{equation}
where $Y\in C^\infty(X,TX)$ is any extension of $Y_p\in H_p$ and
$\pi_p\colon K_{1,p}\to L_p$ is the projection with respect to the decomposition
\eqref{eq:OrthogonalDecompositionOfTangentAlgebra}.
By a \emph{compatible almost CR structure} for $[\Theta]$ we mean any nondegenerate partially integrable
almost CR structure $T^{1,0}$ on $\bdry X$
for which the conformal class $[\theta]$ of pseudohermitian structures is equal to $\iota^*[\Theta]$.

\begin{dfn}
	\label{dfn:ACHMetric}
	Let $(X,[\Theta])$ a $[\Theta]$-manifold.
	An \emph{ACH metric} on $(X,[\Theta])$ is
	a $[\Theta]$-metric $g$ satisfying \eqref{eq:TransverseBoundaryValueAssumption},
	\eqref{eq:TangentialBoundaryValueAssumption} and the following additional conditions:
	\begin{enumerate}
		\item For any $p\in\bdry X$, if $r_p\in R_p$ is the vector such that $(d\rho/\rho)_p(r_p)=1$,
			then the map $L_p\to \ThetaTangentXFiber$, $Z_p\mapsto [r_p,Z_p]$, is the identity map onto $L_p$;
			\label{item:ACHConditionOnDecompositionOfK1}
		\item There is a compatible almost CR structure $T^{1,0}$ such that,
			for some (hence for any) boundary defining function $\rho$
			and a pseudohermitian structure $\theta\in\iota^*[\Theta]$
			characterized by $\iota^*(\rho^4g)=\theta^2$,
			via \eqref{eq:LeviBundleIsomorphism} $g|_L$ agrees with the Levi form on $H$ determined by $\theta$.
			\label{item:ACHConditionOnLeviBundle}
	\end{enumerate}
\end{dfn}

The condition (\ref{item:ACHConditionOnDecompositionOfK1}) above is independent of the choice of $\rho$.
On (\ref{item:ACHConditionOnLeviBundle}), the assumptions of partial integrability and nondegeneracy for $T^{1,0}$
are not restrictive here, since if $\lambda_\rho^*(g|_L)=(d\theta)|_H(\paramdot,J\paramdot)$ holds for
an almost CR structure $(H,J)$ on $M=\bdry X$ satisfying $\ker\iota^*[\Theta]=H$,
then $(d\theta)|_H(\paramdot,J\paramdot)$ is symmetric and hence hermitian,
which implies that $(H,J)$ is partially integrable, and its nondegeneracy is nothing
but the contact condition for $\iota^*[\Theta]$ that we keep imposing.
Furthermore, because of the contact condition, $(H,J)$ is unique.
We say that $(M,T^{1,0})$ is the \emph{CR infinity} of the \emph{ACH manifold} $(X,[\Theta],g)$.

\begin{rem}
	\label{rem:ReformulationOfFirstConditionOfACH}
	Let $g$ be a $[\Theta]$-metric on $(X,[\Theta])$ satisfying \eqref{eq:TransverseBoundaryValueAssumption} and
	\eqref{eq:TangentialBoundaryValueAssumption}.
	We further assume that we have a local frame $\set{N,T,Y_j}$ of $TX$ around $p\in\bdry X$, which is dual to
	$\set{d\rho,\tilde\Theta,\alpha^j}$ for an extension $\tilde\Theta$ of some $\Theta\in[\Theta]$,
	such that $d\tilde\Theta(N,Y_j)=-\tilde\Theta([N,Y_j])=O(\rho)$ and $R_p=\braket{\rho N}/F_p$.
	Then, since $r_p=(\rho N)_p$ and $[\rho N,\rho^2T]=2\rho^2T$, $[\rho N,\rho Y_j]=\rho Y_j$ mod $F_p$,
	the map $L_p\to\ThetaTangentXFiber$, $Z_p\mapsto[r_p,Z_p]$ is the identity if and only if
	$L_p=\braket{\rho Y_1,\dots,\rho Y_{2n}}/F_p$.
\end{rem}

Let $L_\bbC=L^{1,0}\oplus\conj{L^{1,0}}$ be the decomposition corresponding to
$H_\bbC=T^{1,0}\oplus\conj{T^{1,0}}$.
A \emph{distinguished local frame} $\set{W_\infty,W_0,W_\alpha,W_\conjalpha}$ for an ACH metric $g$ is
a local frame of $\ThetaTangentX$ near a point on $\bdry X$ such that, if restricted to $\bdry X$,
$W_\infty$ generates $R$, $W_0$ generates $K_2$,
$W_1$, $\dots$, $W_n$ span $L^{1,0}$, and $W_\conjalpha=\conj{W_\alpha}$.

\begin{prop}
	Let $\Omega\subset\bbC^{n+1}$ be a bounded strictly pseudoconvex domain and $T^{1,0}$ the induced
	CR structure on $M=\bdry\Omega$.
	Then, for any choice of a boundary defining function $r\in C^\infty(\closure\Omega)$,
	the $[\Theta]$-metric \eqref{eq:CompleteKahlerMetricLiftedToX} on the square root of $\closure\Omega$
	is an ACH metric with CR infinity $(M,T^{1,0})$.
\end{prop}

\begin{proof}
	We follow the same notation as in Example \ref{ex:CompleteKahlerMetric} and let
	$\set{N,\tilde T,\tilde Z_\alpha,\tilde Z_\conjalpha}$ be the dual of
	$\set{d\rho,\tilde\Theta,\tilde\Theta^\alpha,\tilde\Theta^\conjalpha}$.
	It is obvious from \eqref{eq:CompleteKahlerMetricLiftedToX} that
	\eqref{eq:TransverseBoundaryValueAssumption}, \eqref{eq:TangentialBoundaryValueAssumption}
	and (\ref{item:ACHConditionOnLeviBundle}) of Definition \ref{dfn:ACHMetric} are satisfied.
	Since \eqref{eq:CompleteKahlerMetricLiftedToX} also shows $R_p=\braket{\rho N}/F_p$ and
	$(L_p)_\bbC=\braket{\rho\tilde Z_1,\dots,\rho\tilde Z_n,\rho\tilde Z_{\conj{1}},\dots,\rho\tilde Z_{\conj{n}}}/F_p$,
	by Remark \ref{rem:ReformulationOfFirstConditionOfACH} we only have to check
	$d\tilde\Theta(N,\tilde Z_\alpha)=O(\rho)$ to prove that (\ref{item:ACHConditionOnDecompositionOfK1}) holds.
	It follows from \eqref{eq:ExtensionOfLeviForm} that $d\tilde\theta$ does not contain
	$dr\wedge\tilde\theta^\alpha$ term, which implies that $d\tilde\Theta(N,\tilde Z_\alpha)=0$.
\end{proof}

For any $[\Theta]$-metric on $X$ satisfying \eqref{eq:TransverseBoundaryValueAssumption},
\eqref{eq:TangentialBoundaryValueAssumption} and a choice of a contact form in $\iota^*[\Theta]$,
there is a special boundary defining function,
which is called a \emph{model boundary defining function,} as shown below.

\begin{lem}
	\label{lem:ModelBoundaryDefiningFunction}
	Let $(X,[\Theta])$ be a $[\Theta]$-manifold and
	$g$ a $[\Theta]$-metric satisfying \eqref{eq:TransverseBoundaryValueAssumption},
	\eqref{eq:TangentialBoundaryValueAssumption}.
	Then, for any $\theta\in\iota^*[\Theta]$, there exists a boundary defining function $\rho$ such that
	\begin{equation}
		\abs{\dfrac{d\rho}{\rho}}_g=\dfrac{1}{2} \qquad\text{near $\bdry X$}
		\label{eq:NormalizationCondition}
	\end{equation}
	and $\iota^*(\rho^4g)=\theta^2$. The germ of $\rho$ along $\bdry X$ is unique.
\end{lem}

\begin{proof}
	This is given in \cite{GuillarmouSaBarreto}, but for readers' convenience we include a proof.
	Let $\rho'$ be any boundary defining function and set $\rho=e^{\psi}\rho'$.
	Then $\abs{d\rho/\rho}_g=1/2$ is equivalent to
	\begin{equation}
		\dfrac{2X_{\rho'}}{\rho'}\psi+\rho\abs{\dfrac{d\psi}{\rho'}}_g^2
		=\dfrac{1}{\rho'}\left(\dfrac{1}{4}-\abs{\dfrac{d\rho'}{\rho'}}_g^2\right),
		\label{eq:PDEForModelDefiningFunction}
	\end{equation}
	where $X_{\rho'}=\sharp_g(d\rho'/\rho')$ is the dual of $d\rho'/\rho'$ with respect to $g$.
	If we express $X_{\rho'}$ in the form \eqref{eq:GeneralFormOfThetaVectorFieldAtBoundary}, then
	the assumption \eqref{eq:TransverseBoundaryValueAssumption} implies that $a=1/4$ on $\bdry X$.
	Hence \eqref{eq:PDEForModelDefiningFunction} is a noncharacteristic first-order PDE.
	After prescribing the boundary value of $\psi$ so that $\iota^*(\rho^4g)=\theta^2$ is satisfied,
	we obtain a unique solution of \eqref{eq:PDEForModelDefiningFunction} near $\bdry X$.
\end{proof}

Fix any contact form $\theta\in\iota^*[\Theta]$ on $M=\bdry X$.
Let $\rho$ be a model boundary defining function associated to $\theta$ and $X_{\rho}:=\sharp_g(d\rho/\rho)$.
We consider the smooth map induced by the flow $\Fl_t$ of the vector field $4X_\rho/\rho$, which is
transverse to $\bdry X$:
\begin{equation*}
	\Phi\colon (\text{an open neighborhood of $M\times\set{0}$ in $M\times[0,\infty)$})\to X,\qquad
	(p,t)\mapsto\Fl_t(p).
\end{equation*}
The manifold-with-boundary $M\times[0,\infty)_t$,
whose boundary $M\times\set{0}$ is identified with $M$, carries a standard $[\Theta]$-structure,
which is given by extending $[\theta]$ in such a way that $[\theta](\partial_t)=0$.
Since $\tilde\Theta(4X_\rho/\rho)=4\rho g(d\rho/\rho,\tilde\Theta/\rho^2)=O(\rho)$,
we conclude that $\Phi$ is a \emph{$[\Theta]$-diffeomorphism,} i.e., a diffeomorphism preserving
$[\Theta]$-structures, onto its image.
By the construction $t\partial_t$ is orthogonal to $\ker(dt/t)$ with respect to the induced $[\Theta]$-metric
$\Phi^*g$,
and we also see that $t=\Phi^*\rho$, which implies that $t$ is
a model boundary defining function for $\Phi^*g$ and $\theta$.

\begin{dfn}
	Let $(M,T^{1,0})$ be a nondegenerate partially integrable almost CR manifold and
	give $M\times[0,\infty)_\rho$ the standard $[\Theta]$-structure.
	Let $\theta$ a pseudohermitian structure on $(M,T^{1,0})$.
	Then a \emph{normal-form ACH metric} $g$ for $(M,T^{1,0})$ and $\theta$ is an ACH metric
	defined near the boundary of $M\times[0,\infty)_\rho$ satisfying the following conditions:
	\begin{enumerate}
		\item $\rho\partial_\rho$ is orthogonal to $\ker(d\rho/\rho)$ with respect to $g$;
		\item $\rho$ is a model boundary defining function for $g$ and $\theta$;
		\item its CR infinity is $(M,T^{1,0})$.
	\end{enumerate}
\end{dfn}

What we have proved is that, for any choice of $\theta\in\iota^*[\Theta]$,
any ACH metric is, if it is restricted to some neighborhood of the boundary,
identified with a normal-form ACH metric for $(M,T^{1,0})$ and $\theta$ via a boundary-fixing
$[\Theta]$-diffeomorphism, where $(M,T^{1,0})$ is the CR infinity of the original ACH manifold.

\begin{prop}
	\label{prop:NormalFormOfACHMetric}
	Let $(M,T^{1,0})$ be a nondegenerate partially integrable almost CR manifold and
	$X\subset M\times[0,\infty)_\rho$ an open neighborhood of $M=M\times\set{0}$ carrying the
	standard $[\Theta]$-structure.
	Let $\set{Z_\alpha}$ in general denote a local frame of $T^{1,0}$,
	$\set{\theta^\alpha}$ a family of 1-forms on $M$ satisfying
	$\theta^\beta(Z_\alpha)=\tensor{\delta}{_\alpha^\beta}$ and $\theta^\conjalpha=\conj{\theta^\alpha}$.
	We fix a pseudohermitian structure $\theta$.
	The 1-forms $\theta$, $\theta^\alpha$ and $\theta^\conjalpha$ are extended in such a way that
	they annihilate $\partial_\rho$ and are constant in the $\rho$-direction.
	Then a $[\Theta]$-metric $g$ on $X$ is a normal-form ACH metric for $(M,T^{1,0})$ and $\theta$
	if and only if it is of the form
	\begin{equation}
		\label{eq:ProductDecompositionOfACHMetric}
			g=4\left(\dfrac{d\rho}{\rho}\right)^2+\tensor{g}{_0_0}\left(\dfrac{\theta}{\rho^2}\right)^2
			+2\tensor{g}{_0_A}\dfrac{\theta}{\rho^2}\dfrac{\theta^A}{\rho}
			+\tensor{g}{_A_B}\dfrac{\theta^A}{\rho}\dfrac{\theta^B}{\rho},
	\end{equation}
	where the indices $A$, $B$ run $\set{1,\dots,n,\conj{1},\dots,\conj{n}}$, and satisfies
	\begin{equation}
		\label{eq:NormalFormConditionOnACHMetric}
		\tensor{g}{_0_0}|_M=1,\qquad
		\tensor{g}{_0_\alpha}|_M=0,\qquad
		\tensor{g}{_\alpha_\conjbeta}|_M=\tensor{h}{_\alpha_\conjbeta}\quad\text{and}\quad
		\tensor{g}{_\alpha_\beta}|_M=0,
	\end{equation}
	where $\tensor{h}{_\alpha_\conjbeta}$ is the Levi form associated with $\theta$.
\end{prop}

\begin{proof}
	The condition $\rho\partial_\rho\perp_g\ker(d\rho/\rho)$,
	together with \eqref{eq:PDEForModelDefiningFunction},
	implies that $g$ is of the form \eqref{eq:ProductDecompositionOfACHMetric}.
	In order $\rho$ to be a model boundary defining function for $\theta$,
	$\tensor{g}{_0_0}$ must be $1$ at $M$.
	By Remark \ref{rem:ReformulationOfFirstConditionOfACH},
	the condition (\ref{item:ACHConditionOnDecompositionOfK1}) in
	Definition \ref{dfn:ACHMetric} is equivalent to $\tensor{g}{_0_\alpha}|_M=0$ in this situation.
	The given almost CR structure $T^{1,0}$ is the one in (\ref{item:ACHConditionOnLeviBundle}) of Definition
	\ref{dfn:ACHMetric} if and only if
	$\tensor{g}{_\alpha_\conjbeta}|_M=\tensor{h}{_\alpha_\conjbeta}$ and $\tensor{g}{_\alpha_\beta}|_M=0$.
\end{proof}


\section{Pseudohermitian geometry}
\label{sec:PseudohermitianGeometry}

Let $(M,T^{1,0})$ be a nondegenerate partially integrable almost CR manifold.
In the presence of a fixed pseudohermitian structure $\theta$, there is a canonical direct sum decomposition
of $T_\bbC M$:
\begin{equation*}
	T_\bbC M=\bbC T\oplus T^{1,0}\oplus T^{0,1}.
\end{equation*}
Here $T$, the \emph{Reeb vector field,} is characterized by
\begin{equation*}
	\theta(T)=1,\qquad
	T\intprod d\theta=0.
\end{equation*}
If $\set{Z_\alpha}$ is a local frame of $T^{1,0}$, the \emph{admissible coframe} $\set{\theta^\alpha}$
is defined in such a way that
$\theta^\alpha(Z_\beta)=\tensor{\delta}{_\beta^\alpha}$ and $\theta^\alpha|_{\bbC T\oplus T^{0,1}}=0$.
This makes $\set{\theta,\theta^\alpha,\theta^\conjalpha}$ into the dual coframe of
$\set{T,Z_\alpha,Z_\conjalpha}$.
The index $0$ is used for components corresponding with $T$ or $\theta$.

The \emph{Tanaka--Webster connection} can be defined as the following proposition shows.
The proof goes in the same manner as in the integrable case. See, e.g., Proposition 3.1 in \cite{Tanaka}.

\begin{prop}
	On a nondegenerate partially integrable almost CR manifold $(M, T^{1,0})$
	with a fixed pseudohermitian structure $\theta$, there is a unique connection $\nabla$ on $TM$
	satisfying the following conditions:
	\begin{enumerate}
		\item $H$, $T$, $J$, $h$ are all parallel with respect to $\nabla$;
		\item The torsion tensor $\Theta(X,Y):=\nabla_XY-\nabla_YX-[X,Y]$ satisfies
			\begin{equation}
				\begin{cases}
					\Theta(X,Y)+\Theta(JX,JY)=2\,d\theta(X, Y)T, & X, Y\in \Gamma(H),\\
					\Theta(T, JX)=-J\Theta(T, X), & X\in \Gamma(H).
				\end{cases}
				\label{eq:TorsionOfTWConnection}
			\end{equation}
	\end{enumerate}
	\label{prop:GeneralizedTanakaWebsterConnection}
\end{prop}

The components $\tensor{\Theta}{_\alpha_\beta^0}$, $\tensor{\Theta}{_\alpha_\beta^\gamma}$,
$\tensor{\Theta}{_\alpha_\beta^{\conj{\gamma}}}$ of the torsion
are not visible in \eqref{eq:TorsionOfTWConnection}.
Following the argument in the integrable case the first two are shown to be zero.
One immediately sees from the definition that the last one is related to the Nijenhuis tensor by
\begin{equation*}
	\tensor{\Theta}{_\alpha_\beta^{\conj{\gamma}}}=-\tensor{N}{_\alpha_\beta^{\conj{\gamma}}}
	\qquad\text{(and }
	\tensor{\Theta}{_{\conj{\alpha}}_{\conj{\beta}}^\gamma}=-\tensor{N}{_{\conj{\alpha}}_{\conj{\beta}}^\gamma}
	\text{)}.
\end{equation*}
The other nonzero components of the torsion are
\begin{equation*}
	\tensor{\Theta}{_\alpha_{\conj{\beta}}^0}=i\tensor{h}{_\alpha_{\conj{\beta}}},
	\qquad
	\tensor{\Theta}{_0_\alpha^{\conj{\beta}}}=-\tensor{\Theta}{_\alpha_0^{\conj{\beta}}}
	=:\tensor{A}{_\alpha^{\conj{\beta}}}
\end{equation*}
and their complex conjugates.
We call $\tensor{A}{_\alpha^{\conj{\beta}}}$ the \emph{Tanaka--Webster torsion tensor}.

\begin{rem}
	There is another generalization of the Tanaka--Webster connection to the partially integrable case
	given by Tanno \cite{Tanno}, which is also used in \cite{BarlettaDragomir}, \cite{BlairDragomir}
	and \cite{Seshadri}.
	Our generalization is different from it in that ours preserves $J$, which facilitates the whole argument
	below, and that $\tensor{\Theta}{_\alpha_\beta^{\conj{\gamma}}}$ is generally nonzero instead.
	It seems that our connection is first considered by Mizner~\cite{Mizner}.
\end{rem}

The first structure equation is as follows:
\begin{gather}
	\label{eq:FirstStructureEquation1}
	d\theta=i\tensor{h}{_\alpha_{\conj{\beta}}}\theta^{\alpha}\wedge\theta^{\conj{\beta}},\\
	\label{eq:FirstStructureEquation2}
	d\theta^{\gamma}
	=\theta^{\alpha}\wedge\tensor{\omega}{_\alpha^\gamma}
	-\tensor{A}{_{\conj{\alpha}}^\gamma}\theta^{\conj{\alpha}}\wedge\theta
	-\tfrac{1}{2}\tensor{N}{_{\conj{\alpha}}_{\conj{\beta}}^\gamma}\theta^{\conj{\alpha}}\wedge\theta^{\conj{\beta}}.
\end{gather}

Let $\set{\tensor{\omega}{_\alpha^\beta}}$ be the connection forms of the Tanaka--Webster connection.
Without any modification, the proof of Lemma 2.1 in \cite{LeePseudoEinstein} applies
to the partially integrable case and we obtain the following lemma.

\begin{lem}
	\label{lem:SpecialFrame}
	In a neighborhood of any point $p\in M$, there exists a frame $\set{Z_{\alpha}}$ of $T^{1,0}$
	for which $\tensor{\omega}{_\alpha^\beta}(p)=0$ holds.
\end{lem}

With such a local frame, it is easy to relate exterior derivatives with covariant derivatives.
For example, one immediately sees that the exterior derivative of a $(1,0)$-form
$\sigma=\tensor{\sigma}{_\alpha}\tensor{\theta}{^\alpha}$ is given by
\begin{equation*}
	d\sigma
	=\tensor{\sigma}{_\alpha_,_\beta}\tensor{\theta}{^\beta}\wedge\tensor{\theta}{^\alpha}
	+\tensor{\sigma}{_\alpha_,_{\conj{\beta}}}\tensor{\theta}{^{\conj{\beta}}}\wedge\tensor{\theta}{^\alpha}
	+\tensor{\sigma}{_\alpha_,_0}\theta\wedge\tensor{\theta}{^\alpha}
	-\tensor{A}{_{\conj{\beta}}^\alpha}\tensor{\sigma}{_\alpha}\tensor{\theta}{^{\conj{\beta}}}\wedge\theta
	-\tfrac{1}{2}\tensor{N}{_{\conj{\beta}}_{\conj{\gamma}}^\alpha}
	\tensor{\sigma}{_\alpha}\tensor{\theta}{^{\conj{\beta}}}\wedge\tensor{\theta}{^{\conj{\gamma}}}.
\end{equation*}
Here covariant derivatives of tensors are denoted by indices after commas.
This notation will be used in the sequel.
In the case of covariant derivatives of a scalar-valued function we omit the comma;
e.g., $\tensor{\nabla}{_\alpha}u=\tensor{u}{_\alpha}$
and $\tensor{\nabla}{_{\conj{\beta}}}\tensor{\nabla}{_\alpha}u=\tensor{u}{_\alpha_{\conj{\beta}}}$.

\begin{prop}
	We have
	\begin{gather}
		\label{eq:SymmetryOfPseudohermitianTorsionTensor}
		\tensor{A}{_\alpha_\beta}=\tensor{A}{_\beta_\alpha},\\
		\label{eq:SymmetryOfNijenhuisTensor}
		\tensor{N}{_\alpha_\beta_\gamma}+\tensor{N}{_\beta_\alpha_\gamma}=0,\qquad
		\tensor{N}{_\alpha_\beta_\gamma}+\tensor{N}{_\beta_\gamma_\alpha}+\tensor{N}{_\gamma_\alpha_\beta}=0.
	\end{gather}
\end{prop}

\begin{proof}
	By differentiating \eqref{eq:FirstStructureEquation1} and considering types we obtain
	\eqref{eq:SymmetryOfPseudohermitianTorsionTensor} and $\tensor{N}{_[_\alpha_\beta_\gamma_]}=0$
	(where the square brackets denotes skew-symmetrization).
	The first identity of \eqref{eq:SymmetryOfNijenhuisTensor} is obvious from the definition of the Nijenhuis
	tensor, and it thereby proves the second one.
\end{proof}

\begin{lem}
	The second covariant derivatives of a scalar-valued function $u$ satisfy the following:
	\begin{equation}
		\label{eq:CommutatorOfCovariantDerivativeOnScalar}
		\tensor{u}{_\alpha_{\conj{\beta}}}-\tensor{u}{_{\conj{\beta}}_\alpha}
		=i\tensor{h}{_\alpha_{\conj{\beta}}}\tensor{u}{_0},\qquad
		\tensor{u}{_\alpha_\beta}-\tensor{u}{_\beta_\alpha}
		=-\tensor{N}{_\alpha_\beta^{\conj{\gamma}}}\tensor{u}{_{\conj{\gamma}}},\qquad
		\tensor{u}{_0_\alpha}-\tensor{u}{_\alpha_0}=\tensor{A}{_\alpha^{\conj{\beta}}}\tensor{u}{_{\conj{\beta}}}.
	\end{equation}
\end{lem}

\begin{proof}
	The same argument as the one in \cite{LeePseudoEinstein} applies to our case.
\end{proof}

Next we shall study the curvature
$R^{\text{TW}}(X, Y)=\nabla_{X}\nabla_{Y}-\nabla_{Y}\nabla_{X}-\nabla_{[X, Y]}$.
If we set
$\tensor{\Pi}{_\alpha^\beta}=d\tensor{\omega}{_\alpha^\beta}-\tensor{\omega}{_\alpha^\gamma}\wedge\tensor{\omega}{_\gamma^\beta}$,
it holds that $R^{\text{TW}}(X, Y)Z_{\alpha}=\tensor{\Pi}{_\alpha^\beta}(X,Y)Z_{\beta}$.
We put
\begin{equation}
	\label{eq:CurvatureForm}
	\tensor{\Pi}{_\alpha_\conjbeta}
	=\tensor{R}{_\alpha_\conjbeta_\sigma_\conjtau}\theta^\sigma\wedge\theta^\conjtau
	+\tensor{W}{_\alpha_\conjbeta_\gamma}\theta^\gamma\wedge\theta
	+\tensor{W}{_\alpha_\conjbeta_\conjgamma}\theta^\conjgamma\wedge\theta
	+\tensor{V}{_\alpha_\conjbeta_\sigma_\tau}\theta^\sigma\wedge\theta^\tau
	+\tensor{V}{_\alpha_\conjbeta_\conjsigma_\conjtau}\theta^\conjsigma\wedge\theta^\conjtau,
\end{equation}
where
$\tensor{V}{_\alpha^\beta_(_\sigma_\tau_)}=\tensor{V}{_\alpha^\beta_(_{\conj{\sigma}}_{\conj{\tau}}_)}=0$.
Since $\nabla h=0$ we have
$\tensor{\Pi}{_\alpha_{\conj{\beta}}}+\tensor{\Pi}{_{\conj{\beta}}_\alpha}=0$,
and hence
\begin{equation}
	\label{eq:SymmetryOfPseudohermitianCurvature}
	\tensor{R}{_\alpha_\conjbeta_\sigma_\conjtau}=\tensor{R}{_\conjbeta_\alpha_\conjtau_\sigma},\qquad
	\tensor{W}{_\alpha_\conjbeta_\conjgamma}=-\tensor{W}{_\conjbeta_\alpha_\conjgamma},\qquad
	\tensor{V}{_\alpha_\conjbeta_\sigma_\tau}=-\tensor{V}{_\conjbeta_\alpha_\sigma_\tau}.
\end{equation}
We substitute \eqref{eq:CurvatureForm} into the exterior derivative of \eqref{eq:FirstStructureEquation2}
and compare the coefficients to obtain
\begin{subequations}
	\begin{gather}
		\label{eq:AsymmetryOfPseudohermitianCurvature}
		\tensor{R}{_\alpha_\conjbeta_\sigma_\conjtau}-\tensor{R}{_\sigma_\conjbeta_\alpha_\conjtau}
		=-\tensor{N}{_\alpha_\sigma^\conjgamma}\tensor{N}{_\conjtau_\conjgamma_\conjbeta},\\
		\tensor{W}{_\alpha_\conjbeta_\gamma}
		=\tensor{A}{_\alpha_\gamma_,_\conjbeta}
		-\tensor{N}{_\gamma_\sigma_\alpha}\tensor{A}{_\conjbeta^\sigma},\qquad
		\tensor{V}{_\alpha_\conjbeta_\sigma_\tau}
		=\tfrac{i}{2}(\tensor{h}{_\sigma_\conjbeta}\tensor{A}{_\alpha_\tau}
		-\tensor{h}{_\tau_\conjbeta}\tensor{A}{_\alpha_\sigma})
		+\tfrac{1}{2}\tensor{N}{_\sigma_\tau_\alpha_,_\conjbeta}.
\end{gather}
\end{subequations}

The component $\tensor{R}{_\alpha_{\conj{\beta}}_\rho_{\conj{\sigma}}}$ is called
the \emph{Tanaka--Webster curvature tensor}.
We put $\tensor{R}{_\alpha_{\conj{\beta}}}:=\tensor{R}{_\gamma^\gamma_\alpha_{\conj{\beta}}}$ and
$R:=\tensor{R}{_\alpha^\alpha}$.
It is seen from the first identity of \eqref{eq:SymmetryOfPseudohermitianCurvature}
that $\tensor{R}{_\alpha_{\conj{\beta}}}=\tensor{R}{_{\conj{\beta}}_\alpha}$,
and from \eqref{eq:AsymmetryOfPseudohermitianCurvature} we have
\begin{equation}
	\tensor{R}{_\alpha^\gamma_\gamma_{\conj{\beta}}}
	=\tensor{R}{_\alpha_{\conj{\beta}}}
	-\tensor{N}{_\alpha_\sigma_\tau}\tensor{N}{_{\conj{\beta}}^\tau^\sigma}.
	\label{eq:AnomalyOfTwoTWRicciTensors}
\end{equation}

As we have discussed above,
a choice of a pseudohermitian structure $\theta$ defines the Tanaka--Webster connection
and supplies various pseudohermitian invariants.
If a certain pseudohermitian invariant is also conserved by any change of pseudohermitian structure,
it is called a \emph{CR invariant}\footnote{Rigorously speaking we should say
``partially-integrable-almost-CR invariant,'' but we prefer this shorter expression.}.
To investigate such invariants, we need the transformation law of the connection and relevant quantities.

\begin{prop}
	\label{prop:TransformationOfPseudohermitianInvariants}
	Let $\theta$ and $\Hat{\theta}=e^{2u}\theta$, $u\in C^{\infty}(M)$, be two pseudohermitian structures
	on a nondegenerate partially integrable almost CR manifold $(M,T^{1,0})$.
	Then, the Tanaka--Webster connection forms, the torsions and the Ricci tensors are related as follows:
	\begin{gather}
		\label{eq:TransformationOfConnectionForms}
		\tensor{\Hat{\omega}}{_\alpha^\beta}
		=\tensor{\omega}{_\alpha^\beta}
		+2(\tensor{u}{_\alpha}\tensor{\theta}{^\beta}-\tensor{u}{^\beta}\tensor{\theta}{_\alpha})
		+2\tensor{\delta}{_\alpha^\beta}\tensor{u}{_\gamma}\tensor{\theta}{^\gamma}
		+2i(\tensor{u}{^\beta_\alpha}+2\tensor{u}{_\alpha}\tensor{u}{^\beta}
		+2\tensor{\delta}{_\alpha^\beta}\tensor{u}{_\gamma}\tensor{u}{^\gamma})\theta,\\
		\label{eq:TransformationOfPseudohermitianTorsion}
		\tensor{\Hat{A}}{_\alpha_\beta}
		=\tensor{A}{_\alpha_\beta}+i(\tensor{u}{_\alpha_\beta}+\tensor{u}{_\beta_\alpha})
		-4i\tensor{u}{_\alpha}\tensor{u}{_\beta}
		+i(\tensor{N}{_\gamma_\alpha_\beta}+\tensor{N}{_\gamma_\beta_\alpha})\tensor{u}{^\gamma},\\
		\label{eq:TransformationOfPseudohermitianRicciTensor}
		\tensor{\Hat{R}}{_\alpha_\conjbeta}
		=\tensor{R}{_\alpha_\conjbeta}
		-(n+2)(\tensor{u}{_\alpha_\conjbeta}+\tensor{u}{_\conjbeta_\alpha})
		-\left(\tensor{u}{_\gamma^\gamma}+\tensor{u}{^\gamma_\gamma}
		+4(n+1)\tensor{u}{_\gamma}\tensor{u}{^\gamma}\right)\tensor{h}{_\alpha_\conjbeta}.
	\end{gather}
\end{prop}

\begin{proof}
	The new Reeb vector field is
	$\Hat{T}=e^{-2u}(T-2i\tensor{u}{^\alpha}\tensor{Z}{_\alpha}+2i\tensor{u}{^{\conj{\alpha}}}\tensor{Z}{_{\conj{\alpha}}})$
	and the admissible coframe dual to $\set{Z_{\alpha}}$ is
	$\set{\Hat{\theta}{^\alpha}=\tensor{\theta}{^\alpha}+2i\tensor{u}{^\alpha}\theta}$.
	To establish \eqref{eq:TransformationOfConnectionForms} and \eqref{eq:TransformationOfPseudohermitianTorsion},
	it is enough to check that
	\begin{equation*}
		d\tensor{\Hat{h}}{_\alpha_\conjbeta}
		=\tensor{\Hat{h}}{_\gamma_\conjbeta}\tensor{\Hat{\omega}}{_\alpha^\gamma}
		+\tensor{\Hat{h}}{_\alpha_\conjgamma}\tensor{\Hat{\omega}}{_\conjbeta^\conjgamma}
	\end{equation*}
	and
	\begin{equation*}
		d\Hat{\theta}^\gamma
		=\Hat{\theta}^\alpha\wedge\tensor{\Hat{\omega}}{_\alpha^\gamma}
		-\tensor{\Hat{h}}{^\gamma^\conjbeta}\tensor{\Hat{A}}{_\conjalpha_\conjbeta}\Hat{\theta}^\conjalpha\wedge\Hat{\theta}
		-\tfrac{1}{2}\tensor{N}{_\conjalpha_\conjbeta^\gamma}\Hat{\theta}^\conjalpha\wedge\Hat{\theta}^\conjbeta.
	\end{equation*}
	They are shown straightforward using \eqref{eq:CommutatorOfCovariantDerivativeOnScalar}.
	
	We compute $\tensor{\Hat{\Pi}}{_\gamma^\gamma}=d\tensor{\Hat{\omega}}{_\gamma^\gamma}$
	modulo $\tensor{\Hat{\theta}}{^\alpha}\wedge\tensor{\Hat{\theta}}{^\beta}$,
	$\tensor{\Hat{\theta}}{^{\conj{\alpha}}}\wedge\tensor{\Hat{\theta}}{^{\conj{\beta}}}$, $\Hat{\theta}$,
	or equivalently, modulo $\tensor{\theta}{^\alpha}\wedge\tensor{\theta}{^\beta}$,
	$\tensor{\theta}{^{\conj{\alpha}}}\wedge\tensor{\theta}{^{\conj{\beta}}}$, $\theta$.
	By the first identity of \eqref{eq:CommutatorOfCovariantDerivativeOnScalar} we obtain that,
	modulo $\tensor{\Hat{\theta}}{^\alpha}\wedge\tensor{\Hat{\theta}}{^\beta}$,
	$\tensor{\Hat{\theta}}{^{\conj{\alpha}}}\wedge\tensor{\Hat{\theta}}{^{\conj{\beta}}}$, $\Hat{\theta}$,
	\begin{equation*}
		\begin{split}
			\tensor{\Hat{\Pi}}{_\gamma^\gamma}
			&\equiv\tensor{\Pi}{_\gamma^\gamma}
			-\left[(n+2)(\tensor{u}{_\alpha_{\conj{\beta}}}+\tensor{u}{_{\conj{\beta}}_\alpha})
			+\left(\tensor{u}{_\gamma^\gamma}+\tensor{u}{_{\conj{\gamma}}^{\conj{\gamma}}}
			+4(n+1)\tensor{u}{_\gamma}\tensor{u}{^\gamma}\right)\tensor{h}{_\alpha_{\conj{\beta}}}\right]
			\tensor{\theta}{^\alpha}\wedge\tensor{\theta}{^{\conj{\beta}}}\\
			&\equiv\left[\tensor{R}{_\alpha_{\conj{\beta}}}
			-(n+2)(\tensor{u}{_\alpha_{\conj{\beta}}}+\tensor{u}{_{\conj{\beta}}_\alpha})
			-\left(\tensor{u}{_\gamma^\gamma}+\tensor{u}{_{\conj{\gamma}}^{\conj{\gamma}}}
			+4(n+1)\tensor{u}{_\gamma}\tensor{u}{^\gamma}\right)\tensor{h}{_\alpha_{\conj{\beta}}}\right]
			\tensor{\Hat{\theta}}{^\alpha}\wedge\tensor{\Hat{\theta}}{^{\conj{\beta}}}.
		\end{split}
	\end{equation*}
	This proves \eqref{eq:TransformationOfPseudohermitianRicciTensor}.
\end{proof}

Finally we sketch the concept of density bundles following \cite{GoverGraham}.
Let us assume that we have fixed a complex line bundle $E(1,0)$ over $M$
together with a duality between $E(1,0)^{\otimes(n+2)}$ and the CR canonical bundle $K$.
Such a choice may not exist globally, but locally it does;
when we use density bundles we restrict our scope to the local theory.
Then $E(w,0)$ is the $w^\mathrm{th}$ tensor power of $E(1,0)$, and we set
\begin{equation*}
	E(w,w')=E(w,0)\otimes E(0,w'),\qquad\text{$w$, $w'\in\bbZ$},
\end{equation*}
where $E(0,w'):=\conj{E(w',0)}$.
We call $E(w,w')$ the \emph{density bundle of biweight $(w,w')$}.
Since there is a specified isomorphism $E(-n-2,0)\cong K$, we can define a connection $\nabla$ on $E(w,w')$
so that it is compatible with the Tanaka--Webster connection on $K$.
The space of local sections of $E(w,w')$ is denoted by $\mathcal{E}(w,w')$.

Farris \cite{Farris} observed that,
if $\zeta$ is a locally defined nonvanishing section of $K$,
there is a unique pseudohermitian structure $\theta$ satisfying \eqref{eq:FarrisSectionOfCanonicalBundle}.
If we replace $\zeta$ with $\lambda\zeta$, $\lambda\in C^\infty(M,\bbC^{\times})$,
then $\theta$ is replaced by $\lvert\lambda\rvert^{2/(n+2)}\theta$.
We set
\begin{equation*}
	\abs{\zeta}^{2/(n+2)}=\zeta^{1/(n+2)}\otimes\conj{\zeta}^{1/(n+2)}\in\mathcal{E}(-1,-1),
\end{equation*}
which is independent of the choice of the $(n+2)^\mathrm{nd}$ root of $\zeta$ and
is in one-to-one correspondence with $\theta$,
and define $\lvert\zeta\rvert^{-2/(n+2)}\in\mathcal{E}(1,1)$ as its dual.
Then we obtain a CR-invariant section
$\bm{\theta}:=\theta\otimes\lvert\zeta\rvert^{-2/(n+2)}$ of $T^{*}M\otimes E(1,1)$.

The Levi form $h$ is a section of the bundle $(T^{1,0})^{*}\otimes(T^{0,1})^{*}$,
which is also denoted by $\tensor{E}{_\alpha_\conjbeta}$ using abstract indices $\alpha$ and $\conjbeta$.
Since $\tensor{h}{_\alpha_\conjbeta}$ and $\theta$ have the same scaling factor,
$\tensor{\bm{h}}{_\alpha_\conjbeta}:=\tensor{h}{_\alpha_\conjbeta}\otimes\lvert\zeta\rvert^{-2/(n+2)}\in\tensor{\caE}{_\alpha_\conjbeta}(1,1)$
is a CR-invariant section of
$\tensor{E}{_\alpha_\conjbeta}(1,1):=\tensor{E}{_\alpha_\conjbeta}\otimes E(1,1)$.
Its dual is $\tensor{\bm{h}}{^\alpha^\conjbeta}\in\tensor{\caE}{^\alpha^\conjbeta}(-1,-1)$.
Indices of density-weighted tensors are lowered and raised by $\tensor{\bm{h}}{_\alpha_\conjbeta}$
and $\tensor{\bm{h}}{^\alpha^\conjbeta}$.

One can show that $\nabla{\bm{\theta}}$ and $\nabla{\bm{h}}$ are both zero.
To see this it is enough to show that $\nabla\lvert\zeta\rvert^{2}=0$,
which follows from $\nabla h=0$.
For details see the proof of Proposition 2.1 in \cite{GoverGraham}.

The density-weighted versions of the Nijenhuis tensor, the Tanaka--Webster torsion tensor
and the curvature tensor are defined by
\begin{gather*}
	\tensor{\bm{N}}{_\alpha_\beta^{\conj{\gamma}}}:=\tensor{N}{_\alpha_\beta^{\conj{\gamma}}}
	\in\tensor{\caE}{_\alpha_\beta^\conjgamma},\qquad
	\tensor{\bm{A}}{_\alpha_\beta}:=\tensor{A}{_\alpha_\beta}
	\in\tensor{\caE}{_\alpha_\beta},\\
	\tensor{\bm{R}}{_\alpha_\conjbeta_\sigma_\conjtau}:=\tensor{R}{_\alpha_\conjbeta_\sigma_\conjtau}
	\otimes\abs{\zeta}^{-2/(n+2)}\in\tensor{\caE}{_\alpha_\conjbeta_\sigma_\conjtau}(1,1).
\end{gather*}

When dealing with density-weighted tensors,
we let $\tensor{\nabla}{_\alpha}$, $\tensor{\nabla}{_{\conj{\alpha}}}$ and $\tensor{\nabla}{_{\bm{0}}}$ denote
the components of $\nabla$ relative to $\theta^{\alpha}$, $\theta^{\conj{\alpha}}$ and $\bm{\theta}$.
Since the transformation law \eqref{eq:TransformationOfConnectionForms} of the Tanaka--Webster connection forms
does not contain the Nijenhuis tensor,
equation (2.7) and Proposition 2.3 in \cite{GoverGraham} also hold in the partially integrable case.
Using them we can derive the transformation law of any covariant derivative of any density-weighted tensor.


\section{Some low-order terms of Einstein metrics}
\label{sec:RicciTensorAndSomeLowOrderTerms}

Let $(M,T^{1,0})$ be a nondegenerate partially integrable almost CR manifold with a fixed pseudohermitian structure $\theta$
and $X\subset M\times[0,\infty)_\rho$ an open neighborhood of $M=M\times\set{0}$. We take a local frame
\begin{equation}
	\set{\euler,\rho^2T,\rho Z_\alpha,\rho Z_\conjalpha}
	\label{eq:LocalFrameOfThetaTangentBundle}
\end{equation}
of $\ThetaTangentX$, where $T$ is the Reeb vector field associated with $\theta$ and
$\set{Z_\alpha}$ is a local frame of $T^{1,0}$, both extended constantly in the $\rho$-direction.
The corresponding indices are $\infty$, $0$, $1$, $\dots$, $n$, $\conj{1}$, $\dots$, $\conj{n}$.
The local frame \eqref{eq:LocalFrameOfThetaTangentBundle} is denoted by $\set{W_I}$ if needed.

\begin{IndexRule}
	The following rule is observed in the sequel, except in the proof of Proposition
	\ref{prop:ObstructionTensorForIntegrableStructure}:
	\begin{itemize}
		\item $\alpha$, $\beta$, $\gamma$, $\sigma$, $\tau$ run $\set{1,\dots,n}$
			and $\conjalpha$, $\conjbeta$, $\conjgamma$, $\conjsigma$, $\conjtau$
			run $\set{\conj{1},\dots,\conj{n}}$;
		\item $i$, $j$, $k$ run $\set{0,1,\dots,n,\conj{1},\dots,\conj{n}}$;
		\item $I$, $J$, $K$, $L$ run $\set{\infty,0,1,\dots,n,\conj{1},\dots,\conj{n}}$.
	\end{itemize}
	Lowercase Greek indices and their complex conjugates are raised and lowered by the Levi form
	unless otherwise stated.
\end{IndexRule}

We consider a normal-form ACH metric $g$ on $X$.
By Proposition \ref{prop:NormalFormOfACHMetric}, the ACH condition for $g$ is equivalent to
\begin{equation}
	\label{eq:NormalizationOfACHMetricInComponents}
	\begin{split}
		\tensor{g}{_\infty_\infty}&=4,\qquad
		\tensor{g}{_\infty_0}=0,\qquad
		\tensor{g}{_\infty_\alpha}=0,\\
		\tensor{g}{_0_0}&=1+O(\rho),\qquad
		\tensor{g}{_0_\alpha}=O(\rho),\qquad
		\tensor{g}{_\alpha_\conjbeta}=\tensor{h}{_\alpha_\conjbeta}+O(\rho),\qquad
		\tensor{g}{_\alpha_\beta}=O(\rho),
	\end{split}
\end{equation}
where $\tensor{h}{_\alpha_\conjbeta}$ is the Levi form.
Note that $\set{W_I}$ is a distinguished local frame for $g$.
We shall compute the Ricci tensor of $g$ and the Einstein tensor $\Ein:=\Ric+\tfrac{1}{2}(n+2)g$.
Our goal in this section is the following proposition.
By abuse of notation, in what follows we use the same symbol for a tensor on $M$
and its constant extension in the $\rho$-direction.

\begin{prop}
	\label{prop:ACHEToThirdOrder}
	The Einstein tensor $\tensor{\Ein}{_I_J}$ of a normal-form ACH metric $g$ is $O(\rho^3)$ if and only if
	\begin{equation}
		\label{eq:LowestTermsConditionOfACHMetric}
	\begin{split}
		\tensor{g}{_0_0}&=1+O(\rho^3),\qquad
		\tensor{g}{_0_\alpha}=O(\rho^3),\\
		\tensor{g}{_\alpha_\conjbeta}&=\tensor{h}{_\alpha_\conjbeta}+\rho^2\tensor{\Phi}{_\alpha_\conjbeta}
		+O(\rho^3),\qquad
		\tensor{g}{_\alpha_\beta}=\rho^2\tensor{\Phi}{_\alpha_\beta}+O(\rho^3),
	\end{split}
	\end{equation}
	where
	\begin{equation}
		\label{eq:PhiTensorInLowestTerms}
	\begin{split}
		\tensor{\Phi}{_\alpha_\conjbeta}
		&=-\dfrac{2}{n+2}\left(\tensor{R}{_\alpha_\conjbeta}
		-2\tensor{N}{_\alpha_\sigma_\tau}\tensor{N}{_\conjbeta^\tau^\sigma}
		-\dfrac{1}{2(n+1)}(R
		-2\tensor{N}{_\gamma_\sigma_\tau}\tensor{N}{^\gamma^\tau^\sigma})\tensor{h}{_\alpha_\conjbeta}\right),\\
		\tensor{\Phi}{_\alpha_\beta}
		&=-2i\tensor{A}{_\alpha_\beta}
		-\dfrac{2}{n}(\tensor{N}{_\gamma_\alpha_\beta_,^\gamma}+\tensor{N}{_\gamma_\beta_\alpha_,^\gamma}).
	\end{split}
	\end{equation}
\end{prop}

The functions $\tensor{\varphi}{_i_j}$ are defined by
\begin{equation}
	\tensor{g}{_0_0}=1+\tensor{\varphi}{_0_0},\qquad
	\tensor{g}{_0_\alpha}=\tensor{\varphi}{_0_\alpha},\qquad
	\tensor{g}{_\alpha_\conjbeta}=\tensor{h}{_\alpha_\conjbeta}+\tensor{\varphi}{_\alpha_\conjbeta},\qquad
	\tensor{g}{_\alpha_\beta}=\tensor{\varphi}{_\alpha_\beta}.
	\label{eq:NonconstantTermsOfACHMetric}
\end{equation}
The totality of $(\tensor{\varphi}{_i_j})$ is seen as a symmetric 2-tensor on $M$ with
coefficients in $C^\infty(X)$ using the frame $\set{T,Z_\alpha,Z_\conjalpha}$.
Hence the action of the Tanaka--Webster connection operator $\nabla$ on $(\tensor{\varphi}{_i_j})$ is
naturally defined.

We define a connection $\overline{\nabla}$ on $TX$ by setting
$\overline{\nabla}_ZW=\nabla_ZW$ for vector fields $Z$, $W$ on $M$ and
\begin{equation*}
	\overline{\nabla}\partial_\rho=0,\qquad
	\overline{\nabla}_{\partial_\rho}T=\overline{\nabla}_{\partial_\rho}Z_\alpha=0.
\end{equation*}
The connection forms of $\overline{\nabla}$ with respect to the frame $\set{Z_I}$ are given by
\begin{equation}
	\tensor{\overline{\omega}}{_\infty^\infty}=\dfrac{d\rho}{\rho},\qquad
	\tensor{\overline{\omega}}{_0^0}=2\dfrac{d\rho}{\rho},\qquad
	\tensor{\overline{\omega}}{_\alpha^\beta}=\tensor{\omega}{_\alpha^\beta}
	+\tensor{\delta}{_\alpha^\beta}\dfrac{d\rho}{\rho},
	\label{eq:ConnectionFormsOfExtendedTWConnection}
\end{equation}
where $\tensor{\omega}{_\alpha^\beta}$ are the connection forms of $\nabla$ with respect to $\set{Z_\alpha}$.
The torsion $\overline\Theta$ of $\overline\nabla$ is
\begin{equation}
	\label{eq:TorsionOfExtendedTWConnection}
	\begin{split}
		\tensor{\overline\Theta}{_I_J^\infty}&=\tensor{\overline\Theta}{_\infty_\infty^0}
		=\tensor{\overline\Theta}{_0_I^0}=\tensor{\overline\Theta}{_\infty_I^\gamma}
		=\tensor{\overline\Theta}{_0_0^\gamma}
		=\tensor{\overline\Theta}{_0_\alpha^\beta}
		=\tensor{\overline\Theta}{_\alpha_\conjbeta^\gamma}
		=\tensor{\overline\Theta}{_\alpha_\beta^\gamma}=0,\\
		\tensor{\overline\Theta}{_\alpha_\conjbeta^0}&=i\tensor{h}{_\alpha_\conjbeta},\qquad
		\tensor{\overline\Theta}{_0_\alpha^\conjbeta}=\rho^2\tensor{A}{_\alpha^\conjbeta},\qquad
		\tensor{\overline\Theta}{_\alpha_\beta^\conjgamma}=-\rho\tensor{N}{_\alpha_\beta^\conjgamma};
	\end{split}
\end{equation}
the Ricci tensor of $\overline\nabla$, defined by
$\tensor{\overline{R}}{_I_J}:=\tensor{\overline{R}}{_I^K_K_J}$, is given by
\begin{equation}
	\begin{split}
		\tensor{\overline{R}}{_\infty_I}&=\tensor{\overline{R}}{_I_\infty}
		=\tensor{\overline{R}}{_0_I}=\tensor{\overline{R}}{_I_0}=0,\\
		\tensor{\overline{R}}{_\alpha_\conjbeta}
		&=\rho^2
		(\tensor{R}{_\alpha_\conjbeta}-\tensor{N}{_\alpha_\sigma_\tau}\tensor{N}{_\conjbeta^\tau^\sigma}),
		\qquad
		\tensor{\overline{R}}{_\alpha_\beta}
		=\rho^2\left(i(n-1)\tensor{A}{_\alpha_\beta}+\tensor{N}{_\gamma_\beta_\alpha_,^\gamma}\right).
	\end{split}
	\label{eq:RicciOfExtendedTWConnection}
\end{equation}

We sometimes reinterpret a tensor on $X$ as a set of tensors on $M$ with coefficient in $C^\infty(X)$.
For example, a symmetric 2-tensor $\tensor{S}{_I_J}$ is also regarded as the composed object of
a scalar-valued function $\tensor{S}{_\infty_\infty}$, a 1-tensor $\tensor{S}{_\infty_i}$
and a 2-tensor $\tensor{S}{_i_j}$, with coefficients in $C^\infty(X)$.
Thus $\nabla$ can be applied to
$\tensor{S}{_I_J}=(\tensor{S}{_\infty_\infty},\tensor{S}{_\infty_i},\tensor{S}{_i_j})$.
Let $\#(I_1,\dots,I_N):=N+(\text{the number of $0$'s in $I_1$, $\dots$, $I_N$})$.
Then, from \eqref{eq:ConnectionFormsOfExtendedTWConnection} we have the following formulae:
\begin{equation}
	\label{eq:CalculationOfTrivialExtensionOfTWConnection}
	\tensor{\overline{\nabla}}{_\infty}\tensor{S}{_I_J}=(\euler-\#(I,J))\tensor{S}{_I_J},\qquad
	\tensor{\overline{\nabla}}{_0}\tensor{S}{_I_J}=\rho^2\tensor{\nabla}{_0}\tensor{S}{_I_J},\qquad
	\tensor{\overline{\nabla}}{_\alpha}\tensor{S}{_I_J}=\rho\tensor{\nabla}{_\alpha}\tensor{S}{_I_J};
\end{equation}
on the left-hand sides of the equalities
$\set{\euler,\rho^2T,\rho Z_\alpha,\rho Z_\conjalpha}$ is used for covariant differentiation,
while on the right-hand sides $\set{T,Z_\alpha,Z_\conjalpha}$ is used.

We set $\nabla^g_{W_J}W_I=\overline{\nabla}_{W_J}W_I+\tensor{D}{_I^K_J}W_K$,
where $\nabla^g$ is the Levi-Civita connection of $g$. Then the Ricci tensor of $g$ is given by
\begin{equation}
	\tensor{\Ric}{_I_J}=\tensor{\overline{R}}{_I_J}
	+\tensor{\overline\nabla}{_K}\tensor{D}{_I^K_J}-\tensor{\overline\nabla}{_J}\tensor{D}{_I^K_K}
	-\tensor{D}{_I^L_K}\tensor{D}{_J^K_L}+\tensor{D}{_I^L_J}\tensor{D}{_L^K_K}.
	\label{eq:RicciTensorOfACHMetric}
\end{equation}
Thus the calculation of the Ricci tensor essentially reduces to that of $\tensor{D}{_I^K_J}$.
We can compute $\tensor{D}{_I_K_J}:=\tensor{g}{_K_L}\tensor{D}{_I^L_J}$ by the formula
\begin{equation*}
	\tensor{D}{_I_K_J}
	=\tfrac{1}{2}(\tensor{\overline\nabla}{_I}\tensor{g}{_J_K}+\tensor{\overline\nabla}{_J}\tensor{g}{_I_K}
	-\tensor{\overline\nabla}{_K}\tensor{g}{_I_J}
	+\tensor{\overline\Theta}{_I_K_J}+\tensor{\overline\Theta}{_J_K_I}+\tensor{\overline\Theta}{_I_J_K}),
\end{equation*}
where $\tensor{\overline\Theta}{_I_J_K}:=\tensor{g}{_K_L}\tensor{\overline\Theta}{_I_J^L}$.
The result is given in Table \ref{tbl:ConnectionFormsOfACHMetricWithLowerIndices}.

\begin{table}[htbp]
	\small
\begin{tabular}{ll}
	\toprule
	\multicolumn{1}{c}{Type} & \multicolumn{1}{c}{Value} \\ \midrule
	$\tensor{D}{_\infty_\infty_\infty}$ & $-4$ \\ \midrule
	$\tensor{D}{_\infty_0_\infty}$ & $0$ \\ \midrule
	$\tensor{D}{_\infty_\alpha_\infty}$ & $0$ \\ \midrule
	$\tensor{D}{_\infty_\infty_0}$ & $0$ \\ \midrule
	$\tensor{D}{_\infty_0_0}$ & $-2+\tfrac{1}{2}(\euler-4)\tensor{\varphi}{_0_0}$ \\ \midrule
	$\tensor{D}{_\infty_\alpha_0}$ & $\tfrac{1}{2}(\euler-3)\tensor{\varphi}{_0_\alpha}$ \\ \midrule
	$\tensor{D}{_\infty_\infty_\alpha}$ & $0$ \\ \midrule
	$\tensor{D}{_\infty_0_\alpha}$ & $\tfrac{1}{2}(\euler-3)\tensor{\varphi}{_0_\alpha}$ \\ \midrule
	$\tensor{D}{_\infty_\conjbeta_\alpha}$
	& $-\tensor{h}{_\alpha_\conjbeta}+\tfrac{1}{2}(\euler-2)\tensor{\varphi}{_\alpha_\conjbeta}$ \\ \midrule
	$\tensor{D}{_\infty_\beta_\alpha}$ & $\tfrac{1}{2}(\euler-2)\tensor{\varphi}{_\alpha_\beta}$ \\ \midrule
	$\tensor{D}{_0_\infty_0}$ & $2-\tfrac{1}{2}(\euler-4)\tensor{\varphi}{_0_0}$ \\ \midrule
	$\tensor{D}{_0_0_0}$ & $\tfrac{1}{2}\rho^2\tensor{\nabla}{_0}\tensor{\varphi}{_0_0}$ \\ \midrule
	$\tensor{D}{_0_\alpha_0}$
	& $-\tfrac{1}{2}\rho\tensor{\nabla}{_\alpha}\tensor{\varphi}{_0_0}
		+\rho^2(\tensor{\nabla}{_0}\tensor{\varphi}{_0_\alpha}
			+\tensor{A}{_\alpha^\conjbeta}\tensor{\varphi}{_0_\conjbeta})$ \\ \midrule
	$\tensor{D}{_0_\infty_\alpha}$ & $-\tfrac{1}{2}(\euler-3)\tensor{\varphi}{_0_\alpha}$ \\ \midrule
	$\tensor{D}{_0_0_\alpha}$ & $\tfrac{1}{2}\rho\tensor{\nabla}{_\alpha}\tensor{\varphi}{_0_0}$ \\ \midrule
	$\tensor{D}{_0_\conjbeta_\alpha}$
	& $\tfrac{i}{2}\tensor{h}{_\alpha_\conjbeta}
		+\tfrac{i}{2}\tensor{h}{_\alpha_\conjbeta}\tensor{\varphi}{_0_0}
		+\tfrac{1}{2}\rho
			(\tensor{\nabla}{_\alpha}\tensor{\varphi}{_0_\conjbeta}
			-\tensor{\nabla}{_\conjbeta}\tensor{\varphi}{_0_\alpha})
		+\tfrac{1}{2}\rho^2
			(\tensor{\nabla}{_0}\tensor{\varphi}{_\alpha_\conjbeta}
			+\tensor{A}{_\alpha^\conjgamma}\tensor{\varphi}{_\conjbeta_\conjgamma}
			+\tensor{A}{_\conjbeta^\gamma}\tensor{\varphi}{_\alpha_\gamma})$ \\ \midrule
	$\tensor{D}{_0_\beta_\alpha}$
	& $\rho^2\tensor{A}{_\alpha_\beta}
		+\tfrac{1}{2}\rho
			(\tensor{\nabla}{_\alpha}\tensor{\varphi}{_0_\beta}
			-\tensor{\nabla}{_\beta}\tensor{\varphi}{_0_\alpha}
			-\tensor{N}{_\alpha_\beta^\conjgamma}\tensor{\varphi}{_0_\conjgamma})
		+\tfrac{1}{2}\rho^2
			(\tensor{\nabla}{_0}\tensor{\varphi}{_\alpha_\beta}
			+\tensor{A}{_\alpha^\conjgamma}\tensor{\varphi}{_\beta_\conjgamma}
			+\tensor{A}{_\beta^\conjgamma}\tensor{\varphi}{_\alpha_\conjgamma})$ \\ \midrule
	$\tensor{D}{_\alpha_\infty_0}$ & $-\tfrac{1}{2}(\euler-3)\tensor{\varphi}{_0_\alpha}$ \\ \midrule
	$\tensor{D}{_\alpha_0_0}$
	& $\tfrac{1}{2}\rho\tensor{\nabla}{_\alpha}\tensor{\varphi}{_0_0}
		-\rho^2\tensor{A}{_\alpha^\conjbeta}\tensor{\varphi}{_0_\conjbeta}$ \\ \midrule
	$\tensor{D}{_\alpha_\conjbeta_0}$
	& $\tfrac{i}{2}\tensor{h}{_\alpha_\conjbeta}
		+\tfrac{i}{2}\tensor{h}{_\alpha_\conjbeta}\tensor{\varphi}{_0_0}
		+\tfrac{1}{2}\rho
			(\tensor{\nabla}{_\alpha}\tensor{\varphi}{_0_\conjbeta}
			-\tensor{\nabla}{_\conjbeta}\tensor{\varphi}{_0_\alpha})
		+\tfrac{1}{2}\rho^2
			(\tensor{\nabla}{_0}\tensor{\varphi}{_\alpha_\conjbeta}
			-\tensor{A}{_\alpha^\conjgamma}\tensor{\varphi}{_\conjbeta_\conjgamma}
			+\tensor{A}{_\conjbeta^\gamma}\tensor{\varphi}{_\alpha_\gamma})$ \\ \midrule
	$\tensor{D}{_\alpha_\beta_0}$
	& $\tfrac{1}{2}\rho
			(\tensor{\nabla}{_\alpha}\tensor{\varphi}{_0_\beta}
			-\tensor{\nabla}{_\beta}\tensor{\varphi}{_0_\alpha}
			-\tensor{N}{_\alpha_\beta^\conjgamma}\tensor{\varphi}{_0_\conjgamma})
		+\tfrac{1}{2}\rho^2
			(\tensor{\nabla}{_0}\tensor{\varphi}{_\alpha_\beta}
			-\tensor{A}{_\alpha^\conjgamma}\tensor{\varphi}{_\beta_\conjgamma}
			+\tensor{A}{_\beta^\conjgamma}\tensor{\varphi}{_\alpha_\conjgamma})$ \\ \midrule
	$\tensor{D}{_\alpha_\infty_\conjbeta}$
	& $\tensor{h}{_\alpha_\conjbeta}-\tfrac{1}{2}(\euler-2)\tensor{\varphi}{_\alpha_\conjbeta}$ \\ \midrule
	$\tensor{D}{_\alpha_0_\conjbeta}$
	& $\tfrac{i}{2}\tensor{h}{_\alpha_\conjbeta}
		+\tfrac{i}{2}\tensor{h}{_\alpha_\conjbeta}\tensor{\varphi}{_0_0}
		+\tfrac{1}{2}\rho
			(\tensor{\nabla}{_\alpha}\tensor{\varphi}{_0_\conjbeta}
			+\tensor{\nabla}{_\conjbeta}\tensor{\varphi}{_0_\alpha})
		-\tfrac{1}{2}\rho^2
			(\tensor{\nabla}{_0}\tensor{\varphi}{_\alpha_\conjbeta}
			+\tensor{A}{_\alpha^\conjgamma}\tensor{\varphi}{_\conjbeta_\conjgamma}
			+\tensor{A}{_\conjbeta^\gamma}\tensor{\varphi}{_\alpha_\gamma})$ \\ \midrule
	$\tensor{D}{_\alpha_\conjgamma_\conjbeta}$
	& $\tfrac{i}{2}(\tensor{h}{_\alpha_\conjgamma}\tensor{\varphi}{_0_\conjbeta}
			+\tensor{h}{_\alpha_\conjbeta}\tensor{\varphi}{_0_\conjgamma})
		+\tfrac{1}{2}\rho
			(\tensor{\nabla}{_\alpha}\tensor{\varphi}{_\conjbeta_\conjgamma}
			+\tensor{\nabla}{_\conjbeta}\tensor{\varphi}{_\alpha_\conjgamma}
			-\tensor{\nabla}{_\conjgamma}\tensor{\varphi}{_\alpha_\conjbeta}
			-\tensor{N}{_\conjbeta_\conjgamma^\sigma}\tensor{\varphi}{_\alpha_\sigma})$ \\ \midrule
	$\tensor{D}{_\alpha_\gamma_\conjbeta}$
	& $-\tfrac{i}{2}(\tensor{h}{_\gamma_\conjbeta}\tensor{\varphi}{_0_\alpha}
			-\tensor{h}{_\alpha_\conjbeta}\tensor{\varphi}{_0_\gamma})
		+\tfrac{1}{2}\rho
			(\tensor{\nabla}{_\alpha}\tensor{\varphi}{_\gamma_\conjbeta}
			+\tensor{\nabla}{_\conjbeta}\tensor{\varphi}{_\alpha_\gamma}
			-\tensor{\nabla}{_\gamma}\tensor{\varphi}{_\alpha_\conjbeta}
			-\tensor{N}{_\alpha_\gamma^\conjsigma}\tensor{\varphi}{_\conjbeta_\conjsigma})$ \\ \midrule
	$\tensor{D}{_\alpha_\infty_\beta}$ & $-\tfrac{1}{2}(\euler-2)\tensor{\varphi}{_\alpha_\beta}$ \\ \midrule
	$\tensor{D}{_\alpha_0_\beta}$
	& $-\rho^2\tensor{A}{_\alpha_\beta}
			+\tfrac{1}{2}\rho
			(\tensor{\nabla}{_\alpha}\tensor{\varphi}{_0_\beta}
			+\tensor{\nabla}{_\beta}\tensor{\varphi}{_0_\alpha}
			+\tensor{N}{_\alpha_\beta^\conjgamma}\tensor{\varphi}{_0_\conjgamma})
			-\tfrac{1}{2}\rho^2
			(\tensor{\nabla}{_0}\tensor{\varphi}{_\alpha_\beta}
			+\tensor{A}{_\alpha^\conjgamma}\tensor{\varphi}{_\beta_\conjgamma}
			+\tensor{A}{_\beta^\conjgamma}\tensor{\varphi}{_\alpha_\conjgamma})$ \\ \midrule
	$\tensor{D}{_\alpha_\conjgamma_\beta}$
	& $\tfrac{i}{2}(\tensor{h}{_\alpha_\conjgamma}\tensor{\varphi}{_0_\beta}
			+\tensor{h}{_\beta_\conjgamma}\tensor{\varphi}{_0_\alpha})
		+\tfrac{1}{2}\rho
			(\tensor{\nabla}{_\alpha}\tensor{\varphi}{_\beta_\conjgamma}
			+\tensor{\nabla}{_\beta}\tensor{\varphi}{_\alpha_\conjgamma}
			-\tensor{\nabla}{_\conjgamma}\tensor{\varphi}{_\alpha_\beta}
			-\tensor{N}{_\alpha_\beta^\conjsigma}\tensor{\varphi}{_\conjgamma_\conjsigma})$ \\ \midrule
	$\tensor{D}{_\alpha_\gamma_\beta}$
	& $-\rho\tensor{N}{_\alpha_\gamma_\beta}
		+\tfrac{1}{2}\rho
			(\tensor{\nabla}{_\alpha}\tensor{\varphi}{_\beta_\gamma}
			+\tensor{\nabla}{_\beta}\tensor{\varphi}{_\alpha_\gamma}
			-\tensor{\nabla}{_\gamma}\tensor{\varphi}{_\alpha_\beta}
			-\tensor{N}{_\alpha_\beta^\conjsigma}\tensor{\varphi}{_\gamma_\conjsigma}
			-\tensor{N}{_\alpha_\gamma^\conjsigma}\tensor{\varphi}{_\beta_\conjsigma}
			-\tensor{N}{_\beta_\gamma^\conjsigma}\tensor{\varphi}{_\alpha_\conjsigma})$ \\ \bottomrule
	\end{tabular}
	\vspace{0.5em}
	\caption{$\tensor{D}{_I_K_J}$ for a normal-form ACH metric $g$.
	$\tensor{D}{_0_K_\infty}$ and $\tensor{D}{_\alpha_K_\infty}$ are omitted;
	we have $\tensor{D}{_0_K_\infty}=\tensor{D}{_\infty_K_0}$
	and $\tensor{D}{_\alpha_K_\infty}=\tensor{D}{_\infty_K_\alpha}$}
	\label{tbl:ConnectionFormsOfACHMetricWithLowerIndices}
\end{table}

To prove Proposition \ref{prop:ACHEToThirdOrder}, it is enough to calculate everything modulo $O(\rho^3)$.
However, for later use, we shall carry out more precise computation.
What we allow ourselves to neglect are
\begin{itemize}
	\item[(N1)] any term at least quadratic in $\tensor{\varphi}{_i_j_,_k_\dots}$ with $O(1)$ coefficients,
	\item[(N2)] any term linear in $\tensor{\varphi}{_0_0_,_k_\dots}$, $\tensor{\varphi}{_0_\alpha_,_k_\dots}$,
		$\tensor{\varphi}{_0_\conjalpha_,_k_\dots}$ or $\tensor{\varphi}{_\alpha_\conjbeta_,_k_\dots}$
		with $O(\rho)$ coefficients which vanish in the case of the CR sphere with standard
		pseudohermitian structure, and
	\item[(N3)] any term linear in $\tensor{\varphi}{_\alpha_\beta_,_k_\dots}$ or
		$\tensor{\varphi}{_\conjalpha_\conjbeta_,_k_\dots}$
		with $O(\rho^2)$ coefficients which vanish in the case of the CR sphere with standard
		pseudohermitian structure.
\end{itemize}
Modulo terms of type (N1), $\tensor{g}{^I^J}$ is given by
\begin{equation}
	\label{eq:InverseOfNormalACH}
	\begin{split}
		\tensor{g}{^\infty^\infty}&\equiv\tfrac{1}{4},\qquad
		\tensor{g}{^\infty^0}\equiv\tensor{g}{^\infty^\alpha}\equiv 0,\\
		\tensor{g}{^0^0}&\equiv 1-\tensor{\varphi}{_0_0},\qquad
		\tensor{g}{^0^\alpha}\equiv -\tensor{\varphi}{_0^\alpha},\qquad
		\tensor{g}{^\alpha^\conjbeta}\equiv\tensor{h}{^\alpha^\conjbeta}-\tensor{\varphi}{^\alpha^\conjbeta},\qquad
		\tensor{g}{^\alpha^\beta}\equiv -\tensor{\varphi}{^\alpha^\beta}.
	\end{split}
\end{equation}
By these formulae and Table \ref{tbl:ConnectionFormsOfACHMetricWithLowerIndices},
we compute $\tensor{D}{_I^K_J}$ modulo terms of type (N1)--(N3) using the equality
$\tensor{D}{_I^K_J}=\tensor{g}{^K^L}\tensor{D}{_I_L_J}$.
Table \ref{tbl:ConnectionFormsOfACHMetricWithUpperIndices} is the result.

\begin{table}[htbp]
	\small
\begin{tabular}{ll}
	\toprule
	\multicolumn{1}{c}{Type} & \multicolumn{1}{c}{Value (modulo terms of type (N1)--(N3))} \\ \midrule
	$\tensor{D}{_\infty^\infty_\infty}$ & $-1$ \\ \midrule
	$\tensor{D}{_\infty^0_\infty}$ & $0$ \\ \midrule
	$\tensor{D}{_\infty^\alpha_\infty}$ & $0$ \\ \midrule
	$\tensor{D}{_\infty^\infty_0}$ & $0$ \\ \midrule
	$\tensor{D}{_\infty^0_0}$ & $-2+\tfrac{1}{2}\euler\tensor{\varphi}{_0_0}$ \\ \midrule
	$\tensor{D}{_\infty^\alpha_0}$ & $\tfrac{1}{2}(\euler+1)\tensor{\varphi}{_0^\alpha}$ \\ \midrule
	$\tensor{D}{_\infty^\infty_\alpha}$ & $0$ \\ \midrule
	$\tensor{D}{_\infty^0_\alpha}$ & $\tfrac{1}{2}(\euler-1)\tensor{\varphi}{_0_\alpha}$ \\ \midrule
	$\tensor{D}{_\infty^\beta_\alpha}$
	& $-\tensor{\delta}{_\alpha^\beta}+\tfrac{1}{2}\euler\tensor{\varphi}{_\alpha^\beta}$ \\ \midrule
	$\tensor{D}{_\infty^\conjbeta_\alpha}$ & $\tfrac{1}{2}\euler\tensor{\varphi}{_\alpha^\conjbeta}$ \\ \midrule
	$\tensor{D}{_0^\infty_0}$ & $\tfrac{1}{2}-\tfrac{1}{8}(\euler-4)\tensor{\varphi}{_0_0}$ \\ \midrule
	$\tensor{D}{_0^0_0}$ & $\tfrac{1}{2}\rho^2\tensor{\nabla}{_0}\tensor{\varphi}{_0_0}$ \\ \midrule
	$\tensor{D}{_0^\alpha_0}$
	& $-\tfrac{1}{2}\rho\tensor{\nabla}{^\alpha}\tensor{\varphi}{_0_0}
		+\rho^2\tensor{\nabla}{_0}\tensor{\varphi}{_0^\alpha}$ \\ \midrule
	$\tensor{D}{_0^\infty_\alpha}$ & $-\tfrac{1}{8}(\euler-3)\tensor{\varphi}{_0_\alpha}$ \\ \midrule
	$\tensor{D}{_0^0_\alpha}$
	& $-\tfrac{i}{2}\tensor{\varphi}{_0_\alpha}+\tfrac{1}{2}\rho\tensor{\nabla}{_\alpha}\tensor{\varphi}{_0_0}$
	\\ \midrule
	$\tensor{D}{_0^\beta_\alpha}$
	& $\tfrac{i}{2}\tensor{\delta}{_\alpha^\beta}
		+\tfrac{i}{2}\tensor{\delta}{_\alpha^\beta}\tensor{\varphi}{_0_0}
		-\tfrac{i}{2}\tensor{\varphi}{_\alpha^\beta}
		+\tfrac{1}{2}\rho(\tensor{\nabla}{_\alpha}\tensor{\varphi}{_0^\beta}
		-\tensor{\nabla}{^\beta}\tensor{\varphi}{_0_\alpha})
		+\tfrac{1}{2}\rho^2\tensor{\nabla}{_0}\tensor{\varphi}{_\alpha^\beta}$ \\ \midrule
	$\tensor{D}{_0^\conjbeta_\alpha}$
	& $\rho^2\tensor{A}{_\alpha^\conjbeta}-\tfrac{i}{2}\tensor{\varphi}{_\alpha^\conjbeta}
		+\tfrac{1}{2}\rho(\tensor{\nabla}{_\alpha}\tensor{\varphi}{_0^\conjbeta}
			-\tensor{\nabla}{^\conjbeta}\tensor{\varphi}{_0_\alpha})
		+\tfrac{1}{2}\rho^2\tensor{\nabla}{_0}\tensor{\varphi}{_\alpha^\conjbeta}$ \\ \midrule
	$\tensor{D}{_\alpha^\infty_0}$ & $-\tfrac{1}{8}(\euler-3)\tensor{\varphi}{_0_\alpha}$ \\ \midrule
	$\tensor{D}{_\alpha^0_0}$
	& $-\tfrac{i}{2}\tensor{\varphi}{_0_\alpha}
		+\tfrac{1}{2}\rho\tensor{\nabla}{_\alpha}\tensor{\varphi}{_0_0}$ \\ \midrule
	$\tensor{D}{_\alpha^\beta_0}$
	& $\tfrac{i}{2}\tensor{\delta}{_\alpha^\beta}
		+\tfrac{i}{2}\tensor{\delta}{_\alpha^\beta}\tensor{\varphi}{_0_0}
		-\tfrac{i}{2}\tensor{\varphi}{_\alpha^\beta}
		+\tfrac{1}{2}\rho(\tensor{\nabla}{_\alpha}\tensor{\varphi}{_0^\beta}
			-\tensor{\nabla}{^\beta}\tensor{\varphi}{_0_\alpha})
		+\tfrac{1}{2}\rho^2\tensor{\nabla}{_0}\tensor{\varphi}{_\alpha^\beta}$ \\ \midrule
	$\tensor{D}{_\alpha^\conjbeta_0}$
	& $-\tfrac{i}{2}\tensor{\varphi}{_\alpha^\conjbeta}
		+\tfrac{1}{2}\rho(\tensor{\nabla}{_\alpha}\tensor{\varphi}{_0^\conjbeta}
			-\tensor{\nabla}{^\conjbeta}\tensor{\varphi}{_0_\alpha})
		+\tfrac{1}{2}\rho^2\tensor{\nabla}{_0}\tensor{\varphi}{_\alpha^\conjbeta}$ \\ \midrule
	$\tensor{D}{_\alpha^\infty_\conjbeta}$
	& $\tfrac{1}{4}\tensor{h}{_\alpha_\conjbeta}-\tfrac{1}{8}(\euler-2)\tensor{\varphi}{_\alpha_\conjbeta}$
	\\ \midrule
	$\tensor{D}{_\alpha^0_\conjbeta}$
	& $\tfrac{i}{2}\tensor{h}{_\alpha_\conjbeta}
		+\tfrac{1}{2}\rho(\tensor{\nabla}{_\alpha}\tensor{\varphi}{_0_\conjbeta}
		+\tensor{\nabla}{_\conjbeta}\tensor{\varphi}{_0_\alpha})
		-\tfrac{1}{2}\rho^2\tensor{\nabla}{_0}\tensor{\varphi}{_\alpha_\conjbeta}$ \\ \midrule
	$\tensor{D}{_\alpha^\gamma_\conjbeta}$
	& $\tfrac{i}{2}\tensor{\delta}{_\alpha^\gamma}\tensor{\varphi}{_0_\conjbeta}
		+\tfrac{1}{2}\rho(\tensor{\nabla}{_\alpha}\tensor{\varphi}{_\conjbeta^\gamma}
			+\tensor{\nabla}{_\conjbeta}\tensor{\varphi}{_\alpha^\gamma}
			-\tensor{\nabla}{^\gamma}\tensor{\varphi}{_\alpha_\conjbeta})
			-\tfrac{1}{2}\rho\tensor{N}{_\conjbeta^\gamma^\sigma}\tensor{\varphi}{_\alpha_\sigma}$ \\ \midrule
	$\tensor{D}{_\alpha^\conjgamma_\conjbeta}$
	& $-\tfrac{i}{2}\tensor{\delta}{_\conjbeta^\conjgamma}\tensor{\varphi}{_0_\alpha}
		+\tfrac{1}{2}\rho(\tensor{\nabla}{_\alpha}\tensor{\varphi}{_\conjbeta^\conjgamma}
			+\tensor{\nabla}{_\conjbeta}\tensor{\varphi}{_\alpha^\conjgamma}
			-\tensor{\nabla}{^\conjgamma}\tensor{\varphi}{_\alpha_\conjbeta})
		-\tfrac{1}{2}\rho\tensor{N}{_\alpha^\conjgamma^\conjsigma}\tensor{\varphi}{_\conjbeta_\conjsigma}$
		\\ \midrule
	$\tensor{D}{_\alpha^\infty_\beta}$ & $-\tfrac{1}{8}(\euler-2)\tensor{\varphi}{_\alpha_\beta}$ \\ \midrule
	$\tensor{D}{_\alpha^0_\beta}$
	& $-\rho^2\tensor{A}{_\alpha_\beta}
		+\tfrac{1}{2}\rho(\tensor{\nabla}{_\alpha}\tensor{\varphi}{_0_\beta}
			+\tensor{\nabla}{_\beta}\tensor{\varphi}{_0_\alpha})
		-\tfrac{1}{2}\rho\tensor{N}{_\alpha_\beta^\conjgamma}\tensor{\varphi}{_0_\conjgamma}
		-\tfrac{1}{2}\rho^2\tensor{\nabla}{_0}\tensor{\varphi}{_\alpha_\beta}$ \\ \midrule
	$\tensor{D}{_\alpha^\gamma_\beta}$
	& $\tfrac{i}{2}(\tensor{\delta}{_\alpha^\gamma}\tensor{\varphi}{_0_\beta}
			+\tensor{\delta}{_\beta^\gamma}\tensor{\varphi}{_0_\alpha})
		+\tfrac{1}{2}\rho(\tensor{\nabla}{_\alpha}\tensor{\varphi}{_\beta^\gamma}
			+\tensor{\nabla}{_\beta}\tensor{\varphi}{_\alpha^\gamma}
			-\tensor{\nabla}{^\gamma}\tensor{\varphi}{_\alpha_\beta})
			-\tfrac{1}{2}\rho\tensor{N}{_\alpha_\beta^\conjsigma}\tensor{\varphi}{_\conjsigma^\gamma}$ \\ \midrule
	$\tensor{D}{_\alpha^\conjgamma_\beta}$
	& $-\rho\tensor{N}{_\alpha^\conjgamma_\beta}
		+\tfrac{1}{2}\rho(\tensor{\nabla}{_\alpha}\tensor{\varphi}{_\beta^\conjgamma}
			+\tensor{\nabla}{_\beta}\tensor{\varphi}{_\alpha^\conjgamma}
			-\tensor{\nabla}{^\conjgamma}\tensor{\varphi}{_\alpha_\beta})$ \\ \bottomrule
\end{tabular}
	\vspace{0.5em}
	\caption{$\tensor{D}{_I^K_J}$ for a normal-form ACH metric $g$.
	$\tensor{D}{_0^K_\infty}$ and $\tensor{D}{_\alpha^K_\infty}$ are omitted;
	we have $\tensor{D}{_0^K_\infty}=\tensor{D}{_\infty^K_0}$
	and $\tensor{D}{_\alpha^K_\infty}=\tensor{D}{_\infty^K_\alpha}$}
	\label{tbl:ConnectionFormsOfACHMetricWithUpperIndices}
\end{table}

Finally we can show the following formulae for the Einstein tensor.
We define the sublaplacian by
$\sublaplacian:=-(\tensor{\nabla}{^\alpha}\tensor{\nabla}{_\alpha}
+\tensor{\nabla}{^\conjalpha}\tensor{\nabla}{_\conjalpha})$.

\begin{lem}
	\label{lem:RicciTensorModuloHighOrderTerm}
	The Einstein tensor of an ACH metric $g$ is, modulo terms of type (N1)--(N3),
	\allowdisplaybreaks
	\begin{align*}
		\tensor{\Ein}{_\infty_\infty}
		&\equiv -\tfrac{1}{2}\euler(\euler-4)\tensor{\varphi}{_0_0}
		-\euler(\euler-2)\tensor{\varphi}{_\alpha^\alpha},\\
		\tensor{\Ein}{_\infty_0}
		&\equiv\tfrac{1}{2}\rho(\euler+1)
		(\tensor{\nabla}{^\alpha}\tensor{\varphi}{_0_\alpha}
		+\tensor{\nabla}{^\conjalpha}\tensor{\varphi}{_0_\conjalpha})
		-\rho^2(\euler+1)\tensor{\nabla}{_0}\tensor{\varphi}{_\alpha^\alpha},\\
		\begin{split}
			\tensor{\Ein}{_\infty_\alpha}
			&\equiv -\tfrac{1}{2}i(\euler+1)\tensor{\varphi}{_0_\alpha}
			-\tfrac{1}{2}\rho(\euler-1)\tensor{\nabla}{_\alpha}\tensor{\varphi}{_0_0}
			-\rho^2\partial_\rho\tensor{\nabla}{_\alpha}\tensor{\varphi}{_\beta^\beta}\\
			&\phantom{=\;}
			+\tfrac{1}{2}\rho^2\partial_\rho(\tensor{\nabla}{^\conjbeta}\tensor{\varphi}{_\alpha_\conjbeta}
			+\tensor{\nabla}{^\beta}\tensor{\varphi}{_\alpha_\beta})
			+\tfrac{1}{2}\rho^2
				\partial_\rho\tensor{N}{_\alpha^\conjbeta^\conjgamma}\tensor{\varphi}{_\conjbeta_\conjgamma}
			+\tfrac{1}{2}\rho^2(\euler-1)\tensor{\nabla}{_0}\tensor{\varphi}{_0_\alpha},
		\end{split}\\
		\begin{split}
			\tensor{\Ein}{_0_0}
			&\equiv -\tfrac{1}{8}\left( (\euler)^2-(2n+4)\euler-4n \right)\tensor{\varphi}{_0_0}
			+\tfrac{1}{2}(\euler-2)\tensor{\varphi}{_\alpha^\alpha}\\
			&\phantom{=\;}
			+i\rho(\tensor{\nabla}{^\alpha}\tensor{\varphi}{_0_\alpha}
			-\tensor{\nabla}{^\conjalpha}\tensor{\varphi}{_0_\conjalpha})
			+\tfrac{1}{2}\rho^2\sublaplacian\tensor{\varphi}{_0_0}
			+\rho^3(\tensor{\nabla}{_0}\tensor{\nabla}{^\alpha}\tensor{\varphi}{_0_\alpha}
			+\tensor{\nabla}{_0}\tensor{\nabla}{^\conjalpha}\tensor{\varphi}{_0_\conjalpha})
			-\rho^4\tensor{\nabla}{_0}\tensor{\nabla}{_0}\tensor{\varphi}{_\alpha^\alpha}
		\end{split}\\
		\begin{split}
			\tensor{\Ein}{_0_\alpha}
			&\equiv\rho^3\tensor{A}{_\alpha_\beta_,^\beta}
			+\rho^3\tensor{N}{_\alpha^\conjbeta^\conjgamma}\tensor{A}{_\conjbeta_\conjgamma}
			-\tfrac{1}{8}(\euler+1)(\euler-2n-3)\tensor{\varphi}{_0_\alpha}\\
			&\phantom{=\;}
				+\tfrac{3}{4}i\rho\tensor{\nabla}{_\alpha}\tensor{\varphi}{_0_0}
				+\tfrac{1}{2}i\rho\tensor{\nabla}{_\alpha}\tensor{\varphi}{_\beta^\beta}
				-i\rho\tensor{\nabla}{^\conjbeta}\tensor{\varphi}{_\alpha_\conjbeta}
				+\tfrac{1}{2}\rho^2\sublaplacian\tensor{\varphi}{_0_\alpha}
				-\tfrac{1}{2}i\rho^2\tensor{\nabla}{_0}\tensor{\varphi}{_0_\alpha}\\
			&\phantom{=\;}
				+\tfrac{1}{2}\rho^2(\tensor{\nabla}{_\alpha}\tensor{\nabla}{^\beta}\tensor{\varphi}{_0_\beta}
					+\tensor{\nabla}{_\alpha}\tensor{\nabla}{^\conjbeta}\tensor{\varphi}{_0_\conjbeta})
				-\rho^3\tensor{\nabla}{_0}\tensor{\nabla}{_\alpha}\tensor{\varphi}{_\beta^\beta}
				+\tfrac{1}{2}\rho^3(\tensor{\nabla}{_0}\tensor{\nabla}{^\conjbeta}
					\tensor{\varphi}{_\alpha_\conjbeta}
					+\tensor{\nabla}{_0}\tensor{\nabla}{^\beta}\tensor{\varphi}{_\alpha_\beta}),
		\end{split}\\
		\begin{split}
			\tensor{\Ein}{_\alpha_\conjbeta}
			&\equiv \rho^2\tensor{R}{_\alpha_\conjbeta}
			-2\rho^2\tensor{N}{_\alpha^\conjgamma_\rho}\tensor{N}{_\conjbeta^\rho_\conjgamma}
			-\tfrac{1}{8}\left( (\euler)^2-(2n+2)\euler-8 \right)\tensor{\varphi}{_\alpha_\conjbeta}\\
			&\phantom{=\;}
			+\tfrac{1}{8}\tensor{h}{_\alpha_\conjbeta}(\euler-4)\tensor{\varphi}{_0_0}
			+\tfrac{1}{4}\tensor{h}{_\alpha_\conjbeta}\euler\tensor{\varphi}{_\gamma^\gamma}
			+i\rho(\tensor{\nabla}{_\alpha}\tensor{\varphi}{_0_\conjbeta}
				-\tensor{\nabla}{_\conjbeta}\tensor{\varphi}{_0_\alpha})\\
			&\phantom{=\;}
			-\tfrac{1}{4}i\rho^2\tensor{h}{_\alpha_\conjbeta}\tensor{\nabla}{_0}\tensor{\varphi}{_0_0}
			-\tfrac{1}{2}i\rho^2\tensor{h}{_\alpha_\conjbeta}\tensor{\nabla}{_0}\tensor{\varphi}{_\gamma^\gamma}
			-\tfrac{1}{2}\rho^2\tensor{\nabla}{_\alpha}\tensor{\nabla}{_\conjbeta}\tensor{\varphi}{_0_0}
			-\rho^2\tensor{\nabla}{_\alpha}\tensor{\nabla}{_\conjbeta}\tensor{\varphi}{_\gamma^\gamma}\\
			&\phantom{=\;}
			+\tfrac{1}{2}\rho^2(\sublaplacian\tensor{\varphi}{_\alpha_\conjbeta}
				+\tensor{\nabla}{_\alpha}\tensor{\nabla}{^\gamma}\tensor{\varphi}{_\conjbeta_\gamma}
				+\tensor{\nabla}{_\alpha}\tensor{\nabla}{^\conjgamma}\tensor{\varphi}{_\conjbeta_\conjgamma}
				+\tensor{\nabla}{_\conjbeta}\tensor{\nabla}{^\conjgamma}\tensor{\varphi}{_\alpha_\conjgamma}
				+\tensor{\nabla}{_\conjbeta}\tensor{\nabla}{^\gamma}\tensor{\varphi}{_\alpha_\gamma})\\
			&\phantom{=\;}
			+\tfrac{1}{2}\rho^3(\tensor{\nabla}{_0}\tensor{\nabla}{_\alpha}\tensor{\varphi}{_0_\conjbeta}
				+\tensor{\nabla}{_0}\tensor{\nabla}{_\conjbeta}\tensor{\varphi}{_0_\alpha})
			-\tfrac{1}{2}\rho^4\tensor{\nabla}{_0}\tensor{\nabla}{_0}\tensor{\varphi}{_\alpha_\conjbeta},
		\end{split}\\
		\begin{split}
			\tensor{\Ein}{_\alpha_\beta}
			&\equiv in\rho^2\tensor{A}{_\alpha_\beta}
			+\rho^2(\tensor{N}{_\gamma_\alpha_\beta_,^\gamma}+\tensor{N}{_\gamma_\beta_\alpha_,^\gamma})
			-\rho^4\tensor{A}{_\alpha_\beta_,_0}
			-\tfrac{1}{8}\euler(\euler-2n-2)\tensor{\varphi}{_\alpha_\beta}\\
			&\phantom{=\;}
			-\tfrac{1}{2}\rho^2\tensor{\nabla}{_\alpha}\tensor{\nabla}{_\beta}\tensor{\varphi}{_0_0}
			-\rho^2\tensor{\nabla}{_\alpha}\tensor{\nabla}{_\beta}\tensor{\varphi}{_\gamma^\gamma}\\
			&\phantom{=\;}
			+\tfrac{1}{2}\rho^2(\sublaplacian\tensor{\varphi}{_\alpha_\beta}
				+\tensor{\nabla}{_\alpha}\tensor{\nabla}{^\conjgamma}\tensor{\varphi}{_\beta_\conjgamma}
				+\tensor{\nabla}{_\alpha}\tensor{\nabla}{^\gamma}\tensor{\varphi}{_\beta_\gamma}
				+\tensor{\nabla}{_\beta}\tensor{\nabla}{^\conjgamma}\tensor{\varphi}{_\alpha_\conjgamma}
				+\tensor{\nabla}{_\beta}\tensor{\nabla}{^\gamma}\tensor{\varphi}{_\alpha_\gamma}
				+2i\tensor{\nabla}{_0}\tensor{\varphi}{_\alpha_\beta})\\
			&\phantom{=\;}
			+\tfrac{1}{2}\rho^3(\tensor{\nabla}{_0}\tensor{\nabla}{_\alpha}\tensor{\varphi}{_0_\beta}
				+\tensor{\nabla}{_0}\tensor{\nabla}{_\beta}\tensor{\varphi}{_0_\alpha})
			-\tfrac{1}{2}\rho^4\tensor{\nabla}{_0}\tensor{\nabla}{_0}\tensor{\varphi}{_\alpha_\beta}.
		\end{split}
	\end{align*}
	\allowdisplaybreaks[0]%
\end{lem}

\begin{proof}
	Using Table \ref{tbl:ConnectionFormsOfACHMetricWithUpperIndices} we compute,
	modulo terms of type (N1)--(N3),
	\begin{equation*}
		\tensor{\overline\nabla}{_K}\tensor{D}{_I^K_J},\qquad
		\tensor{\overline\nabla}{_J}\tensor{D}{_I^K_K},\qquad
		\tensor{D}{_I^L_K}\tensor{D}{_J^K_L}\qquad\text{and}\qquad
		\tensor{D}{_I^L_J}\tensor{D}{_L^K_K}
	\end{equation*}
	to obtain Tables \ref{tbl:DerivativesOfConnectionForms1}--\ref{tbl:ProductOfConnectionForms2}.
	Then, from \eqref{eq:RicciOfExtendedTWConnection} and \eqref{eq:RicciTensorOfACHMetric}, the lemma follows.
\end{proof}

\begin{table}[htbp]
	\small
\begin{tabular}{ll}
	\toprule
	\multicolumn{1}{c}{Type} & \multicolumn{1}{c}{Value (modulo terms of type (N1)--(N3))} \\ \midrule
	$\tensor{\overline\nabla}{_K}\tensor{D}{_\infty^K_\infty}$ & $1$ \\ \midrule
	$\tensor{\overline\nabla}{_K}\tensor{D}{_\infty^K_0}$
	& $\tfrac{1}{2}\rho^3\partial_\rho\tensor{\nabla}{_0}\tensor{\varphi}{_0_0}
		+\tfrac{1}{2}\rho(\euler+1)
		(\tensor{\nabla}{^\alpha}\tensor{\varphi}{_0_\alpha}
			+\tensor{\nabla}{^\conjalpha}\tensor{\varphi}{_0_\conjalpha})$ \\ \midrule
	$\tensor{\overline\nabla}{_K}\tensor{D}{_\infty^K_\alpha}$
	& $\tfrac{1}{2}\rho^2(\euler-1)\tensor{\nabla}{_0}\tensor{\varphi}{_0_\alpha}
		+\tfrac{1}{2}\rho^2\partial_\rho(\tensor{\nabla}{^\conjbeta}\tensor{\varphi}{_\alpha_\conjbeta}
		+\tensor{\nabla}{^\beta}\tensor{\varphi}{_\alpha_\beta})$ \\ \midrule
	$\tensor{\overline\nabla}{_K}\tensor{D}{_0^K_0}$
	& $-\tfrac{3}{2}-\tfrac{1}{8}(\euler-3)(\euler-4)\tensor{\varphi}{_0_0}
		+\tfrac{1}{2}\rho^2\sublaplacian\tensor{\varphi}{_0_0}
		+\rho^3(\tensor{\nabla}{_0}\tensor{\nabla}{^\alpha}\tensor{\varphi}{_0_\alpha}
			+\tensor{\nabla}{_0}\tensor{\nabla}{^\conjalpha}\tensor{\varphi}{_0_\conjalpha})$ \\
	& $+\tfrac{1}{2}\rho^4\tensor{\nabla}{_0}\tensor{\nabla}{_0}\tensor{\varphi}{_0_0}$ \\ \midrule
	$\tensor{\overline\nabla}{_K}\tensor{D}{_0^K_\alpha}$
	& $-\tfrac{1}{8}(\euler-2)(\euler-3)\tensor{\varphi}{_0_\alpha}
		+\tfrac{i}{2}\rho\tensor{\nabla}{_\alpha}\tensor{\varphi}{_0_0}
		-\tfrac{i}{2}\rho(\tensor{\nabla}{^\conjbeta}\tensor{\varphi}{_\alpha_\conjbeta}
		+\tensor{\nabla}{^\beta}\tensor{\varphi}{_\alpha_\beta})$ \\
	& $+\tfrac{1}{2}\rho^2\sublaplacian\tensor{\varphi}{_0_\alpha}
		+\tfrac{1}{2}\rho^2(\tensor{\nabla}{_\alpha}\tensor{\nabla}{^\beta}\tensor{\varphi}{_0_\beta}
			+\tensor{\nabla}{_\alpha}\tensor{\nabla}{^\conjbeta}\tensor{\varphi}{_0_\conjbeta})$ \\
	& $+\rho^3\tensor{A}{_\alpha_\beta_,^\beta}
		+\tfrac{1}{2}\rho^3\tensor{\nabla}{_0}\tensor{\nabla}{_\alpha}\tensor{\varphi}{_0_0}
		+\tfrac{1}{2}\rho^3(\tensor{\nabla}{_0}\tensor{\nabla}{^\beta}\tensor{\varphi}{_\alpha_\beta}
			+\tensor{\nabla}{_0}\tensor{\nabla}{^\conjbeta}\tensor{\varphi}{_\alpha_\conjbeta})$ \\ \midrule
	$\tensor{\overline\nabla}{_K}\tensor{D}{_\alpha^K_\conjbeta}$
	& $-\tfrac{1}{4}\tensor{h}{_\alpha_\conjbeta}
		-\tfrac{1}{8}(\euler-1)(\euler-2)\tensor{\varphi}{_\alpha_\conjbeta}
		+\tfrac{i}{2}\rho(\tensor{\nabla}{_\alpha}\tensor{\varphi}{_0_\conjbeta}
		-\tensor{\nabla}{_\conjbeta}\tensor{\varphi}{_0_\alpha})$ \\
	& $+\tfrac{1}{2}\rho^2(\sublaplacian\tensor{\varphi}{_\alpha_\conjbeta}
		+\tensor{\nabla}{_\alpha}\tensor{\nabla}{^\gamma}\tensor{\varphi}{_\conjbeta_\gamma}
		+\tensor{\nabla}{_\alpha}\tensor{\nabla}{^\conjgamma}\tensor{\varphi}{_\conjbeta_\conjgamma}
		+\tensor{\nabla}{_\conjbeta}\tensor{\nabla}{^\conjgamma}\tensor{\varphi}{_\alpha_\conjgamma}
		+\tensor{\nabla}{_\conjbeta}\tensor{\nabla}{^\gamma}\tensor{\varphi}{_\alpha_\gamma})$ \\
	& $+\tfrac{1}{2}\rho^3
		(\tensor{\nabla}{_0}\tensor{\nabla}{_\alpha}\tensor{\varphi}{_0_\conjbeta}
		+\tensor{\nabla}{_0}\tensor{\nabla}{_\conjbeta}\tensor{\varphi}{_0_\alpha})
		-\tfrac{1}{2}\rho^4\tensor{\nabla}{_0}\tensor{\nabla}{_0}\tensor{\varphi}{_\alpha_\conjbeta}$ \\ \midrule
	$\tensor{\overline\nabla}{_K}\tensor{D}{_\alpha^K_\beta}$
	& $\rho^2\tensor{N}{_\gamma_\alpha_\beta_,^\gamma}-\rho^4\tensor{A}{_\alpha_\beta_,_0}
		-\tfrac{1}{8}(\euler-1)(\euler-2)\tensor{\varphi}{_\alpha_\beta}
		+\tfrac{i}{2}\rho(\tensor{\nabla}{_\alpha}\tensor{\varphi}{_0_\beta}
		+\tensor{\nabla}{_\beta}\tensor{\varphi}{_0_\alpha})$ \\
	& $+\tfrac{1}{2}\rho^2(\sublaplacian\tensor{\varphi}{_\alpha_\beta}
		+\tensor{\nabla}{_\alpha}\tensor{\nabla}{^\conjgamma}\tensor{\varphi}{_\beta_\conjgamma}
		+\tensor{\nabla}{_\alpha}\tensor{\nabla}{^\gamma}\tensor{\varphi}{_\beta_\conjgamma}
		+\tensor{\nabla}{_\beta}\tensor{\nabla}{^\conjgamma}\tensor{\varphi}{_\alpha_\conjgamma}
		+\tensor{\nabla}{_\beta}\tensor{\nabla}{^\gamma}\tensor{\varphi}{_\alpha_\gamma})
		+i\rho^2\tensor{\nabla}{_0}\tensor{\varphi}{_\alpha_\beta}$ \\
	& $+\tfrac{1}{2}\rho^3
		(\tensor{\nabla}{_0}\tensor{\nabla}{_\alpha}\tensor{\varphi}{_0_\beta}
		+\tensor{\nabla}{_0}\tensor{\nabla}{_\beta}\tensor{\varphi}{_0_\alpha})
		-\tfrac{1}{2}\rho^4\tensor{\nabla}{_0}\tensor{\nabla}{_0}\tensor{\varphi}{_\alpha_\beta}$ \\ \bottomrule
\end{tabular}
	\vspace{0.5em}
	\caption{$\tensor{\overline\nabla}{_K}\tensor{D}{_I^K_J}$ for a normal-form ACH metric $g$}
	\label{tbl:DerivativesOfConnectionForms1}
\end{table}

\begin{table}[htbp]
	\small
\begin{tabular}{ll}
	\toprule
	\multicolumn{1}{c}{Type} & \multicolumn{1}{c}{Value (modulo terms of type (N1)--(N3))} \\ \midrule
	$\tensor{\overline\nabla}{_\infty}\tensor{D}{_\infty^K_K}$
	& $2n+3+\tfrac{1}{2}\euler(\euler-1)\tensor{\varphi}{_0_0}+\euler(\euler-1)\tensor{\varphi}{_\alpha^\alpha}$
	\\ \midrule
	$\tensor{\overline\nabla}{_0}\tensor{D}{_\infty^K_K}$
	& $\tfrac{1}{2}\rho^3\partial_\rho\tensor{\nabla}{_0}\tensor{\varphi}{_0_0}
	+\rho^3\partial_\rho\tensor{\nabla}{_0}\tensor{\varphi}{_\alpha^\alpha}$ \\ \midrule
	$\tensor{\overline\nabla}{_\alpha}\tensor{D}{_\infty^K_K}$
	& $\tfrac{1}{2}\rho^2\partial_\rho\tensor{\nabla}{_\alpha}\tensor{\varphi}{_0_0}
	+\rho^2\partial_\rho\tensor{\nabla}{_\alpha}\tensor{\varphi}{_\beta^\beta}$ \\ \midrule
	$\tensor{\overline\nabla}{_0}\tensor{D}{_0^K_K}$
	& $\tfrac{1}{2}\rho^4\tensor{\nabla}{_0}\tensor{\nabla}{_0}\tensor{\varphi}{_0_0}
	+\rho^4\tensor{\nabla}{_0}\tensor{\nabla}{_0}\tensor{\varphi}{_\alpha^\alpha}$ \\ \midrule
	$\tensor{\overline\nabla}{_\alpha}\tensor{D}{_0^K_K}$
	& $\tfrac{1}{2}\rho^3\tensor{\nabla}{_0}\tensor{\nabla}{_\alpha}\tensor{\varphi}{_0_0}
	+\rho^3\tensor{\nabla}{_0}\tensor{\nabla}{_\alpha}\tensor{\varphi}{_\beta^\beta}$ \\ \midrule
	$\tensor{\overline\nabla}{_\conjbeta}\tensor{D}{_\alpha^K_K}$
	& $\tfrac{1}{2}\rho^2\tensor{\nabla}{_\alpha}\tensor{\nabla}{_\conjbeta}\tensor{\varphi}{_0_0}
	+\rho^2\tensor{\nabla}{_\alpha}\tensor{\nabla}{_\conjbeta}\tensor{\varphi}{_\gamma^\gamma}
	+\tfrac{i}{2}\rho^2\tensor{h}{_\alpha_\conjbeta}\tensor{\nabla}{_0}\tensor{\varphi}{_0_0}
	+i\rho^2\tensor{h}{_\alpha_\conjbeta}\tensor{\nabla}{_0}\tensor{\varphi}{_\gamma^\gamma}$ \\ \midrule
	$\tensor{\overline\nabla}{_\beta}\tensor{D}{_\alpha^K_K}$
	& $\tfrac{1}{2}\rho^2\tensor{\nabla}{_\alpha}\tensor{\nabla}{_\beta}\tensor{\varphi}{_0_0}
	+\rho^2\tensor{\nabla}{_\alpha}\tensor{\nabla}{_\beta}\tensor{\varphi}{_\gamma^\gamma}$ \\ \bottomrule
\end{tabular}
	\vspace{0.5em}
	\caption{$\tensor{\overline\nabla}{_J}\tensor{D}{_I^K_K}$ for a normal-form ACH metric $g$}
	\label{tbl:DerivativesOfConnectionForms2}
\end{table}

\begin{table}[htbp]
	\small
\begin{tabular}{ll}
	\toprule
	\multicolumn{1}{c}{Type} & \multicolumn{1}{c}{Value (modulo terms of type (N1)--(N3))} \\ \midrule
	$\tensor{D}{_\infty^L_K}\tensor{D}{_\infty^K_L}$
	& $2n+5-2\euler\tensor{\varphi}{_0_0}-2\euler\tensor{\varphi}{_\alpha^\alpha}$ \\ \midrule
	$\tensor{D}{_\infty^L_K}\tensor{D}{_0^K_L}$
	& $-\rho^2\tensor{\nabla}{_0}\tensor{\varphi}{_0_0}-\rho^2\tensor{\nabla}{_0}\tensor{\varphi}{_\alpha^\alpha}$
	\\ \midrule
	$\tensor{D}{_\infty^L_K}\tensor{D}{_\alpha^K_L}$
	& $\tfrac{i}{2}(\euler+1)\tensor{\varphi}{_0_\beta}-\rho\tensor{\nabla}{_\alpha}\tensor{\varphi}{_0_0}
	-\rho\tensor{\nabla}{_\alpha}\tensor{\varphi}{_\beta^\beta}
	-\tfrac{1}{2}\rho(\euler-2)
	\tensor{N}{_\alpha^\conjbeta^\conjgamma}\tensor{\varphi}{_\conjbeta_\conjgamma}$ \\ \midrule
	$\tensor{D}{_0^L_K}\tensor{D}{_0^K_L}$
	& $-\tfrac{1}{2}(n+4)+(\euler-n-2)\tensor{\varphi}{_0_0}+\tensor{\varphi}{_\alpha^\alpha}
	-i\rho(\tensor{\nabla}{^\alpha}\tensor{\varphi}{_0_\alpha}
		-\tensor{\nabla}{^\conjalpha}\tensor{\varphi}{_0_\conjalpha})$ \\ \midrule
	$\tensor{D}{_0^L_K}\tensor{D}{_\alpha^K_L}$
	& $-\rho^3\tensor{N}{_\alpha^\conjbeta^\conjgamma}\tensor{A}{_\conjbeta_\conjgamma}
		+\tfrac{1}{4}(3\euler-2n-5)\tensor{\varphi}{_0_\alpha}
		+\tfrac{i}{2}(\tensor{\nabla}{^\conjbeta}\tensor{\varphi}{_\alpha_\conjbeta}
			-\tensor{\nabla}{^\beta}\tensor{\varphi}{_\alpha_\beta})
		-\tfrac{i}{2}\rho\tensor{N}{_\alpha^\conjbeta^\conjgamma}\tensor{\varphi}{_\conjbeta_\conjgamma}$ \\
	& $+\tfrac{i}{2}\rho^2\tensor{\nabla}{_0}\tensor{\varphi}{_0_\alpha}$ \\ \midrule
	$\tensor{D}{_\alpha^L_K}\tensor{D}{_\conjbeta^K_L}$
	& $\rho^2\tensor{N}{_\alpha^\conjgamma_\rho}\tensor{N}{_\conjbeta^\rho_\conjgamma}
	+\tfrac{1}{2}(\euler-2)\tensor{\varphi}{_\alpha_\conjbeta}
	+\tfrac{1}{2}\tensor{h}{_\alpha_\conjbeta}\tensor{\varphi}{_0_0}
	-\tfrac{i}{2}\rho(\tensor{\nabla}{_\alpha}\tensor{\varphi}{_0_\conjbeta}
		-\tensor{\nabla}{_\conjbeta}\tensor{\varphi}{_0_\alpha})$ \\ \midrule
	$\tensor{D}{_\alpha^L_K}\tensor{D}{_\beta^K_L}$
	& $\tfrac{1}{2}\euler\tensor{\varphi}{_\alpha_\beta}+i\rho^2\tensor{A}{_\alpha_\beta}
		+\tfrac{i}{2}\rho(\tensor{\nabla}{_\alpha}\tensor{\varphi}{_0_\beta}
		+\tensor{\nabla}{_\beta}\tensor{\varphi}{_0_\alpha})$ \\ \bottomrule
\end{tabular}
	\vspace{0.5em}
	\caption{$\tensor{D}{_I^L_K}\tensor{D}{_J^K_L}$ for a normal-form ACH metric $g$}
	\label{tbl:ProductOfConnectionForms1}
\end{table}

\begin{table}[htbp]
	\small
\begin{tabular}{ll}
	\toprule
	\multicolumn{1}{c}{Type} & \multicolumn{1}{c}{Value (modulo terms of type (N1)--(N3))} \\ \midrule
	$\tensor{D}{_\infty^L_\infty}\tensor{D}{_L^K_K}$
	& $2n+3-\tfrac{1}{2}\euler\tensor{\varphi}{_0_0}-\euler\tensor{\varphi}{_\alpha^\alpha}$ \\ \midrule
	$\tensor{D}{_\infty^L_0}\tensor{D}{_L^K_K}$
	& $-\rho^2\tensor{\nabla}{_0}\tensor{\varphi}{_0_0}-2\rho^2\tensor{\nabla}{_0}\tensor{\varphi}{_\alpha^\alpha}$
	\\ \midrule
	$\tensor{D}{_\infty^L_\alpha}\tensor{D}{_L^K_K}$
	& $-\tfrac{1}{2}\rho\tensor{\nabla}{_\alpha}\tensor{\varphi}{_0_0}
	-\rho\tensor{\nabla}{_\alpha}\tensor{\varphi}{_\beta^\beta}
	+\rho\tensor{N}{_\alpha^\conjbeta^\conjgamma}\tensor{\varphi}{_\conjbeta_\conjgamma}$ \\ \midrule
	$\tensor{D}{_0^L_0}\tensor{D}{_L^K_K}$
	& $-\tfrac{1}{2}(2n+3)+\tfrac{1}{4}\euler\tensor{\varphi}{_0_0}
	+\tfrac{1}{2}\euler\tensor{\varphi}{_\alpha^\alpha}
	+\tfrac{1}{8}(2n+3)(\euler-4)\tensor{\varphi}{_0_0}$ \\ \midrule
	$\tensor{D}{_0^L_\alpha}\tensor{D}{_L^K_K}$
	& $\tfrac{1}{8}(2n+3)(\euler-3)\tensor{\varphi}{_0_\alpha}
	+\tfrac{i}{4}\rho\tensor{\nabla}{_\alpha}\tensor{\varphi}{_0_0}
	+\tfrac{i}{2}\rho\tensor{\nabla}{_\alpha}\tensor{\varphi}{_\beta^\beta}
	-\tfrac{i}{2}\rho\tensor{N}{_\alpha^\conjbeta^\conjgamma}\tensor{\varphi}{_\conjbeta_\conjgamma}$ \\ \midrule
	$\tensor{D}{_\alpha^L_\conjbeta}\tensor{D}{_L^K_K}$
	& $-\tfrac{1}{4}(2n+3)\tensor{h}{_\alpha_\conjbeta}
		+\tfrac{1}{8}(2n+3)(\euler-2)\tensor{\varphi}{_\alpha_\conjbeta}
		+\tfrac{1}{8}\tensor{h}{_\alpha_\conjbeta}\euler\tensor{\varphi}{_0_0}
		+\tfrac{1}{4}\tensor{h}{_\alpha_\conjbeta}\euler\tensor{\varphi}{_\gamma^\gamma}$ \\
	& $+\tfrac{i}{4}\rho^2\tensor{h}{_\alpha_\conjbeta}\tensor{\nabla}{_0}\tensor{\varphi}{_0_0}
		+\tfrac{i}{2}\rho^2
		\tensor{h}{_\alpha_\conjbeta}\tensor{\nabla}{_0}\tensor{\varphi}{_\gamma^\gamma}$ \\ \midrule
	$\tensor{D}{_\alpha^L_\beta}\tensor{D}{_L^K_K}$
	& $\tfrac{1}{8}(2n+3)(\euler-2)\tensor{\varphi}{_\alpha_\beta}$ \\ \bottomrule
\end{tabular}
	\vspace{0.5em}
	\caption{$\tensor{D}{_I^L_J}\tensor{D}{_L^K_K}$ for a normal-form ACH metric $g$}
	\label{tbl:ProductOfConnectionForms2}
\end{table}

Since by definition $\tensor{\varphi}{_i_j}$ is $O(\rho)$, from Lemma \ref{lem:RicciTensorModuloHighOrderTerm}
we have
\begin{equation}
	\label{eq:FirstOrderTermOfACHMetric}
\begin{split}
	\tensor{\Ein}{_\infty_\infty}
	&=\tfrac{3}{2}\tensor{\varphi}{_0_0}+\tensor{\varphi}{_\alpha^\alpha}+O(\rho^2),\qquad
	\tensor{\Ein}{_\infty_0}=O(\rho^2),\qquad
	\tensor{\Ein}{_\infty_\alpha}=-i\tensor{\varphi}{_0_\alpha}+O(\rho^2),\\
	\tensor{\Ein}{_0_0}
	&=\tfrac{3}{8}(2n+1)\tensor{\varphi}{_0_0}-\tfrac{1}{2}\tensor{\varphi}{_\alpha^\alpha}+O(\rho^2),\qquad
	\tensor{\Ein}{_0_\alpha}=\tfrac{1}{2}(n+1)\tensor{\varphi}{_0_\alpha}+O(\rho^2),\\
	\tensor{\Ein}{_\alpha_\conjbeta}
	&=\tfrac{1}{8}(2n+9)\tensor{\varphi}{_\alpha_\conjbeta}
	-\tfrac{3}{8}\tensor{h}{_\alpha_\conjbeta}\tensor{\varphi}{_0_0}
	+\tfrac{1}{4}\tensor{h}{_\alpha_\conjbeta}\tensor{\varphi}{_\gamma^\gamma}+O(\rho^2),\\
	\tensor{\Ein}{_\alpha_\beta}&=\tfrac{1}{8}(2n+1)\tensor{\varphi}{_\alpha_\beta}+O(\rho^2).
\end{split}
\end{equation}
These identities show that all $\tensor{\varphi}{_i_j}$ must be $O(\rho^2)$ in order $\tensor{\Ein}{_I_J}$
to be $O(\rho^2)$.
If $\tensor{\varphi}{_i_j}=O(\rho^2)$, by repeating this process we obtain the following,
which immediately show Proposition \ref{prop:ACHEToThirdOrder}.
\begin{equation}
	\label{eq:SecondOrderTermOfACHMetric}
\begin{split}
	\tensor{\Ein}{_\infty_\infty}&=2\tensor{\varphi}{_0_0}+O(\rho^3),\qquad
	\tensor{\Ein}{_\infty_0}=O(\rho^3),\qquad
	\tensor{\Ein}{_\infty_\alpha}=-\tfrac{3}{2}i\tensor{\varphi}{_0_\alpha}+O(\rho^3),\\
	\tensor{\Ein}{_0_0}&=\tfrac{1}{2}(2n+1)\tensor{\varphi}{_0_0}+O(\rho^3),\qquad
	\tensor{\Ein}{_0_\alpha}=\tfrac{3}{8}(2n+1)\tensor{\varphi}{_0_\alpha}+O(\rho^3),\\
	\tensor{\Ein}{_\alpha_\conjbeta}&=\rho^2\tensor{R}{_\alpha_\conjbeta}
		-2\rho^2\tensor{N}{_\alpha^\conjgamma_\rho}\tensor{N}{_\conjbeta^\rho_\conjgamma}
		+\tfrac{1}{2}(n+2)\tensor{\varphi}{_\alpha_\conjbeta}
		-\tfrac{1}{4}\tensor{h}{_\alpha_\conjbeta}\tensor{\varphi}{_0_0}
		+\tfrac{1}{2}\tensor{h}{_\alpha_\conjbeta}\tensor{\varphi}{_\gamma^\gamma}+O(\rho^3),\\
	\tensor{\Ein}{_\alpha_\beta}&=in\rho^2\tensor{A}{_\alpha_\beta}
		+\rho^2(\tensor{N}{_\gamma_\alpha_\beta_,^\gamma}+\tensor{N}{_\gamma_\beta_\alpha_,^\gamma})
		+\tfrac{1}{2}n\tensor{\varphi}{_\alpha_\beta}+O(\rho^3).
\end{split}
\end{equation}

\section{Higher-order perturbation}
\label{sec:HigherOrderPerturbation}

Taking over the setting from the last section, we
next introduce a perturbation in $g$ and see what happens to the Einstein tensor.
Let $m\ge 1$ be a fixed integer
and $\tensor{\psi}{_i_j}$ a $2$-tensor on $M$ with coefficients in $C^\infty(X)$ such that
\begin{align*}
	\tensor{\psi}{_0_0}&=O(\rho^{m+2}),\qquad
	\tensor{\psi}{_0_\alpha}=O(\rho^{\max\set{m+1,3}}),\\
	\tensor{\psi}{_\alpha_\conjbeta}&=O(\rho^{m+2}),\qquad
	\tensor{\psi}{_\alpha_\beta}=O(\rho^{\max\set{m,3}}).
\end{align*}
Let $g$ be a normal-form ACH metric satisfying \eqref{eq:LowestTermsConditionOfACHMetric} and
consider another metric $g'$ with the following components with respect to
$\set{W_I}=\set{\euler,\rho^2T,\rho Z_\alpha,\rho Z_\conjalpha}$:
\begin{equation}
	\tensor{{g'}}{_i_j}=\tensor{g}{_i_j}+\tensor{\psi}{_i_j}.
	\label{eq:PerturbedMetric}
\end{equation}
Note that $g'$ also satisfies \eqref{eq:LowestTermsConditionOfACHMetric}.
We can read off from Lemma \ref{lem:RicciTensorModuloHighOrderTerm}
the amount to which the Einstein tensor changes, which is denoted by $\tensor{\delta\Ein}{_I_J}$.
For example we have
\begin{equation}
	\label{eq:EinPerturbationLowerOrder}
	\begin{split}
		\tensor{\delta\Ein}{_\infty_\alpha}
		&=-\tfrac{1}{2}i(\euler+1)\tensor{\psi}{_0_\alpha}
			+\tfrac{1}{2}\rho^2\partial_\rho\tensor{\nabla}{^\beta}\tensor{\psi}{_\alpha_\beta}
			+\tfrac{1}{2}\rho
				\tensor{N}{_\alpha^\conjbeta^\conjgamma}\euler\tensor{\psi}{_\conjbeta_\conjgamma}
			+O(\rho^{m+2}),\\
		\tensor{\delta\Ein}{_0_\alpha}
		&=-\tfrac{1}{8}\left( (\euler)^2-(2n+2)\euler-2n-3 \right)\tensor{\psi}{_0_\alpha}+O(\rho^{m+2}),\\
		\tensor{\delta\Ein}{_\alpha_\beta}
		&=-\tfrac{1}{8}\euler(\euler-2n-2)\tensor{\psi}{_\alpha_\beta}+O(\rho^{m+1}).
\end{split}
\end{equation}
In the same way we can compute $\tensor{\delta\Ein}{_\infty_\infty}$, $\tensor{\delta\Ein}{_\infty_0}$,
$\tensor{\delta\Ein}{_0_0}$ and $\tensor{\delta\Ein}{_\alpha_\conjbeta}$ modulo $O(\rho^{m+2})$.
But we want them to be given modulo one order higher.
In this section we shall prove the following.

\begin{prop}
	\label{prop:HigherOrderPerturbationAndEinsteinTensor}
	The components $\tensor{\delta\Ein}{_\infty_\infty}$, $\tensor{\delta\Ein}{_\infty_0}$,
	$\tensor{\delta\Ein}{_0_0}$, $\tensor{\delta\Ein}{_\alpha_\conjbeta}$ of the difference
	$\delta\Ein=\Ein'-\Ein$ between the Einstein tensors of $g$ and $g'$ are given by,
	modulo $O(\rho^{m+3})$,
	\allowdisplaybreaks
	\begin{subequations}
		\label{eq:EinPerturbationHigherOrder}
	\begin{align}
		\label{eq:EinPerturbationHigherOrderInftyInfty}
		\begin{split}
			\tensor{\delta\Ein}{_\infty_\infty}
			&\equiv -\tfrac{1}{2}\euler(\euler-4)\tensor{\psi}{_0_0}
			-\euler(\euler-2)\tensor{\psi}{_\alpha^\alpha}\\
			&\phantom{\equiv \;}
				+\tfrac{1}{2}\rho^2(\euler)^2(\tensor{\Phi}{^\alpha^\beta}\tensor{\psi}{_\alpha_\beta}
					+\tensor{\Phi}{^\conjalpha^\conjbeta}\tensor{\psi}{_\conjalpha_\conjbeta}),
		\end{split}\\
		\label{eq:EinPerturbationHigherOrderInftyZero}
		\tensor{\delta\Ein}{_\infty_0}
		&\equiv \tfrac{1}{2}\rho(\euler+1)(\tensor{\nabla}{^\alpha}\tensor{\psi}{_0_\alpha}
			+\tensor{\nabla}{^\conjalpha}\tensor{\psi}{_0_\conjalpha})
			-\tfrac{1}{2}\rho^3\partial_\rho(\tensor{A}{^\alpha^\beta}\tensor{\psi}{_\alpha_\beta}
				+\tensor{A}{^\conjalpha^\conjbeta}\tensor{\psi}{_\conjalpha_\conjbeta}),\\
		\label{eq:EinPerturbationHigherOrderZeroZero}
		\begin{split}
			\tensor{\delta\Ein}{_0_0}
			&\equiv -\tfrac{1}{8}\left( (\euler)^2-(2n+4)\euler-4n \right)\tensor{\psi}{_0_0}
				+\tfrac{1}{2}(\euler-2)\tensor{\psi}{_\alpha^\alpha}\\
				&\phantom{\equiv \;}
				+i\rho(\tensor{\nabla}{^\alpha}\tensor{\psi}{_0_\alpha}
					-\tensor{\nabla}{^\conjalpha}\tensor{\psi}{_0_\conjalpha})
				-\tfrac{1}{4}\rho^3\partial_\rho(\tensor{\Phi}{^\alpha^\beta}\tensor{\psi}{_\alpha_\beta}
					+\tensor{\Phi}{^\conjalpha^\conjbeta}\tensor{\psi}{_\conjalpha_\conjbeta}),
		\end{split}\\
		\label{eq:EinPerturbationHigherOrderTrace}
		\begin{split}
			\tensor{\delta\Ein}{_\alpha^\alpha}
			&\equiv \tfrac{1}{8}n(\euler-4)\tensor{\psi}{_0_0}
			-\tfrac{1}{8}\left( (\euler)^2-(4n+2)\euler-8 \right)\tensor{\psi}{_\alpha^\alpha}\\
			&\phantom{\equiv \;}
			-i\rho(\tensor{\nabla}{^\alpha}\tensor{\psi}{_0_\alpha}
			-\tensor{\nabla}{^\conjalpha}\tensor{\psi}{_0_\conjalpha})\\
			&\phantom{\equiv \;}
			-\tfrac{1}{8}\rho^2\left( (n-2)\euler+(2n+4) \right)
			(\tensor{\Phi}{^\alpha^\beta}\tensor{\psi}{_\alpha_\beta}
			+\tensor{\Phi}{^\conjalpha^\conjbeta}\tensor{\psi}{_\conjalpha_\conjbeta})\\
			&\phantom{\equiv \;}
			+\tfrac{1}{2}\rho^2(\tensor{\nabla}{^\alpha}\tensor{\nabla}{^\beta}\tensor{\psi}{_\alpha_\beta}
			+\tensor{\nabla}{^\conjalpha}\tensor{\nabla}{^\conjbeta}\tensor{\psi}{_\conjalpha_\conjbeta})
			+\tfrac{1}{2}\rho^2(\tensor{N}{^\gamma^\alpha^\beta_,_\gamma}\tensor{\psi}{_\alpha_\beta}
			+\tensor{N}{^\conjgamma^\conjalpha^\conjbeta_,_\conjgamma}\tensor{\psi}{_\conjalpha_\conjbeta})\\
			&\phantom{\equiv \;}
			+\tfrac{1}{2}\rho^2
			(\tensor{N}{^\gamma^\alpha^\beta}\tensor{\nabla}{_\gamma}\tensor{\psi}{_\alpha_\beta}
			+\tensor{N}{^\conjgamma^\conjalpha^\conjbeta}\tensor{\nabla}{_\conjgamma}
			\tensor{\psi}{_\conjalpha_\conjbeta}),
		\end{split}\\
		\label{eq:EinPerturbationHigherOrderHermitianTraceFree}
		\begin{split}
			\tf(\tensor{\delta\Ein}{_\alpha_\conjbeta})
			&\equiv -\tfrac{1}{8}\left( (\euler)^2-(2n+2)\euler-8 \right)\tf(\tensor{\psi}{_\alpha_\conjbeta})\\
			&\phantom{\equiv \;}
			+i\rho\tf(\tensor{\nabla}{_\alpha}\tensor{\psi}{_0_\conjbeta}
			-\tensor{\nabla}{_\conjbeta}\tensor{\psi}{_0_\alpha})+\rho^2\tf(\tensor{\Psi}{_\alpha_\conjbeta}),
		\end{split}
	\end{align}
	\end{subequations}
	\allowdisplaybreaks[0]%
	where $\tensor{\delta\Ein}{_\alpha^\alpha}$ is the trace of $\tensor{\delta\Ein}{_\alpha_\conjbeta}$
	with respect to $\tensor{h}{_\alpha_\conjbeta}$, $\tf$ denotes the trace-free part, and
	\begin{equation*}
		\begin{split}
			\tensor{\Psi}{_\alpha_\conjbeta}
			&=\tfrac{1}{4}(\euler-2)(\tensor{\Phi}{_\alpha^\conjgamma}\tensor{\psi}{_\conjbeta_\conjgamma}
			+\tensor{\Phi}{_\conjbeta^\gamma}\tensor{\psi}{_\alpha_\gamma})
			+\tfrac{1}{2}(\tensor{\nabla}{^\conjgamma}\tensor{\nabla}{_\alpha}\tensor{\psi}{_\conjbeta_\conjgamma}
			+\tensor{\nabla}{^\gamma}\tensor{\nabla}{_\conjbeta}\tensor{\psi}{_\alpha_\gamma})\\
			&\phantom{=\;}
			-\tensor{N}{_\alpha^\conjgamma^\conjsigma_,_\conjgamma}\tensor{\psi}{_\conjbeta_\conjsigma}
			-\tensor{N}{_\conjbeta^\gamma^\sigma_,_\gamma}\tensor{\psi}{_\alpha_\sigma}
			+\tensor{N}{_\alpha^\conjgamma^\conjsigma}
			(\tensor{\nabla}{_\conjbeta}\tensor{\psi}{_\conjgamma_\conjsigma}
			-\tensor{\nabla}{_\conjsigma}\tensor{\psi}{_\conjbeta_\conjgamma})
			+\tensor{N}{_\conjbeta^\gamma^\sigma}
			(\tensor{\nabla}{_\alpha}\tensor{\psi}{_\gamma_\sigma}
			-\tensor{\nabla}{_\sigma}\tensor{\psi}{_\alpha_\gamma}).
		\end{split}
	\end{equation*}
\end{prop}

First, let
\begin{equation*}
	\nabla^{g'}_{W_J}W_I=\overline{\nabla}_{W_J}W_I+\tensor{ {D'} }{_I^K_J}W_K
\end{equation*}
and $\tensor{ {D'} }{_I_K_J}:=\tensor{ {g'} }{_K_L}\tensor{ {D'} }{_I^L_J}$.
Then $\tensor{\delta D}{_I_K_J}=\tensor{ {D'} }{_I_K_J}-\tensor{D}{_I_K_J}$ is given in Table
\ref{tbl:PerturbationOfConnectionFormsWithLowerIndices}, which is read off immediately from
Table \ref{tbl:ConnectionFormsOfACHMetricWithLowerIndices}.

\begin{table}[htbp]
	\small
\begin{tabular}{ll}
	\toprule
	\multicolumn{1}{c}{Type} & \multicolumn{1}{c}{Value (modulo $O(\rho^{m+3})$)} \\ \midrule
	$\tensor{\delta D}{_\infty_\infty_\infty}$ & $0$ \\ \midrule
	$\tensor{\delta D}{_\infty_0_\infty}$ & $0$ \\ \midrule
	$\tensor{\delta D}{_\infty_\alpha_\infty}$ & $0$ \\ \midrule
	$\tensor{\delta D}{_\infty_\infty_0}$ & $0$ \\ \midrule
	$\tensor{\delta D}{_\infty_0_0}$ & $\tfrac{1}{2}(\euler-4)\tensor{\psi}{_0_0}$ \\ \midrule
	$\tensor{\delta D}{_\infty_\alpha_0}$ & $\tfrac{1}{2}(\euler-3)\tensor{\psi}{_0_\alpha}$ \\ \midrule
	$\tensor{\delta D}{_\infty_\infty_\alpha}$ & $0$ \\ \midrule
	$\tensor{\delta D}{_\infty_0_\alpha}$ & $\tfrac{1}{2}(\euler-3)\tensor{\psi}{_0_\alpha}$ \\ \midrule
	$\tensor{\delta D}{_\infty_\conjbeta_\alpha}$
	& $\tfrac{1}{2}(\euler-2)\tensor{\psi}{_\alpha_\conjbeta}$ \\ \midrule
	$\tensor{\delta D}{_\infty_\beta_\alpha}$ & $\tfrac{1}{2}(\euler-2)\tensor{\psi}{_\alpha_\beta}$ \\ \midrule
	$\tensor{\delta D}{_0_\infty_0}$ & $-\tfrac{1}{2}(\euler-4)\tensor{\psi}{_0_0}$ \\ \midrule
	$\tensor{\delta D}{_0_0_0}$ & $0$ \\ \midrule
	$\tensor{\delta D}{_0_\alpha_0}$ & $0$ \\ \midrule
	$\tensor{\delta D}{_0_\infty_\alpha}$ & $-\tfrac{1}{2}(\euler-3)\tensor{\psi}{_0_\alpha}$ \\ \midrule
	$\tensor{\delta D}{_0_0_\alpha}$ & $0$ \\ \midrule
	$\tensor{\delta D}{_0_\conjbeta_\alpha}$
	& $\tfrac{i}{2}\tensor{h}{_\alpha_\conjbeta}\tensor{\psi}{_0_0}
		+\tfrac{1}{2}\rho
			(\tensor{\nabla}{_\alpha}\tensor{\psi}{_0_\conjbeta}
			-\tensor{\nabla}{_\conjbeta}\tensor{\psi}{_0_\alpha})
		+\tfrac{1}{2}\rho^2
			(\tensor{A}{_\alpha^\conjgamma}\tensor{\psi}{_\conjbeta_\conjgamma}
			+\tensor{A}{_\conjbeta^\gamma}\tensor{\psi}{_\alpha_\gamma})$ \\ \midrule
	$\tensor{\delta D}{_0_\beta_\alpha}$
	& $\tfrac{1}{2}\rho
			(\tensor{\nabla}{_\alpha}\tensor{\psi}{_0_\beta}-\tensor{\nabla}{_\beta}\tensor{\psi}{_0_\alpha}
			-\tensor{N}{_\alpha_\beta^\conjgamma}\tensor{\psi}{_0_\conjgamma})
		+\tfrac{1}{2}\rho^2\tensor{\nabla}{_0}\tensor{\psi}{_\alpha_\beta}$ \\ \midrule
	$\tensor{\delta D}{_\alpha_\infty_\conjbeta}$
	& $-\tfrac{1}{2}(\euler-2)\tensor{\psi}{_\alpha_\conjbeta}$ \\ \midrule
	$\tensor{\delta D}{_\alpha_0_\conjbeta}$
	& $\tfrac{i}{2}\tensor{h}{_\alpha_\conjbeta}\tensor{\psi}{_0_0}
		+\tfrac{1}{2}\rho
			(\tensor{\nabla}{_\alpha}\tensor{\psi}{_0_\conjbeta}
			+\tensor{\nabla}{_\conjbeta}\tensor{\psi}{_0_\alpha})
		-\tfrac{1}{2}\rho^2
			(\tensor{A}{_\alpha^\conjgamma}\tensor{\psi}{_\conjbeta_\conjgamma}
			+\tensor{A}{_\conjbeta^\gamma}\tensor{\psi}{_\alpha_\gamma})$ \\ \midrule
	$\tensor{\delta D}{_\alpha_\conjgamma_\conjbeta}$
	& $\tfrac{i}{2}\tensor{h}{_\alpha_\conjbeta}\tensor{\psi}{_0_\conjgamma}
		+\tfrac{i}{2}\tensor{h}{_\alpha_\conjgamma}\tensor{\psi}{_0_\conjbeta}
		+\tfrac{1}{2}\rho(\tensor{\nabla}{_\alpha}\tensor{\psi}{_\conjbeta_\conjgamma}
			-\tensor{N}{_\conjbeta_\conjgamma^\sigma}\tensor{\psi}{_\alpha_\sigma})$ \\ \midrule
	$\tensor{\delta D}{_\alpha_\gamma_\conjbeta}$
	& $\tfrac{i}{2}\tensor{h}{_\alpha_\conjbeta}\tensor{\psi}{_0_\gamma}
		-\tfrac{i}{2}\tensor{h}{_\gamma_\conjbeta}\tensor{\psi}{_0_\alpha}
		+\tfrac{1}{2}\rho(\tensor{\nabla}{_\conjbeta}\tensor{\psi}{_\alpha_\gamma}
			-\tensor{N}{_\alpha_\gamma^\conjsigma}\tensor{\psi}{_\conjbeta_\conjsigma})$ \\ \midrule
	$\tensor{\delta D}{_\alpha_\infty_\beta}$ & $-\tfrac{1}{2}(\euler-2)\tensor{\psi}{_\alpha_\beta}$ \\ \midrule
	$\tensor{\delta D}{_\alpha_0_\beta}$
	& $\tfrac{1}{2}\rho
			(\tensor{\nabla}{_\alpha}\tensor{\psi}{_0_\beta}+\tensor{\nabla}{_\beta}\tensor{\psi}{_0_\alpha}
			-\tensor{N}{_\alpha_\beta^\conjgamma}\tensor{\psi}{_0_\conjgamma})
		-\tfrac{1}{2}\rho^2\tensor{\nabla}{_0}\tensor{\psi}{_\alpha_\beta}$ \\ \midrule
	$\tensor{\delta D}{_\alpha_\conjgamma_\beta}$
	& $\tfrac{i}{2}\tensor{h}{_\alpha_\conjgamma}\tensor{\psi}{_0_\beta}
		+\tfrac{i}{2}\tensor{h}{_\beta_\conjgamma}\tensor{\psi}{_0_\alpha}
		-\tfrac{1}{2}\rho(\tensor{\nabla}{_\conjgamma}\tensor{\psi}{_\alpha_\beta}
			+\tensor{N}{_\alpha_\beta^\conjsigma}\tensor{\psi}{_\conjgamma_\conjsigma})$ \\ \midrule
	$\tensor{\delta D}{_\alpha_\gamma_\beta}$
	& $\tfrac{1}{2}\rho(\tensor{\nabla}{_\alpha}\tensor{\psi}{_\beta_\gamma}
			+\tensor{\nabla}{_\beta}\tensor{\psi}{_\alpha_\gamma}
			-\tensor{\nabla}{_\gamma}\tensor{\psi}{_\alpha_\beta})$ \\ \bottomrule
\end{tabular}
	\vspace{0.5em}
	\caption{$\tensor{\delta D}{_I_K_J}=\tensor{ {D'} }{_I_K_J}-\tensor{D}{_I_K_J}$
	for a perturbation \eqref{eq:PerturbedMetric}}
	\label{tbl:PerturbationOfConnectionFormsWithLowerIndices}
\end{table}

Next we compute $\tensor{\delta D}{_I^K_J}:=\tensor{ {D'} }{_I^K_J}-\tensor{D}{_I^K_J}$.
To do this we need the knowledge of the following quantities:
$\tensor{D}{_I_K_J}$ modulo $O(\rho^3)$,
$\tensor{g}{^I^J}$ modulo $O(\rho^3)$ and
$\tensor{\delta g}{^I^J}:=\tensor{ {g'} }{^I^J}-\tensor{g}{^I^J}$ modulo $O(\rho^{m+3})$.
They can be read off from Table \ref{tbl:ConnectionFormsOfACHMetricWithLowerIndices},
\eqref{eq:LowestTermsConditionOfACHMetric} and \eqref{eq:InverseOfNormalACH}.
Namely, $\tensor{D}{_I_K_J}$ mod $O(\rho^3)$ are given by
\allowdisplaybreaks
\begin{alignat*}{4}
	\tensor{D}{_\infty_\infty_\infty}&\equiv -4,&\qquad
	\tensor{D}{_\infty_0_\infty}&\equiv 0,&\qquad
	\tensor{D}{_\infty_\alpha_\infty}&\equiv 0,\\
	\tensor{D}{_\infty_\infty_0}&\equiv 0,&\qquad
	\tensor{D}{_\infty_0_0}&\equiv -2,&\qquad
	\tensor{D}{_\infty_\alpha_0}&\equiv 0,\\
	\tensor{D}{_\infty_\infty_\alpha}&\equiv 0,&\qquad
	\tensor{D}{_\infty_0_\alpha}&\equiv 0,&\qquad
	\tensor{D}{_\infty_\conjbeta_\alpha}&\equiv -\tensor{h}{_\alpha_\conjbeta},&\qquad
	\tensor{D}{_\infty_\beta_\alpha}&\equiv 0,\\
	\tensor{D}{_0_\infty_0}&\equiv 2,&\qquad
	\tensor{D}{_0_0_0}&\equiv 0,&\qquad
	\tensor{D}{_0_\alpha_0}&\equiv 0,\\
	\tensor{D}{_0_\infty_\alpha}&\equiv 0,&\qquad
	\tensor{D}{_0_0_\alpha}&\equiv 0,&\qquad
	\tensor{D}{_0_\conjbeta_\alpha}&\equiv \tfrac{1}{2}i\tensor{h}{_\alpha_\conjbeta},&\qquad
	\tensor{D}{_0_\beta_\alpha}&\equiv \rho^2\tensor{A}{_\alpha_\beta},\\
	\tensor{D}{_\alpha_\infty_0}&\equiv 0,&\qquad
	\tensor{D}{_\alpha_0_0}&\equiv 0,&\qquad
	\tensor{D}{_\alpha_\conjbeta_0}&\equiv \tfrac{1}{2}i\tensor{h}{_\alpha_\conjbeta},&\qquad
	\tensor{D}{_\alpha_\beta_0}&\equiv 0,\\
	\tensor{D}{_\alpha_\infty_\conjbeta}&\equiv \tensor{h}{_\alpha_\conjbeta},&\qquad
	\tensor{D}{_\alpha_0_\conjbeta}&\equiv \tfrac{1}{2}i\tensor{h}{_\alpha_\conjbeta},&\qquad
	\tensor{D}{_\alpha_\conjgamma_\conjbeta}&\equiv 0,&\qquad
	\tensor{D}{_\alpha_\gamma_\conjbeta}&\equiv 0,\\
	\tensor{D}{_\alpha_\infty_\beta}&\equiv 0,&\qquad
	\tensor{D}{_\alpha_0_\beta}&\equiv -\rho^2\tensor{A}{_\alpha_\beta},&\qquad
	\tensor{D}{_\alpha_\conjgamma_\beta}&\equiv 0,&\qquad
	\tensor{D}{_\alpha_\gamma_\beta}&\equiv -\rho\tensor{N}{_\alpha_\gamma_\beta};
\end{alignat*}
\allowdisplaybreaks[0]%
$\tensor{g}{^I^J}$ mod $O(\rho^3)$ are
\begin{align*}
	\tensor{g}{^\infty^\infty}&\equiv\tfrac{1}{4},\qquad
	\tensor{g}{^\infty^0}\equiv\tensor{g}{^\infty^\alpha}\equiv 0,\qquad
	\tensor{g}{^0^0}\equiv 1,\qquad
	\tensor{g}{^0^\alpha}\equiv 0,\\
	\tensor{g}{^\alpha^\conjbeta}&\equiv\tensor{h}{^\alpha^\conjbeta}-\rho^2\tensor{\Phi}{^\alpha^\conjbeta},\qquad
	\tensor{g}{^\alpha^\beta}\equiv -\rho^2\tensor{\Phi}{^\alpha^\beta};
\end{align*}
$\delta\tensor{g}{^I^J}$ mod $O(\rho^{m+3})$ are
\begin{equation}
	\label{eq:HighOrderPerturbationOfInverseOfACH}
	\begin{split}
		\tensor{\delta g}{^\infty^\infty}&\equiv\tensor{\delta g}{^\infty^0}
		\equiv\tensor{g}{^\infty^\alpha}\equiv 0,\qquad
		\tensor{\delta g}{^0^0}\equiv -\tensor{\psi}{_0_0},\qquad
		\tensor{\delta g}{^0^\alpha}\equiv -\tensor{\psi}{_0^\alpha},\\
		\tensor{\delta g}{^\alpha^\conjbeta}
		&\equiv -\tensor{\psi}{^\alpha^\conjbeta}
		+\rho^2(\tensor{\Phi}{^\alpha_\conjgamma}\tensor{\psi}{^\conjbeta^\conjgamma}
			+\tensor{\Phi}{^\conjbeta_\gamma}\tensor{\psi}{^\alpha^\gamma}),\\
		\tensor{\delta g}{^\alpha^\beta}&\equiv -\tensor{\psi}{^\alpha^\beta}
		+\rho^2(\tensor{\Phi}{^\alpha_\gamma}\tensor{\psi}{^\beta^\gamma}
			+\tensor{\Phi}{^\beta_\gamma}\tensor{\psi}{^\alpha^\gamma}).
	\end{split}
\end{equation}
Since Table \ref{tbl:PerturbationOfConnectionFormsWithLowerIndices} and
\eqref{eq:HighOrderPerturbationOfInverseOfACH} shows that $\delta\tensor{g}{^I^J}$ and $\delta\tensor{D}{_I_K_J}$
are both $O(\rho^{\max\set{m,3}})$,
we have $\delta\tensor{D}{^K^L}\cdot\delta\tensor{D}{_I_L_J}=O(\rho^{m+3})$ and hence
\begin{equation*}
	\delta\tensor{D}{_I^K_J}
	\equiv\tensor{g}{^K^L}\cdot\delta\tensor{D}{_I_L_J}+\delta\tensor{g}{^K^L}\cdot\tensor{D}{_I_L_J}
	\mod O(\rho^{m+3}),
\end{equation*}
where $\tensor{\delta D}{_I^K_J}:=\tensor{ {D'} }{_I^K_J}-\tensor{D}{_I^K_J}$.
Thus we obtain Table \ref{tbl:PerturbationOfConnectionForms}.

\begin{table}[htbp]
	\small
\begin{tabular}{ll}
	\toprule
	\multicolumn{1}{c}{Type} & \multicolumn{1}{c}{Value (modulo $O(\rho^{m+3})$)} \\ \midrule
	$\tensor{\delta D}{_\infty^\infty_\infty}$ & $0$ \\ \midrule
	$\tensor{\delta D}{_\infty^0_\infty}$ & $0$ \\ \midrule
	$\tensor{\delta D}{_\infty^\alpha_\infty}$ & $0$ \\ \midrule
	$\tensor{\delta D}{_\infty^\infty_0}$ & $0$ \\ \midrule
	$\tensor{\delta D}{_\infty^0_0}$ & $\tfrac{1}{2}\euler\tensor{\psi}{_0_0}$ \\ \midrule
	$\tensor{\delta D}{_\infty^\alpha_0}$ & $\tfrac{1}{2}(\euler+1)\tensor{\psi}{_0^\alpha}$ \\ \midrule
	$\tensor{\delta D}{_\infty^\infty_\alpha}$ & $0$ \\ \midrule
	$\tensor{\delta D}{_\infty^0_\alpha}$ & $\tfrac{1}{2}(\euler-1)\tensor{\psi}{_0_\alpha}$ \\ \midrule
	$\tensor{\delta D}{_\infty^\beta_\alpha}$
	& $\tfrac{1}{2}\euler\tensor{\psi}{_\alpha^\beta}
		-\tfrac{1}{2}\rho^3\partial_\rho\tensor{\Phi}{^\beta^\gamma}\tensor{\psi}{_\alpha_\gamma}
		-\rho^2\tensor{\Phi}{_\alpha_\gamma}\tensor{\psi}{^\beta^\gamma}$ \\ \midrule
	$\tensor{\delta D}{_\infty^\conjbeta_\alpha}$
	& $\tfrac{1}{2}\euler\tensor{\psi}{_\alpha^\conjbeta}
	-\tfrac{1}{2}\rho^3\partial_\rho\tensor{\Phi}{^\gamma^\conjbeta}\tensor{\psi}{_\alpha_\gamma}
	-\rho^2\tensor{\Phi}{_\alpha_\conjgamma}\tensor{\psi}{^\conjbeta^\conjgamma}$ \\ \midrule
	$\tensor{\delta D}{_0^\infty_0}$ & $-\tfrac{1}{8}(\euler-4)\tensor{\psi}{_0_0}$ \\ \midrule
	$\tensor{\delta D}{_0^0_0}$ & $0$ \\ \midrule
	$\tensor{\delta D}{_0^\alpha_0}$ & $0$ \\ \midrule
	$\tensor{\delta D}{_0^\infty_\alpha}$ & $-\tfrac{1}{8}(\euler-3)\tensor{\psi}{_0_\alpha}$ \\ \midrule
	$\tensor{\delta D}{_0^0_\alpha}$ & $-\tfrac{i}{2}\tensor{\psi}{_0_\alpha}$ \\ \midrule
	$\tensor{\delta D}{_0^\beta_\alpha}$
	& $\tfrac{i}{2}\tensor{\delta}{_\alpha^\beta}\tensor{\psi}{_0_0}
		-\tfrac{i}{2}\tensor{\psi}{_\alpha^\beta}
		+\tfrac{1}{2}\rho
			(\tensor{\nabla}{_\alpha}\tensor{\psi}{_0^\beta}-\tensor{\nabla}{^\beta}\tensor{\psi}{_0_\alpha})$\\
	& $-\tfrac{1}{2}\rho^2
			(\tensor{A}{_\alpha_\gamma}\tensor{\psi}{^\beta^\gamma}
			-\tensor{A}{^\beta^\gamma}\tensor{\psi}{_\alpha_\gamma})
		+\tfrac{i}{2}\rho^2
			(\tensor{\Phi}{^\beta^\gamma}\tensor{\psi}{_\alpha_\gamma}
			+\tensor{\Phi}{_\alpha_\gamma}\tensor{\psi}{^\beta^\gamma})$ \\ \midrule
	$\tensor{\delta D}{_0^\conjbeta_\alpha}$
	& $-\tfrac{i}{2}\tensor{\psi}{_\alpha^\conjbeta}
		+\tfrac{1}{2}\rho
			(\tensor{\nabla}{_\alpha}\tensor{\psi}{_0^\conjbeta}
			-\tensor{\nabla}{^\conjbeta}\tensor{\psi}{_0_\alpha}
			-\tensor{N}{_\alpha^\conjbeta^\conjgamma}\tensor{\psi}{_0_\conjgamma})
		+\tfrac{1}{2}\rho^2\tensor{\nabla}{_0}\tensor{\psi}{_\alpha^\conjbeta}
		+\tfrac{i}{2}\rho^2(\tensor{\Phi}{^\gamma^\conjbeta}\tensor{\psi}{_\alpha_\gamma}
			+\tensor{\Phi}{_\alpha_\conjgamma}\tensor{\psi}{^\conjbeta^\conjgamma})$ \\ \midrule
	$\tensor{\delta D}{_\alpha^\infty_\conjbeta}$
	& $-\tfrac{1}{8}(\euler-2)\tensor{\psi}{_\alpha_\conjbeta}$ \\ \midrule
	$\tensor{\delta D}{_\alpha^0_\conjbeta}$
	& $\tfrac{1}{2}\rho
			(\tensor{\nabla}{_\alpha}\tensor{\psi}{_0_\conjbeta}
			+\tensor{\nabla}{_\conjbeta}\tensor{\psi}{_0_\alpha})
		-\tfrac{1}{2}\rho^2
			(\tensor{A}{_\alpha^\conjgamma}\tensor{\psi}{_\conjbeta_\conjgamma}
			+\tensor{A}{_\conjbeta^\gamma}\tensor{\psi}{_\alpha_\gamma})$ \\ \midrule
	$\tensor{\delta D}{_\alpha^\gamma_\conjbeta}$
	& $\tfrac{i}{2}\tensor{\delta}{_\alpha^\gamma}\tensor{\psi}{_0_\conjbeta}
		+\tfrac{1}{2}\rho(\tensor{\nabla}{_\alpha}\tensor{\psi}{_\conjbeta^\gamma}
			-\tensor{N}{_\conjbeta^\gamma^\sigma}\tensor{\psi}{_\alpha_\sigma})$ \\ \midrule
	$\tensor{\delta D}{_\alpha^\conjgamma_\conjbeta}$
	& $-\tfrac{i}{2}\tensor{\delta}{_\conjbeta^\conjgamma}\tensor{\psi}{_0_\alpha}
		+\tfrac{1}{2}\rho(\tensor{\nabla}{_\conjbeta}\tensor{\psi}{_\alpha^\conjgamma}
			-\tensor{N}{_\alpha^\conjgamma^\conjsigma}\tensor{\psi}{_\conjbeta_\conjsigma})$ \\ \midrule
	$\tensor{\delta D}{_\alpha^\infty_\beta}$
	& $-\tfrac{1}{8}(\euler-2)\tensor{\psi}{_\alpha_\beta}$ \\ \midrule
	$\tensor{\delta D}{_\alpha^0_\beta}$
	& $\tfrac{1}{2}\rho
			(\tensor{\nabla}{_\alpha}\tensor{\psi}{_0_\beta}+\tensor{\nabla}{_\beta}\tensor{\psi}{_0_\alpha})
		-\tfrac{1}{2}\rho(\tensor{N}{^\conjgamma_\alpha_\beta}+\tensor{N}{^\conjgamma_\beta_\alpha})
				\tensor{\psi}{_0_\conjgamma}
		-\tfrac{1}{2}\rho^2\tensor{\nabla}{_0}\tensor{\psi}{_\alpha_\beta}$ \\ \midrule
	$\tensor{\delta D}{_\alpha^\gamma_\beta}$
	& $\tfrac{i}{2}\tensor{\delta}{_\alpha^\gamma}\tensor{\psi}{_0_\beta}
		+\tfrac{i}{2}\tensor{\delta}{_\beta^\gamma}\tensor{\psi}{_0_\alpha}
		-\tfrac{1}{2}\rho\tensor{\nabla}{^\gamma}\tensor{\psi}{_\alpha_\beta}
		-\tfrac{1}{2}\rho(\tensor{N}{^\conjsigma_\alpha_\beta}+\tensor{N}{^\conjsigma_\beta_\alpha})
			\tensor{\psi}{_\conjsigma^\gamma}$ \\ \midrule
	$\tensor{\delta D}{_\alpha^\conjgamma_\beta}$
	& $\tfrac{1}{2}\rho(\tensor{\nabla}{_\alpha}\tensor{\psi}{_\beta^\conjgamma}
			+\tensor{\nabla}{_\beta}\tensor{\psi}{_\alpha^\conjgamma}
			-\tensor{\nabla}{^\conjgamma}\tensor{\psi}{_\alpha_\beta})$ \\ \bottomrule
\end{tabular}
	\vspace{0.5em}
	\caption{$\tensor{\delta D}{_I^K_J}=\tensor{ {D'} }{_I^K_J}-\tensor{D}{_I^K_J}$
	for a perturbation \eqref{eq:PerturbedMetric}}
	\label{tbl:PerturbationOfConnectionForms}
\end{table}

On the other hand, Table \ref{tbl:ConnectionFormsOfACHMetricWithUpperIndices} shows that,
modulo $O(\rho^3)$,
\begin{equation}
	\label{eq:ACHConnectionFormsToSecondOrder}
\begin{split}
	\tensor{D}{_\infty^\infty_\infty}&\equiv -1,\quad
	\tensor{D}{_\infty^0_\infty}\equiv 0,\quad
	\tensor{D}{_\infty^\alpha_\infty}\equiv 0,\\
	\tensor{D}{_\infty^\infty_0}&\equiv 0,\quad
	\tensor{D}{_\infty^0_0}\equiv -2,\quad
	\tensor{D}{_\infty^\alpha_0}\equiv 0,\\
	\tensor{D}{_\infty^\infty_\alpha}&\equiv 0,\quad
	\tensor{D}{_\infty^0_\alpha}\equiv 0,\quad
	\tensor{D}{_\infty^\beta_\alpha}\equiv -\tensor{\delta}{_\alpha^\beta}+\rho^2\tensor{\Phi}{_\alpha^\beta},\quad
	\tensor{D}{_\infty^\conjbeta_\alpha}\equiv \rho^2\tensor{\Phi}{_\alpha^\conjbeta},\\
	\tensor{D}{_0^\infty_0}&\equiv \tfrac{1}{2},\quad
	\tensor{D}{_0^0_0}\equiv 0,\quad
	\tensor{D}{_0^\alpha_0}\equiv 0,\\
	\tensor{D}{_0^\infty_\alpha}&\equiv 0,\quad
	\tensor{D}{_0^0_\alpha}\equiv 0,\quad
	\tensor{D}{_0^\beta_\alpha}
		\equiv \tfrac{1}{2}i\tensor{\delta}{_\alpha^\beta}-\tfrac{1}{2}i\rho^2\tensor{\Phi}{_\alpha^\beta},\quad
	\tensor{D}{_0^\conjbeta_\alpha}
		\equiv -\tfrac{1}{2}i\rho^2\tensor{\Phi}{_\alpha^\conjbeta}+\rho^2\tensor{A}{_\alpha^\conjbeta},\\
	\tensor{D}{_\alpha^\infty_0}&\equiv 0,\quad
	\tensor{D}{_\alpha^0_0}\equiv 0,\quad
	\tensor{D}{_\alpha^\beta_0}
		\equiv \tfrac{1}{2}i\tensor{\delta}{_\alpha^\beta}-\tfrac{1}{2}i\rho^2\tensor{\Phi}{_\alpha^\beta},\quad
	\tensor{D}{_\alpha^\conjbeta_0}
		\equiv -\tfrac{1}{2}i\rho^2\tensor{\Phi}{_\alpha^\conjbeta},\\
	\tensor{D}{_\alpha^\infty_\conjbeta}&\equiv \tfrac{1}{4}\tensor{h}{_\alpha_\conjbeta},\quad
	\tensor{D}{_\alpha^0_\conjbeta}\equiv \tfrac{1}{2}i\tensor{h}{_\alpha_\conjbeta},\quad
	\tensor{D}{_\alpha^\gamma_\conjbeta}\equiv 0,\quad
	\tensor{D}{_\alpha^\conjgamma_\conjbeta}\equiv 0,\\
	\tensor{D}{_\alpha^\infty_\beta}&\equiv 0,\quad
	\tensor{D}{_\alpha^0_\beta}\equiv -\rho^2\tensor{A}{_\alpha_\beta},\quad
	\tensor{D}{_\alpha^\gamma_\beta}\equiv 0,\quad
	\tensor{D}{_\alpha^\conjgamma_\beta}\equiv -\rho\tensor{N}{_\alpha^\conjgamma_\beta}.
\end{split}
\end{equation}
Using Table \ref{tbl:PerturbationOfConnectionForms} and \eqref{eq:ACHConnectionFormsToSecondOrder},
we compute
\begin{equation*}
	\tensor{\overline\nabla}{_K}(\delta\tensor{D}{_I^K_J}),\quad
	\tensor{\overline\nabla}{_J}(\delta\tensor{D}{_I^K_K}),\quad
	\tensor{D}{_I^L_K}\cdot\delta\tensor{D}{_J^K_L},\quad
	\tensor{D}{_I^K_L}\cdot\delta\tensor{D}{_K^L_L}\quad\text{and}\quad
	\tensor{D}{_K^L_L}\cdot\delta\tensor{D}{_I^K_J},
\end{equation*}
all modulo $O(\rho^{m+3})$. The result is Tables
\ref{tbl:DerivativeOfPerturbationOfConnectionForms1}--\ref{tbl:ProductOfPerturbationOfConnectionForms3}.
From these tables and
\begin{equation*}
	\begin{split}
		\tensor{\delta\Ein}{_I_J}
		&\equiv \tfrac{1}{2}(n+2)\delta\tensor{g}{_I_J}
		+\tensor{\overline\nabla}{_K}(\tensor{\delta D}{_I^K_J})
		-\tensor{\overline\nabla}{_J}(\tensor{\delta D}{_I^K_K})\\
		&\phantom{=\;}
		-\tensor{D}{_I^L_K}\cdot\tensor{\delta D}{_J^K_L}-\tensor{D}{_J^L_K}\cdot\tensor{\delta D}{_I^K_L}\\
		&\phantom{=\;}
		+\tensor{D}{_I^L_J}\cdot\tensor{\delta D}{_L^K_K}+\tensor{D}{_L^K_K}\cdot\tensor{\delta D}{_I^L_J}
		\mod O(\rho^{m+3}),
	\end{split}
\end{equation*}
we can verify Proposition \ref{prop:HigherOrderPerturbationAndEinsteinTensor}.

\begin{table}[htbp]
	\small
\begin{tabular}{ll}
	\toprule
	\multicolumn{1}{c}{Type} & \multicolumn{1}{c}{Value (modulo $O(\rho^{m+3})$)} \\ \midrule
	$\tensor{\overline\nabla}{_K}(\tensor{\delta D}{_\infty^K_\infty})$ & $0$ \\ \midrule
	$\tensor{\overline\nabla}{_K}(\tensor{\delta D}{_\infty^K_0})$
	& $\tfrac{1}{2}\rho(\euler+1)
		(\tensor{\nabla}{^\alpha}\tensor{\psi}{_0_\alpha}
		+\tensor{\nabla}{^\conjalpha}\tensor{\psi}{_0_\conjalpha})$ \\ \midrule
	$\tensor{\overline\nabla}{_K}(\tensor{\delta D}{_0^K_0})$
	& $-\tfrac{1}{8}(\euler-3)(\euler-4)\tensor{\psi}{_0_0}$ \\ \midrule
	$\tensor{\overline\nabla}{_K}(\tensor{\delta D}{_\alpha^K_\conjbeta})$
	& $-\tfrac{1}{8}(\euler-1)(\euler-2)\tensor{\psi}{_\alpha_\conjbeta}
		+\tfrac{i}{2}\rho(\tensor{\nabla}{_\alpha}\tensor{\psi}{_0_\conjbeta}
		-\tensor{\nabla}{_\conjbeta}\tensor{\psi}{_0_\alpha})$ \\
	& $+\tfrac{1}{2}\rho^2\tensor{\nabla}{^\conjgamma}
		(\tensor{\nabla}{_\alpha}\tensor{\psi}{_\conjbeta_\conjgamma}
		-\tensor{N}{_\conjbeta_\conjgamma^\sigma}\tensor{\psi}{_\alpha_\sigma})
		+\tfrac{1}{2}\rho^2\tensor{\nabla}{^\gamma}
		(\tensor{\nabla}{_\conjbeta}\tensor{\psi}{_\alpha_\gamma}
		-\tensor{N}{_\alpha_\gamma^\conjsigma}\tensor{\psi}{_\conjbeta_\conjsigma})$ \\ \bottomrule
\end{tabular}
	\vspace{0.5em}
	\caption{$\tensor{\overline\nabla}{_K}(\delta\tensor{D}{_I^K_J})$
	for a perturbation \eqref{eq:PerturbedMetric}}
	\label{tbl:DerivativeOfPerturbationOfConnectionForms1}
\end{table}

\begin{table}[htbp]
	\small
\begin{tabular}{ll}
	\toprule
	\multicolumn{1}{c}{Type} & \multicolumn{1}{c}{Value (modulo $O(\rho^{m+3})$)} \\ \midrule
	$\tensor{\overline\nabla}{_\infty}(\tensor{\delta D}{_\infty^K_K})$
	& $\tfrac{1}{2}\euler(\euler-1)\tensor{\psi}{_0_0}+\euler(\euler-1)\tensor{\psi}{_\alpha^\alpha}$ \\
	& $-\tfrac{1}{2}\rho^2(\euler+1)(\euler+2)
			(\tensor{\Phi}{^\alpha^\beta}\tensor{\psi}{_\alpha_\beta}
			+\tensor{\Phi}{^\conjalpha^\conjbeta}\tensor{\psi}{_\conjalpha_\conjbeta})$ \\ \midrule
	$\tensor{\overline\nabla}{_0}(\tensor{\delta D}{_\infty^K_K})$ & $0$ \\ \midrule
	$\tensor{\overline\nabla}{_0}(\tensor{\delta D}{_0^K_K})$ & $0$ \\ \midrule
	$\tensor{\overline\nabla}{_\conjbeta}(\tensor{\delta D}{_\alpha^K_K})$ & $0$ \\ \bottomrule
\end{tabular}
	\vspace{0.5em}
	\caption{$\tensor{\overline\nabla}{_J}(\delta\tensor{D}{_I^K_K})$
	for a perturbation \eqref{eq:PerturbedMetric}}
	\label{tbl:DerivativeOfPerturbationOfConnectionForms2}
\end{table}

\begin{table}[htbp]
	\small
\begin{tabular}{ll}
	\toprule
	\multicolumn{1}{c}{Type} & \multicolumn{1}{c}{Value (modulo $O(\rho^{m+3})$)} \\ \midrule
	$\tensor{D}{_\infty^L_K}\cdot\tensor{\delta D}{_\infty^K_L}$
	& $-\euler\tensor{\psi}{_0_0}-\euler\tensor{\psi}{_\alpha^\alpha}
	+\rho^2(\euler+1)(\tensor{\Phi}{^\alpha^\beta}\tensor{\psi}{_\alpha_\beta}
		+\tensor{\Phi}{^\conjalpha^\conjbeta}\tensor{\psi}{_\conjalpha_\conjbeta})$ \\ \midrule
	$\tensor{D}{_\infty^L_K}\cdot\tensor{\delta D}{_0^K_L}$
	& $-\tfrac{i}{2}\rho^2(\tensor{\Phi}{^\alpha^\beta}\tensor{\psi}{_\alpha_\beta}
		-\tensor{\Phi}{^\conjalpha^\conjbeta}\tensor{\psi}{_\conjalpha_\conjbeta})$ \\ \midrule
	$\tensor{D}{_0^L_K}\cdot\tensor{\delta D}{_\infty^K_L}$
	& $\tfrac{i}{2}\rho^2(\tensor{\Phi}{^\alpha^\beta}\tensor{\psi}{_\alpha_\beta}
		-\tensor{\Phi}{^\conjalpha^\conjbeta}\tensor{\psi}{_\conjalpha_\conjbeta})
		+\tfrac{1}{2}\rho^2(\tensor{A}{^\alpha^\beta}\euler\tensor{\psi}{_\alpha_\beta}
			+\tensor{A}{^\conjalpha^\conjbeta}\euler\tensor{\psi}{_\conjalpha_\conjbeta})$ \\ \midrule
	$\tensor{D}{_0^L_K}\cdot\tensor{\delta D}{_0^K_L}$
	& $\tfrac{1}{2}(\euler-n-2)\tensor{\psi}{_0_0}+\tfrac{1}{2}\tensor{\psi}{_\alpha^\alpha}
		-\tfrac{i}{2}\rho(\tensor{\nabla}{^\alpha}\tensor{\psi}{_0_\alpha}
			-\tensor{\nabla}{^\conjalpha}\tensor{\psi}{_0_\conjalpha})$ \\
	& $-\tfrac{1}{4}\rho^2(\tensor{\Phi}{^\alpha^\beta}\tensor{\psi}{_\alpha_\beta}
		+\tensor{\Phi}{^\conjalpha^\conjbeta}\tensor{\psi}{_\conjalpha_\conjbeta})$ \\ \midrule
	$\tensor{D}{_\alpha^L_K}\cdot\tensor{\delta D}{_\conjbeta^K_L}$
	& $\tfrac{1}{4}(\euler-2)\tensor{\psi}{_\alpha_\conjbeta}
		+\tfrac{1}{4}\tensor{h}{_\alpha_\conjbeta}\tensor{\psi}{_0_0}
		+\tfrac{i}{2}\rho\tensor{\nabla}{_\conjbeta}\tensor{\psi}{_0_\alpha}
		-\tfrac{1}{4}\rho^2(\euler-2)\tensor{\Phi}{_\alpha^\conjgamma}\tensor{\psi}{_\conjbeta_\conjgamma}$ \\
	& $-\tfrac{i}{2}\rho^2(\tensor{A}{_\alpha^\conjgamma}\tensor{\psi}{_\conjbeta_\conjgamma}
			+\tensor{A}{_\conjbeta^\gamma}\tensor{\psi}{_\alpha_\gamma})
			-\tfrac{1}{2}\rho^2\tensor{N}{_\alpha^\conjgamma^\conjsigma}
			(\tensor{\nabla}{_\conjbeta}\tensor{\psi}{_\conjgamma_\conjsigma}
			+\tensor{\nabla}{_\conjgamma}\tensor{\psi}{_\conjbeta_\conjsigma}
			-\tensor{\nabla}{_\conjsigma}\tensor{\psi}{_\conjbeta_\conjgamma})$ \\ \bottomrule
\end{tabular}
	\vspace{0.5em}
	\caption{$\tensor{D}{_I^L_K}\cdot\delta\tensor{D}{_J^K_L}$
	for a perturbation \eqref{eq:PerturbedMetric}}
	\label{tbl:ProductOfPerturbationOfConnectionForms1}
\end{table}

\begin{table}[htbp]
	\small
\begin{tabular}{ll}
	\toprule
	\multicolumn{1}{c}{Type} & \multicolumn{1}{c}{Value (modulo $O(\rho^{m+3})$)} \\ \midrule
	$\tensor{D}{_\infty^K_\infty}\cdot\tensor{\delta D}{_K^L_L}$
	& $-\tfrac{1}{2}\euler\tensor{\psi}{_0_0}-\euler\tensor{\psi}{_\alpha^\alpha}
		+\tfrac{1}{2}\rho^2(\euler+2)(\tensor{\Phi}{^\alpha^\beta}\tensor{\psi}{_\alpha_\beta}
			+\tensor{\Phi}{^\conjalpha^\conjbeta}\tensor{\psi}{_\conjalpha_\conjbeta})$ \\ \midrule
	$\tensor{D}{_\infty^K_0}\cdot\tensor{\delta D}{_K^L_L}$ & $0$ \\ \midrule
	$\tensor{D}{_0^K_0}\cdot\tensor{\delta D}{_K^L_L}$
	& $\tfrac{1}{4}\euler\tensor{\psi}{_0_0}+\tfrac{1}{2}\euler\tensor{\psi}{_\alpha^\alpha}
		-\tfrac{1}{4}\rho^2(\euler+2)(\tensor{\Phi}{^\alpha^\beta}\tensor{\psi}{_\alpha_\beta}
			+\tensor{\Phi}{^\conjalpha^\conjbeta}\tensor{\psi}{_\conjalpha_\conjbeta})$ \\ \midrule
	$\tensor{D}{_\alpha^K_\conjbeta}\cdot\tensor{\delta D}{_K^L_L}$
	& $\tfrac{1}{8}\tensor{h}{_\alpha_\conjbeta}\euler\tensor{\psi}{_0_0}
	+\tfrac{1}{4}\tensor{h}{_\alpha_\conjbeta}\euler\tensor{\psi}{_\gamma^\gamma}
	-\tfrac{1}{8}\rho^2\tensor{h}{_\alpha_\conjbeta}(\euler+2)
	(\tensor{\Phi}{^\sigma^\tau}\tensor{\psi}{_\sigma_\tau}
	+\tensor{\Phi}{^\conjsigma^\conjtau}\tensor{\psi}{_\conjsigma_\conjtau})$ \\ \bottomrule
\end{tabular}
	\vspace{0.5em}
	\caption{$\tensor{D}{_I^K_L}\cdot\delta\tensor{D}{_K^L_L}$
	for a perturbation \eqref{eq:PerturbedMetric}}
	\label{tbl:ProductOfPerturbationOfConnectionForms2}
\end{table}

\begin{table}[htbp]
	\small
\begin{tabular}{ll}
	\toprule
	\multicolumn{1}{c}{Type} & \multicolumn{1}{c}{Value (modulo $O(\rho^{m+3})$)} \\ \midrule
	$\tensor{D}{_K^L_L}\cdot\tensor{\delta D}{_\infty^K_\infty}$ & $0$ \\ \midrule
	$\tensor{D}{_K^L_L}\cdot\tensor{\delta D}{_\infty^K_0}$ & $0$ \\ \midrule
	$\tensor{D}{_K^L_L}\cdot\tensor{\delta D}{_0^K_0}$
	& $\tfrac{1}{8}(2n+3)(\euler-4)\tensor{\psi}{_0_0}$ \\ \midrule
	$\tensor{D}{_K^L_L}\cdot\tensor{\delta D}{_\alpha^K_\conjbeta}$
	& $\tfrac{1}{8}(2n+3)(\euler-2)\tensor{\psi}{_\alpha_\conjbeta}$ \\ \bottomrule
\end{tabular}
	\vspace{0.5em}
	\caption{$\tensor{D}{_K^L_L}\cdot\delta\tensor{D}{_I^K_J}$
	for a perturbation \eqref{eq:PerturbedMetric}}
	\label{tbl:ProductOfPerturbationOfConnectionForms3}
\end{table}

\begin{prop}
	\label{prop:SmoothHigherOrderPerturbationAndEinsteinTensor}
	Let $g$ be a normal-form ACH metric satisfying \eqref{eq:LowestTermsConditionOfACHMetric} and
	$g'$ given by \eqref{eq:PerturbedMetric}. Then,
	\allowdisplaybreaks
	\begin{subequations}
		\label{eq:PerturbationAndEinsteinTensor}
	\begin{align}
		\label{eq:PerturbationAndEinsteinTensorInfties}
		\begin{split}
			\tensor{\delta\Ein}{_\infty_\infty}
			&=-\tfrac{1}{2}(m+2)(m-2)\tensor{\psi}{_0_0}-m(m+2)\tensor{\psi}{_\alpha^\alpha}\\
			&\phantom{=\;}
				+\tfrac{1}{2}m^2\rho^2(\tensor{\Phi}{^\alpha^\beta}\tensor{\psi}{_\alpha_\beta}
				+\tensor{\Phi}{^\conjalpha^\conjbeta}\tensor{\psi}{_\conjalpha_\conjbeta})
				+O(\rho^{m+3}),
		\end{split}\\
		\label{eq:PerturbationAndEinsteinTensorInftyZero}
		\begin{split}
			\tensor{\delta\Ein}{_\infty_0}
			&=\tfrac{1}{2}(m+2)\rho(\tensor{\nabla}{^\alpha}\tensor{\psi}{_0_\alpha}
				+\tensor{\nabla}{^\conjalpha}\tensor{\psi}{_0_\conjalpha})
			-\tfrac{1}{2}m\rho^2(\tensor{A}{^\alpha^\beta}\tensor{\psi}{_\alpha_\beta}
				+\tensor{A}{^\conjalpha^\conjbeta}\tensor{\psi}{_\conjalpha_\conjbeta})\\
			&\phantom{=\;}
			+O(\rho^{m+3}),
		\end{split}\\
		\label{eq:PerturbationAndEinsteinTensorInftyAlpha}
		\tensor{\delta\Ein}{_\infty_\alpha}
		&=-\tfrac{1}{2}i(m+2)\tensor{\psi}{_0_\alpha}
			+\tfrac{1}{2}m\rho\tensor{\nabla}{^\beta}\tensor{\psi}{_\alpha_\beta}
			+\tfrac{1}{2}m\rho
				\tensor{N}{_\alpha^\conjbeta^\conjgamma}\tensor{\psi}{_\conjbeta_\conjgamma}
			+O(\rho^{m+2}),\\
		\label{eq:PerturbationAndEinsteinTensorZeros}
		\begin{split}
			\tensor{\delta\Ein}{_0_0}
			&=-\tfrac{1}{8}(m^2-2nm-8n-4)\tensor{\psi}{_0_0}
				+\tfrac{1}{2}m\tensor{\psi}{_\alpha^\alpha}
				+i\rho(\tensor{\nabla}{^\alpha}\tensor{\psi}{_0_\alpha}
					-\tensor{\nabla}{^\conjalpha}\tensor{\psi}{_0_\conjalpha})\\
			&\phantom{=\;}
				-\tfrac{1}{4}\rho^2m(\tensor{\Phi}{^\alpha^\beta}\tensor{\psi}{_\alpha_\beta}
					+\tensor{\Phi}{^\conjalpha^\conjbeta}\tensor{\psi}{_\conjalpha_\conjbeta})+O(\rho^{m+3}),
		\end{split}\\
		\label{eq:PerturbationAndEinsteinTensorZeroAlpha}
		\tensor{\delta\Ein}{_0_\alpha}
		&=-\tfrac{1}{8}(m+2)(m-2n-2)\tensor{\psi}{_0_\alpha}+O(\rho^{m+2}),\\
		\label{eq:PerturbationAndEinsteinTensorTrace}
		\begin{split}
			\tensor{\delta\Ein}{_\alpha^\alpha}
			&=\tfrac{1}{8}n(m-2)\tensor{\psi}{_0_0}
			-\tfrac{1}{8}\left( m^2-(4n-2)m-8n-8 \right)\tensor{\psi}{_\alpha^\alpha}\\
			&\phantom{=\;}
			+(\text{$O(\rho^{m+2})$ terms depending on
			$\tensor{\psi}{_0_\alpha}$ and $\tensor{\psi}{_\alpha_\beta}$})+O(\rho^{m+3}),
		\end{split}\\
		\label{eq:PerturbationAndEinsteinTensorTraceFreeHermitian}
		\begin{split}
			\tf(\tensor{\delta\Ein}{_\alpha_\conjbeta})
			&=-\tfrac{1}{8}(m^2-2nm-2n-9)\tf(\tensor{\psi}{_\alpha_\conjbeta})\\
			&\phantom{=\;}
			+(\text{$O(\rho^{m+2})$ terms depending on
			$\tensor{\psi}{_0_\alpha}$ and $\tensor{\psi}{_\alpha_\beta}$})+O(\rho^{m+3}),
		\end{split}\\
		\label{eq:PerturbationAndEinsteinTensorDoubleHolomorphic}
		\tensor{\delta\Ein}{_\alpha_\beta}
		&=-\tfrac{1}{8}m(m-2n-2)\tensor{\psi}{_\alpha_\beta}+O(\rho^{m+1}).
	\end{align}
	\end{subequations}
	\allowdisplaybreaks[0]%
\end{prop}

\begin{proof}
	This follows from \eqref{eq:EinPerturbationLowerOrder}, \eqref{eq:EinPerturbationHigherOrder} and
	the fact that the Euler vector field $\euler$ acts on an $O(\rho^m)$ function as, modulo
	$O(\rho^{m+1})$, a scalar multiplication by $m$.
\end{proof}

\section{Approximate solutions and obstruction}
\label{sec:ApproximateSolution}

By using the results in \S\ref{sec:RicciTensorAndSomeLowOrderTerms} and \S\ref{sec:HigherOrderPerturbation},
in this section we construct a normal-form ACH metric
whose Einstein tensor vanishes to as high order as possible.
First we observe the contracted Bianchi identity satisfied by the Einstein tensor.

\begin{lem}
	\label{lem:ContractedBianchiIdentityForACHMetric}
	Let $m\ge 1$ be a positive integer. Suppose that $g$ is a normal-form ACH metric satisfying
	\begin{align*}
		\tensor{\Ein}{_\infty_\infty}&=O(\rho^{m+2}),\qquad
		\tensor{\Ein}{_\infty_0}=O(\rho^{m+2}),\qquad
		\tensor{\Ein}{_\infty_\alpha}=O(\rho^{\max\set{m+1,3}}),\\
		\tensor{\Ein}{_0_0}&=O(\rho^{m+2}),\qquad
		\tensor{\Ein}{_0_\alpha}=O(\rho^{\max\set{m+1,3}}),\\
		\tensor{\Ein}{_\alpha_\conjbeta}&=O(\rho^{m+2}),\qquad
		\tensor{\Ein}{_\alpha_\beta}=O(\rho^{\max\set{m,3}}).
	\end{align*}
	Then we have
	\begin{subequations}
		\label{eq:ContractedBianchiIdentityForACHMetric}
	\begin{align}
		\label{eq:ContractedBianchiIdentityForACHMetric1}
		\begin{split}
			O(\rho^{m+3})&=
			(m-4n-2)\tensor{\Ein}{_\infty_\infty}-4(m-2)\tensor{\Ein}{_0_0}
			-8m\tensor{\Ein}{_\alpha^\alpha}\\
			&\phantom{=\;}
			+8\rho(\tensor{\nabla}{^\alpha}\tensor{\Ein}{_\infty_\alpha}
			+\tensor{\nabla}{^\conjalpha}\tensor{\Ein}{_\infty_\conjalpha})
			+4\rho^2(m-2)(\tensor{\Phi}{^\alpha^\beta}\tensor{\Ein}{_\alpha_\beta}
			+\tensor{\Phi}{^\conjalpha^\conjbeta}\tensor{\Ein}{_\conjalpha_\conjbeta}),
		\end{split}\\
		\label{eq:ContractedBianchiIdentityForACHMetric2}
		\begin{split}
			O(\rho^{m+3})&=
			(m-2n-2)\tensor{\Ein}{_\infty_0}
			+4\rho(\tensor{\nabla}{^\alpha}\tensor{\Ein}{_0_\alpha}
			+\tensor{\nabla}{^\conjalpha}\tensor{\Ein}{_0_\conjalpha})\\
			&\phantom{=\;}
			+4\rho^2(\tensor{A}{^\alpha^\beta}\tensor{\Ein}{_\alpha_\beta}
			+\tensor{A}{^\conjalpha^\conjbeta}\tensor{\Ein}{_\conjalpha_\conjbeta}),
		\end{split}\\
		\label{eq:ContractedBianchiIdentityForACHMetric3}
		\begin{split}
			O(\rho^{m+2})&=
			2(m-2n-2)\tensor{\Ein}{_\infty_\alpha}+4\rho\tensor{\nabla}{^\beta}\tensor{\Ein}{_\alpha_\beta}
			-4i\tensor{\Ein}{_0_\alpha}
			+4\rho\tensor{N}{_\alpha^\conjbeta^\conjgamma}\tensor{\Ein}{_\conjbeta_\conjgamma}.
		\end{split}
	\end{align}
	\end{subequations}
\end{lem}

\begin{proof}
	We have the contracted Bianchi identity
	$\tensor{g}{^I^J}\tensor{{\nabla^g}}{_K}\tensor{\Ric}{_I_J}
	=2\tensor{g}{^I^J}\tensor{{\nabla^g}}{_I}\tensor{\Ric}{_J_K}$,
	where $\nabla^g$ is the Levi-Civita connection determined by $g$.
	Since $\nabla^g$ is a metric connection we further have
	\begin{equation*}
		\tensor{g}{^I^J}\tensor{{\nabla^g}}{_K}\tensor{\Ein}{_I_J}
	=2\tensor{g}{^I^J}\tensor{{\nabla^g}}{_I}\tensor{\Ein}{_J_K}.
	\end{equation*}
	In terms of the extended Tanaka--Webster connection $\overline{\nabla}$ and the tensor $D$,
	we can rewrite this identity as
	\begin{equation*}
		\tensor{g}{^I^J}
		(\tensor{\overline\nabla}{_K}\tensor{\Ein}{_I_J}-2\tensor{D}{_I^L_K}\tensor{\Ein}{_J_L})
		=2\tensor{g}{^I^J}(\tensor{\overline\nabla}{_I}\tensor{\Ein}{_J_K}
		-\tensor{D}{_J^L_I}\tensor{\Ein}{_L_K}-\tensor{D}{_K^L_I}\tensor{\Ein}{_J_L}),
	\end{equation*}
	or equivalently,
	\begin{equation*}
		0=\tensor{g}{^I^J}(\tensor{\overline\nabla}{_K}\tensor{\Ein}{_I_J}
		-2\tensor{\overline{\nabla}}{_I}\tensor{\Ein}{_J_K}+2\tensor{D}{_I^L_J}\tensor{\Ein}{_K_L}
		-2\tensor{\overline{\Theta}}{_I_K^L}\tensor{\Ein}{_J_L}),
	\end{equation*}
	where $\overline\Theta$ is the torsion form of $\overline{\nabla}$.
	Since $\tensor{g}{^0^\alpha}=O(\rho^3)$ and $\tensor{\Ein}{_I_J}=O(\rho^m)$, we obtain
	\begin{equation*}
		\begin{split}
			O(\rho^{m+3})
			&=\tensor{g}{^\infty^\infty}(\tensor{\overline\nabla}{_K}\tensor{\Ein}{_\infty_\infty}
			-2\tensor{\overline{\nabla}}{_\infty}\tensor{\Ein}{_\infty_K}
			+2\tensor{D}{_\infty^L_\infty}\tensor{\Ein}{_K_L}
			-2\tensor{\overline{\Theta}}{_\infty_K^L}\tensor{\Ein}{_\infty_L})\\
			&\phantom{=\;}
			+\tensor{g}{^0^0}(\tensor{\overline\nabla}{_K}\tensor{\Ein}{_0_0}
			-2\tensor{\overline{\nabla}}{_0}\tensor{\Ein}{_0_K}+2\tensor{D}{_0^L_0}\tensor{\Ein}{_K_L}
			-2\tensor{\overline{\Theta}}{_0_K^L}\tensor{\Ein}{_0_L})\\
			&\phantom{=\;}
			+2\tensor{g}{^\beta^\conjgamma}
			\big(\tensor{\overline\nabla}{_K}\tensor{\Ein}{_\beta_\conjgamma}
			-\tensor{\overline{\nabla}}{_\beta}\tensor{\Ein}{_\conjgamma_K}
			-\tensor{\overline{\nabla}}{_\conjgamma}\tensor{\Ein}{_\beta_K}
			+(\tensor{D}{_\beta^L_\conjgamma}+\tensor{D}{_\conjgamma^L_\beta})\tensor{\Ein}{_K_L}\\
			&\phantom{=\;+2\tensor{g}{^\beta^\conjgamma}\big(}
			-\tensor{\overline{\Theta}}{_\beta_K^L}\tensor{\Ein}{_\conjgamma_L}
			-\tensor{\overline{\Theta}}{_\conjgamma_K^L}\tensor{\Ein}{_\beta_L}\big)\\
			&\phantom{=\;}
			+\tensor{g}{^\beta^\gamma}(\tensor{\overline\nabla}{_K}\tensor{\Ein}{_\beta_\gamma}
			-2\tensor{\overline{\nabla}}{_\beta}\tensor{\Ein}{_\gamma_K}
			+2\tensor{D}{_\beta^L_\gamma}\tensor{\Ein}{_K_L}
			-2\tensor{\overline{\Theta}}{_\beta_K^L}\tensor{\Ein}{_\gamma_L})\\
			&\phantom{=\;}
			+\tensor{g}{^\conjbeta^\conjgamma}(\tensor{\overline\nabla}{_K}\tensor{\Ein}{_\conjbeta_\conjgamma}
			-2\tensor{\overline{\nabla}}{_\conjbeta}\tensor{\Ein}{_\conjgamma_K}
			+2\tensor{D}{_\conjbeta^L_\conjgamma}\tensor{\Ein}{_K_L}
			-2\tensor{\overline{\Theta}}{_\conjbeta_K^L}\tensor{\Ein}{_\conjgamma_L}).
		\end{split}
	\end{equation*}
	Substituting $K=\infty$, $K=0$ and $K=\alpha$ into this formula,
	in view of \eqref{eq:TorsionOfExtendedTWConnection}, \eqref{eq:CalculationOfTrivialExtensionOfTWConnection}
	and \eqref{eq:ACHConnectionFormsToSecondOrder} we find that
	\allowdisplaybreaks
	\begin{subequations}
		\label{eq:ContractedBianchiOriginal}
	\begin{align}
		\label{eq:ContractedBianchiOriginal1}
		\begin{split}
			O(\rho^{m+3})&=
			(\euler-4n-4)\tensor{\Ein}{_\infty_\infty}-4(\euler-4)\tensor{\Ein}{_0_0}
			-8(\euler-2)\tensor{\Ein}{_\alpha^\alpha}\\
			&\phantom{=\;}
			+8\rho(\tensor{\nabla}{^\alpha}\tensor{\Ein}{_\infty_\alpha}
			+\tensor{\nabla}{^\conjalpha}\tensor{\Ein}{_\infty_\conjalpha})\\
			&\phantom{=\;}
			+4\rho^2(\euler-2)(\tensor{\Phi}{^\alpha^\beta}\tensor{\Ein}{_\alpha_\beta}
			+\tensor{\Phi}{^\conjalpha^\conjbeta}\tensor{\Ein}{_\conjalpha_\conjbeta}),
		\end{split}\\
		\label{eq:ContractedBianchiOriginal2}
		\begin{split}
			O(\rho^{m+3})&=
			(\euler-2n-4)\tensor{\Ein}{_\infty_0}
			+4\rho(\tensor{\nabla}{^\alpha}\tensor{\Ein}{_0_\alpha}
			+\tensor{\nabla}{^\conjalpha}\tensor{\Ein}{_0_\conjalpha})\\
			&\phantom{=\;}
			+4\rho^2(\tensor{A}{^\alpha^\beta}\tensor{\Ein}{_\alpha_\beta}
			+\tensor{A}{^\conjalpha^\conjbeta}\tensor{\Ein}{_\conjalpha_\conjbeta}),
		\end{split}\\
		\label{eq:ContractedBianchiOriginal3}
		\begin{split}
			O(\rho^{m+2})&=
			2(\euler-2n-3)\tensor{\Ein}{_\infty_\alpha}
			+4\rho\tensor{\nabla}{^\beta}\tensor{\Ein}{_\alpha_\beta}
			-4i\tensor{\Ein}{_0_\alpha}
			+4\rho\tensor{N}{_\alpha^\conjbeta^\conjgamma}\tensor{\Ein}{_\conjbeta_\conjgamma},
		\end{split}
	\end{align}
	\end{subequations}
	\allowdisplaybreaks[0]%
	which imply \eqref{eq:ContractedBianchiIdentityForACHMetric}.
\end{proof}

Let
	\begin{equation}
		\label{eq:DefinitionOfOrderModification}
		a(I,J)=
		\begin{cases}
			3,&\text{$(I,J)=(\infty,\infty)$, $(\infty,0)$, $(0,0)$, $(\alpha,\conjbeta)$},\\
			2,&\text{$(I,J)=(\infty,\alpha)$, $(0,\alpha)$},\\
			1,&(I,J)=(\alpha,\beta).
		\end{cases}
	\end{equation}
The next theorem proves Theorem \ref{thm:MainTheoremOnExistence}.

\begin{thm}
	\label{thm:ApproximateACHEMetric}
	Let $(M,T^{1,0})$ be a nondegenerate partially integrable almost CR manifold,
	$\theta$ any pseudohermitian structure
	and $X$ an open neighborhood of $M=M\times\set{0}$ in $M\times[0,\infty)$.
	Then there exists a normal-form ACH metric $g$ on $X$ which satisfies
	\begin{equation}
		\label{eq:ApproximatelyEinstein}
		\tensor{\Ein}{_I_J}=O(\rho^{2n+1+a(I,J)})
	\end{equation}
	with respect to the frame \eqref{eq:LocalFrameOfThetaTangentBundle} of $\ThetaTangentX$.
	For such a metric, each $\tensor{g}{_i_j}$ is uniquely determined modulo $O(\rho^{2n+1+a(i,j)})$.
\end{thm}

\begin{proof}
	By Proposition \ref{prop:ACHEToThirdOrder} we already have a normal-form ACH metric $g^{(0)}$ satisfying
	$\tensor{\Ein}{_I_J}=O(\rho^3)$ for every $I$, $J$, with $O(\rho^3)$ ambiguity in each component
	$\tensor*{g}{^{(0)}_i_j}$.
	We shall inductively show that there exists a normal-form ACH metric $g^{(m)}$ satisfying
	\begin{equation}
		\label{eq:GoalOfEachStep}
		\tensor{\Ein}{_I_J}=O(\rho^{\max\set{m+a(I,J),3}}).
	\end{equation}
	for each $m$, $m=1$, $\dots$, $2n+1$, and for such $g^{(m)}$ its components $\tensor*{g}{^{(m)}_i_j}$ are
	unique modulo $O(\rho^{\max\set{m+a(i,j),3}})$.
	
	Suppose we have a normal-form ACH metric $g^{(m-1)}$ that satisfies \eqref{eq:GoalOfEachStep} for $m-1$
	as well as \eqref{eq:LowestTermsConditionOfACHMetric}.
	Consider a new ACH metric $g^{(m)}$ given by
	$\tensor*{g}{^{(m)}_i_j}=\tensor*{g}{^{(m-1)}_i_j}+\tensor{\psi}{_i_j}$,
	where $\tensor{\psi}{_i_j}$ is such that $\tensor{\psi}{_i_j}=O(\rho^{\max\set{m-1+a(i,j),3}})$.
	Then the difference $\delta\Ein=\Ein'-\Ein$ between the Einstein tensors is given in
	Proposition \ref{prop:SmoothHigherOrderPerturbationAndEinsteinTensor}.
	In view of \eqref{eq:PerturbationAndEinsteinTensorZeroAlpha} and
	\eqref{eq:PerturbationAndEinsteinTensorDoubleHolomorphic} we can determine
	$\tensor{\psi}{_0_\alpha}$ mod $O(\rho^{m+2})$ and
	$\tensor{\psi}{_\alpha_\beta}$ mod $O(\rho^{\max\set{m+1,3}})$ so that
	$\tensor*{\Ein}{^{(m)}_0_\alpha}=O(\rho^{m+2})$ and
	$\tensor*{\Ein}{^{(m)}_\alpha_\beta}=O(\rho^{\max\set{m+1,3}})$ hold,
	because the exponents $-\tfrac{1}{8}(m+2)(m-2n-2)$ and $-\tfrac{1}{8}m(m-2n-2)$ are nonzero for
	$m=1$, $\dots$, $2n+1$.
	After that, by a similar reasoning using \eqref{eq:PerturbationAndEinsteinTensorTraceFreeHermitian},
	we can determine $\tf(\tensor{\psi}{_\alpha_\conjbeta})$ mod $O(\rho^{m+3})$
	so that $\tf(\tensor*{\Ein}{^{(m)}_\alpha_\conjbeta})=O(\rho^{m+3})$ hold.
	Next we see \eqref{eq:PerturbationAndEinsteinTensorZeros} and \eqref{eq:PerturbationAndEinsteinTensorTrace}
	as a system of linear equations for $\tensor{\psi}{_0_0}$ and $\tensor{\psi}{_\alpha^\alpha}$.
	The determinant of the coefficients is
	\begin{equation}
		\label{eq:DeterminantOfCoefficients}
	\begin{split}
		&\begin{vmatrix}
			-\tfrac{1}{8}(m^2-2nm-8n-4) & \tfrac{1}{2}m \\
			\tfrac{1}{8}n(m-2) & -\tfrac{1}{8}\left(m^2-(4n-2)m-8n-8\right)
		\end{vmatrix}\\
		&\qquad=\tfrac{1}{64}(m+2)(m+4)(m-2n-2)(m-4n-2),
	\end{split}
	\end{equation}
	which shows that this system is nondegenerate for $m=1$, $\dots$, $2n+1$.
	Hence we can determine $\tensor{\psi}{_0_0}$ and $\tensor{\psi}{_\alpha^\alpha}$, both modulo $O(\rho^{m+3})$,
	so that $\tensor*{\Ein}{^{(m)}_0_0}=O(\rho^{m+3})$ and $\tensor{\Ein}{^{(m)}_\alpha^\alpha}=O(\rho^{m+3})$
	hold.
	Thus we have attained $\tensor{\Ein}{_i_j}=O(\rho^{\max\set{m+a(i,j),3}})$,
	and if $\tensor*{g}{^{(m-1)}_i_j}$ are unique up to $O(\rho^{\max\set{m-1+a(i,j),3}})$,
	the desired uniqueness result holds for $\tensor*{g}{^{(m)}_i_j}$.
	
	Finally we check that $g^{(m)}$ is determined in such a way that it satisfies \eqref{eq:GoalOfEachStep}
	for $I=\infty$, too.
	This is done by using Lemma \ref{lem:ContractedBianchiIdentityForACHMetric}.
	In fact, for $g^{(m)}$, $\tensor*{\Ein}{^{(m)}_\infty_0}=O(\rho^{m+3})$ and
	$\tensor*{\Ein}{^{(m)}_\infty_\alpha}=O(\rho^{m+2})$ should hold,
	because in \eqref{eq:ContractedBianchiIdentityForACHMetric2} and
	\eqref{eq:ContractedBianchiIdentityForACHMetric3} the terms on the right-hand sides are, except
	the first terms in each identity, already $O(\rho^{m+3})$ and $O(\rho^{m+2})$, respectively,
	and the coefficients of the first terms are both nonzero.
	Similarly \eqref{eq:ContractedBianchiIdentityForACHMetric1} shows that
	$\tensor*{\Ein}{^{(m)}_\infty_\infty}=O(\rho^{m+3})$.
	Hence the induction is complete.
\end{proof}

In spite of the success of the inductive determination of $\tensor{g}{_i_j}$ up to the stage
in the theorem above, the next step cannot be executed, as
\eqref{eq:PerturbationAndEinsteinTensorZeroAlpha} and
\eqref{eq:PerturbationAndEinsteinTensorDoubleHolomorphic} indicate;
the freedom of the choice of $g$ satisfying \eqref{eq:ApproximatelyEinstein} does not affect
the $\rho^{2n+2}$-term coefficient of $\tensor{\Ein}{_\alpha_\beta}$ and
the $\rho^{2n+3}$-term coefficient of $\tensor{\Ein}{_0_\alpha}$.
So we define
\begin{equation}
	\tensor{\caO}{_\alpha_\beta}:=\left.\left(\rho^{-2n-2}\tensor{\Ein}{_\alpha_\beta}\right)\right|_{\rho=0}
	\label{eq:DefinitionOfObstructionTensor}
\end{equation}
and call it the \emph{obstruction tensor} associated with $(M,T^{1,0},\theta)$.
In fact, the condition $\tensor{\Ein}{_\alpha_\conjbeta}=O(\rho^{2n+4})$ on the metric from which
$\tensor{\caO}{_\alpha_\beta}$ is computed can be weakened to $\tensor{\Ein}{_\alpha_\conjbeta}=O(\rho^{2n+3})$,
for the $O(\rho^{2n+3})$ ambiguity in $\tf(\tensor{g}{_\alpha_\conjbeta})$ emerging from that does not have
any effect on $\rho^{2n+2}$-term coefficient of $\tensor{\Ein}{_\alpha_\beta}$
as \eqref{eq:PerturbationAndEinsteinTensorDoubleHolomorphic} shows.
This fact further implies that we can use any approximately Einstein ACH metric $g$ that
Theorem \ref{thm:MainTheoremOnExistence} claims its existence,
because if $\rho$ is a model boundary defining function for $g$ and $\theta$, then
there is a boundary-fixing $[\Theta]$-diffeomorphism $\Phi$
such that $\Phi^*g$ is a normal-form ACH metric for which the second coordinate function is equal to $\Phi^*\rho$,
and its Einstein tensor vanishes to the same order as that of $g$ does.

The $\rho^{2n+3}$-term coefficient of $\tensor{\Ein}{_0_\alpha}$ is not a new obstruction,
since by \eqref{eq:ContractedBianchiIdentityForACHMetric3} we have
\begin{equation}
	\label{eq:OnPossibleSecondObstruction}
	\left.\left(\rho^{-2n-3}\tensor{\Ein}{_0_\alpha}\right)\right|_{M}
	=-i\tensor{\nabla}{^\beta}\tensor{\caO}{_\alpha_\beta}
	-i\tensor{N}{_\alpha^\conjbeta^\conjgamma}\tensor{\caO}{_\conjbeta_\conjgamma}.
\end{equation}

\begin{prop}
	\label{prop:TransformationLawOfObstructionTensor}
	Let $\theta$ and $\Hat\theta=e^{2u}\theta$, $u\in C^\infty(M)$, be two pseudohermitian structures on
	$(M,T^{1,0})$. Then
	\begin{equation}
		\tensor{\Hat\caO}{_\alpha_\beta}=e^{-2nu}\tensor{\caO}{_\alpha_\beta},
		\label{eq:TransformationLawOfObstructionTensor}
	\end{equation}
	where $\tensor{\caO}{_\alpha_\beta}$ is the obstruction tensor for $(M,T^{1,0},\theta)$
	and $\tensor{\Hat\caO}{_\alpha_\beta}$ is that for $(M,T^{1,0},\Hat\theta)$.
\end{prop}

\begin{proof}
	Let $(X,[\Theta])$ be a $[\Theta]$-manifold
	such that $\bdry X=M$ and $\iota^*[\Theta]$ is the conformal class of the pseudohermitian structures on $M$,
	and take any ACH metric $g$ satisfying the condition in Theorem \ref{thm:MainTheoremOnExistence}.
	If $\rho$ is a model boundary defining function for $\theta$
	and $\Hat\rho=e^{\psi}\rho$, $\psi\in C^\infty(X)$, is one for $\Hat\theta$,
	then we have $\psi|_M=u$ by the condition $\iota^*({\Hat\rho}^4g)={\Hat\theta}^2$.
	Hence, if $\set{\tilde Z_\alpha}$ is any extension of a local frame of $T^{1,0}$, we have
	$\tensor{\Hat\caO}{_\alpha_\beta}
	=\big(\Hat\rho^{-2n-2}\Ein(\Hat\rho\tilde Z_\alpha,\Hat\rho\tilde Z_\beta)\big)\big|_{M}
	=e^{-2nu}\big(\rho^{-2n-2}\Ein(\rho\tilde Z_\alpha,\rho\tilde Z_\beta)\big)\big|_{M}
	=e^{-2nu}\tensor{\caO}{_\alpha_\beta}$.
\end{proof}

The proposition above implies that the density-weighted version of the obstruction tensor
\begin{equation*}
	\tensor{\bm{\mathcal{O}}}{_\alpha_\beta}
	:=\tensor{\mathcal{O}}{_\alpha_\beta}\otimes\lvert\zeta\rvert^{2n/(n+2)}
	\in\tensor{\mathcal{E}}{_(_\alpha_\beta_)}(-n,-n)
\end{equation*}
is a CR-invariant tensor,
where $\tensor{\mathcal{E}}{_(_\alpha_\beta_)}$ denotes the space of local sections of $\Sym^{2}(T^{1,0})^{*}$.

Next we recall \eqref{eq:OnPossibleSecondObstruction}. Let us also look at a similar result
\begin{equation*}
	\left.\left(\rho^{-2n-3}
	(\tensor{\nabla}{^\alpha}\tensor{\Ein}{_0_\alpha}+\tensor{\nabla}{^\conjalpha}\tensor{\Ein}{_0_\conjalpha})
	\right)\right|_M
	=-\tensor{A}{^\alpha^\beta}\tensor{\caO}{_\alpha_\beta}
	-\tensor{A}{^\conjalpha^\conjbeta}\tensor{\caO}{_\conjalpha_\conjbeta},
\end{equation*}
which follows from \eqref{eq:ContractedBianchiIdentityForACHMetric2}.
Combining these identities we obtain
\begin{equation}
	\label{eq:ImaginaryPartOfDoubleDivergenceOfObstruction}
	\tensor{D}{^\alpha^\beta}\tensor{\caO}{_\alpha_\beta}
	-\tensor{D}{^{\conj{\alpha}}^{\conj{\beta}}}\tensor{\caO}{_{\conj{\alpha}}_{\conj{\beta}}}=0,
\end{equation}
where
\begin{equation}
	\label{eq:DefinitionOfP}
	\tensor{D}{^\alpha^\beta}
	=\tensor{\nabla}{^\alpha}\tensor{\nabla}{^\beta}-i\tensor{A}{^\alpha^\beta}
	-\tensor{N}{^\gamma^\alpha^\beta}\tensor{\nabla}{_\gamma}-\tensor{N}{^\gamma^\alpha^\beta_,_\gamma}.
\end{equation}

Replacing $N$, $A$ with $\bm{N}$, $\bm{A}$ and taking contractions with respect not to $h$ but to
$\bm{h}$, we obtain a differential operator
$\tensor{\bm{D}}{^\alpha^\beta}\colon\tensor{\mathcal{E}}{_(_\alpha_\beta_)}(-n,-n)\to\mathcal{E}(-n-2,-n-2)$
between density-weighted bundles. Then we have
$\tensor{\bm{D}}{^\alpha^\beta}\tensor{\bm{\mathcal{O}}}{_\alpha_\beta}-\tensor{\bm{D}}{^{\conj{\alpha}}^{\conj{\beta}}}\tensor{\bm{\mathcal{O}}}{_{\conj{\alpha}}_{\conj{\beta}}}=0$.
Furthermore, $\tensor{\bm{D}}{^\alpha^\beta}$ belongs to a one-parameter family of
CR-invariant differential operators, as we shall describe in the following proposition.

\begin{prop}
	Let $(M,T^{1,0})$ be a nondegenerate partially integrable almost CR manifold. Let
	\begin{equation*}
		\tensor*{\bm{D}}{_t^\alpha^\beta}\colon
		\tensor{\mathcal{E}}{_(_\alpha_\beta_)}(-n,-n)\to\mathcal{E}(-n-2,-n-2),\quad t\in\mathbb{C}
	\end{equation*}
	be a one-parameter family of differential operators defined by,
	in terms of any pseudohermitian structure $\theta$,
	\begin{equation}
		\label{eq:DefinitionOfOneParameterFamilyPt}
		\tensor*{\bm{D}}{_t^\alpha^\beta}
		=\tensor{\nabla}{^\alpha}\tensor{\nabla}{^\beta}-i\tensor{\bm{A}}{^\alpha^\beta}
		-(1+tn)\tensor{\bm{N}}{^\gamma^\alpha^\beta}\tensor{\nabla}{_\gamma}
		-\left(1+t(n+1)\right)\tensor{\bm{N}}{^\gamma^\alpha^\beta_,_\gamma}.
	\end{equation}
	Then this is well-defined, i.e., the right-hand side of \eqref{eq:DefinitionOfOneParameterFamilyPt} is
	independent of $\theta$.
\end{prop}

\begin{proof}
	This can be checked by using equation (2.7) and Proposition 2.3 of \cite{GoverGraham},
	as we have remarked at the end of \S\ref{sec:PseudohermitianGeometry}. The details are left to the reader.
\end{proof}

The next proposition finishes the proof of Theorem \ref{thm:OnObstructionTensor}.

\begin{prop}
	\label{prop:ObstructionTensorForIntegrableStructure}
	The obstruction tensor $\tensor{\caO}{_\alpha_\beta}$ for a nondegenerate (integrable) CR manifold vanishes.
\end{prop}

\begin{proof}
	Since $\tensor{\caO}{_\alpha_\beta}$ is a certain polynomial of derivatives of pseudohermitian torsion and
	curvature,
	using the formal embedding (see, e.g., \cite{Kuranishi}) we can reduce the problem to the case of
	a (small piece of) nondegenerate real hypersurface $M\subset\bbC^{n+1}$.
	In this proof we use indices $j$, $k$ for components with respect to the complex coordinates
	$(z^1,\dots,z^{n+1})$.

	Let $r$ be Fefferman's approximate solution of the complex Monge--Amp\`ere equation \cite{Fefferman}, i.e.,
	a smooth defining function of $M$ such that $J(r)=1+O(r^{n+2})$, where
	\begin{equation*}
		J(r):=(-1)^{n+1}\det\begin{pmatrix}
			r & \partial r/\partial\conj{z}^k\\
			\partial r/\partial z^j & \partial^2 r/\partial z^j\partial\conj{z}^k
		\end{pmatrix}.
	\end{equation*}
	We set $\tilde\theta:=\tfrac{i}{2}(\partial r-\conj\partial r)$ and
	$\theta:=\iota^*\tilde\theta$, where $\iota\colon M\hookrightarrow\bbC^{n+1}$ is the inclusion.
	On $\Omega=\set{r>0}$, we consider the K\"ahler metric $G$ in Example \ref{ex:CompleteKahlerMetric}
	given by Fefferman's approximate solution $r$.
	It is easily verified that $\det(G_{j\conj{k}})=r^{-(n+2)}J(r)$, and
	the usual formula for the Ricci tensor of a K\"ahler metric shows that
	\begin{equation*}
		\tensor{\Ric(G)}{_j_{\conj{k}}}=-\tfrac{1}{2}(n+2)\tensor{G}{_j_{\conj{k}}}
		+\dfrac{\partial^2}{\partial z^j\partial\conj{z}^k}\log J(r).
	\end{equation*}
	Observe that, if we set $\log J(r)=r^{n+2}f$,
	\begin{equation}
		\label{eq:HighOrderTermInApproximateKahlerEinsteinMetric}
		\begin{split}
			\partial\conj\partial\log J(r)
			&=(n+2)(n+1)r^nf\partial r\wedge\conj\partial r
			+(n+2)r^{n+1}
				(f\partial\conj\partial r+\partial f\wedge\conj\partial r+\partial r\wedge\conj\partial f)\\
			&\phantom{=\;}
			+r^{n+2}\partial\conj\partial f.
		\end{split}
	\end{equation}

	We use the same notation as in Example \ref{ex:CompleteKahlerMetric}.
	Since $\kappa$ is a real-valued function,
	$\tau\intprod\partial\conj\partial r=-i(\xi-\conj\xi)\intprod\partial\conj\partial r
	=-i(\kappa\conj\partial r+\kappa\partial r)=-i\kappa dr$.
	Therefore $T\intprod d\theta=T\intprod\iota^*(-i\partial\conj\partial r)=\iota^*(-\kappa dr)=0$,
	where $T$ is the restriction of $\tau$ to $M$.
	This shows that $T$ is the Reeb vector field on $M$ associated with $\theta$.
	By restricting $\xi_1$, $\dots$, $\xi_n$ to $M$,
	we obtain a local frame $\set{Z_1,\dots,Z_n}$ of $T^{1,0}M$.

	We identify a (one-sided) neighborhood of $M$ in $\closure\Omega$ with $M\times[0,\epsilon)$ by
	\begin{equation*}
		M\times[0,\epsilon)\to\closure\Omega,\qquad
		(p,s)\mapsto\Fl_s(p),
	\end{equation*}
	where $\Fl_s$ is the flow generated by $\nu$. In view of the fact that $s$ is equal to the pullback of $r$,
	we write $r$ instead of $s$ in the sequel.
	The constant extensions of $T$ and $Z_\alpha$ in the $r$-direction are also denoted by $T$ and $Z_\alpha$.
	Then obviously $T=\tau+O(r)$, $Z_\alpha=\xi_\alpha+O(r)$.
	By \eqref{eq:HighOrderTermInApproximateKahlerEinsteinMetric} we have
	\begin{align*}
		\Ein(G)(\nu,\nu)
		&=\Ein(G)\left(\tfrac{1}{2}(\xi+\conj\xi),\tfrac{1}{2}(\xi+\conj\xi)\right)
		=\tfrac{1}{2}\Ein(G)(\xi,\conj\xi)=O(r^{n}),\\
		\Ein(G)(\tau,\tau)
		&=\Ein(G)(-i(\xi-\conj\xi),-i(\xi-\conj\xi))=2\Ein(G)(\xi,\conj\xi)=O(r^{n}),\\
		\Ein(G)(\nu,\tau)&=0,\qquad
		\Ein(G)(\nu,\xi_\alpha)=O(r^{n+1}),\qquad
		\Ein(G)(\tau,\xi_\alpha)=O(r^{n+1}),\\
		\Ein(G)(\xi_\alpha,\xi_\conjbeta)&=O(r^{n+1}),\qquad
		\Ein(G)(\xi_\alpha,\xi_\beta)=0.
	\end{align*}
	Hence, with respect to the local frame $\set{\partial_r=\nu,T,Z_\alpha,Z_\conjalpha}$ of
	$T_\bbC(M\times[0,\epsilon))$, we have
	\begin{align*}
		\tensor{\Ein(G)}{_\infty_\infty}&=O(r^{n}),\qquad
		\tensor{\Ein(G)}{_\infty_0}=O(r^{n+1}),\qquad
		\tensor{\Ein(G)}{_\infty_\alpha}=O(r^{n+1}),\\
		\tensor{\Ein(G)}{_0_0}&=O(r^n),\qquad
		\tensor{\Ein(G)}{_0_\alpha}=O(r^{n+1}),\\
		\tensor{\Ein(G)}{_\alpha_\conjbeta}&=O(r^{n+1}),\qquad
		\tensor{\Ein(G)}{_\alpha_\beta}=O(r^{n+2}).
	\end{align*}
	Therefore the Einstein tensor of the induced ACH metric $g$ on the square root of $M\times[0,\epsilon)$
	in the sense of \cite{EpsteinMelroseMendoza}
	satisfies, with respect to the frame $\set{\rho\partial_\rho,\rho^2T,\rho Z_\alpha,\rho Z_\conjalpha}$,
	\begin{align*}
		\tensor{\Ein}{_\infty_\infty}&=O(\rho^{2n+4}),\qquad
		\tensor{\Ein}{_\infty_0}=O(\rho^{2n+6}),\qquad
		\tensor{\Ein}{_\infty_\alpha}=O(\rho^{2n+5}),\\
		\tensor{\Ein}{_0_0}&=O(\rho^{2n+4}),\qquad
		\tensor{\Ein}{_0_\alpha}=O(\rho^{2n+5}),\\
		\tensor{\Ein}{_\alpha_\conjbeta}&=O(\rho^{2n+4}),\qquad
		\tensor{\Ein}{_\alpha_\beta}=O(\rho^{2n+6}).
	\end{align*}
	Hence $g$ satisfies \eqref{eq:ApproximatelyEinstein}.
	Moreover, since $\tensor{\Ein}{_\alpha_\beta}=O(\rho^{2n+3})$,
	it follows that $\tensor{\caO}{_\alpha_\beta}=0$.
\end{proof}

\section{On the first variation the CR obstruction tensor}
\label{sec:FirstVariation}

In this section, we calculate the first-order term of the obstruction tensor with respect to a variation
from the standard CR sphere.
First we introduce a tensor that describes a modification of partially integrable almost CR structures.

\begin{prop}
	Let $(M,T^{1,0})$ be a nondegenerate partially integrable almost CR manifold and
	$\set{Z_\alpha}$ a local frame of the bundle $T^{1,0}$.
	Let $\tensor{\mu}{_\alpha^\conjbeta}\in\tensor{\caE}{_\alpha^\conjbeta}$ and set
	\begin{equation*}
		\Hat{Z}_\alpha:=Z_\alpha+\tensor{\mu}{_\alpha^\conjbeta}Z_\conjbeta;
	\end{equation*}
	$\set{\Hat{Z}_\alpha}$ defines a new almost CR structure on $M$ without changing the contact distribution $H$.
	Then this is partially integrable if and only if
	\begin{equation*}
		\tensor{\mu}{_\alpha_\beta}=\tensor{\mu}{_\beta_\alpha},
	\end{equation*}
	where the upper index is lowered by the Levi form of $(M,T^{1,0})$ associated to any pseudohermitian structure.
\end{prop}

\begin{proof}
	The new almost CR structure is partially integrable if and only if
	\begin{equation*}
		\theta([\Hat{Z}_\alpha,\Hat{Z}_\beta])
		=\theta([Z_\alpha+\tensor{\mu}{_\alpha^\conjsigma}Z_\conjsigma,
		Z_\beta+\tensor{\mu}{_\beta^\conjtau}Z_\conjtau])=0,
	\end{equation*}
	where $\theta$ is any pseudohermitian structure for $(M, T^{1,0})$.
	Since $\theta([Z_\alpha,Z_\beta])=\theta([Z_\conjsigma,Z_\conjtau])=0$,
	this is equivalent to
	\begin{equation*}
		\theta([Z_\conjsigma,Z_\beta])\tensor{\mu}{_\alpha^\conjsigma}
		+\theta([Z_\alpha,Z_\conjtau])\tensor{\mu}{_\beta^\conjtau}=0,
	\end{equation*}
	or $\tensor{\mu}{_\alpha_\beta}-\tensor{\mu}{_\beta_\alpha}=0$.
\end{proof}

Let $M=S^{2n+1}$ be the $(2n+1)$-dimensional sphere and $\theta$ the standard contact form.
Then the obstruction tensor $\tensor{\caO}{_\alpha_\beta}$ with respect to $\theta$
is a function of partially integrable almost CR structures on $\ker\theta$.
For the standard CR structure we have $\tensor{\caO}{_\alpha_\beta}=0$.
We shall compute the derivative of $\tensor{\caO}{_\alpha_\beta}$
at the standard CR structure in the direction of $\tensor{\mu}{_\alpha_\beta}$,
where the second index of $\tensor{\mu}{_\alpha_\beta}$ is understood to be lowered by
the Levi form of the standard CR sphere associated to $\theta$.
The differentials of various quantities at the standard CR structure will be indicated by the bullet $\bullet$.

\begin{prop}
	\label{prop:DerivativeOfPseudohermitianInvariants}
	Consider $\tensor{h}{_\alpha_\conjbeta}$, $\tensor{N}{_\alpha_\beta_\gamma}$,
	$\tensor{A}{_\alpha_\beta}$ and $\tensor{R}{_\alpha_\conjbeta}$
	associated to the standard contact form $\theta$ on the sphere.
	Then, their differentials at the standard CR structure are as follows:
	\begin{align*}
		\tensor*{h}{^\bullet_\alpha_\conjbeta}&=0,\qquad
		\tensor*{N}{^\bullet_\alpha_\beta_\gamma}
		=\tensor{\nabla}{_\alpha}\tensor{\mu}{_\beta_\gamma}-\tensor{\nabla}{_\beta}\tensor{\mu}{_\alpha_\gamma},\\
		\tensor*{A}{^\bullet_\alpha_\beta}&=-\tensor{\nabla}{_0}\tensor{\mu}{_\alpha_\beta},\qquad
		\tensor*{R}{^\bullet_\alpha_\conjbeta}
		=-\tensor{\nabla}{_\alpha}\tensor{\nabla}{^\conjsigma}\tensor{\mu}{_\conjbeta_\conjsigma}
		-\tensor{\nabla}{_\conjbeta}\tensor{\nabla}{^\tau}\tensor{\mu}{_\alpha_\tau}.
	\end{align*}
\end{prop}

\begin{proof}
	Since the both sides of the four equalities are all tensorial, we may take any frame to derive them.
	Let $\set{Z_\alpha}$ be a local frame of $T^{1,0}$ of the standard CR sphere such that
	\begin{equation*}
		[Z_\alpha,Z_\conjbeta]=-i\tensor{h}{_\alpha_\conjbeta}T,\qquad
		[Z_\alpha,Z_\beta]=[Z_\alpha,T]=0
	\end{equation*}
	and
	\begin{equation*}
		\tensor{h}{_\alpha_\conjbeta}=
		\begin{cases}
			1,&\text{if $\alpha=\beta$},\\
			0,&\text{otherwise},
		\end{cases}
	\end{equation*}
	where $T$ is the Reeb vector field associated with $\theta$. Then the differentials of the Lie brackets are
	given by
	\begin{align*}
		[\Hat{Z}_\alpha,\Hat{Z}_\conjbeta]^\bullet
		&=(\tensor{\nabla}{_\alpha}\tensor{\mu}{_\conjbeta^\sigma})Z_\sigma
		-(\tensor{\nabla}{_\conjbeta}\tensor{\mu}{_\alpha^\conjtau})Z_\conjtau,\\
		[\Hat{Z}_\alpha,\Hat{Z}_\beta]^\bullet
		&=(\tensor{\nabla}{_\alpha}\tensor{\mu}{_\beta^\conjgamma}
		-\tensor{\nabla}{_\conjbeta}\tensor{\mu}{_\alpha^\conjgamma})Z_\conjgamma,\\
		[\Hat{Z}_\alpha,T]^\bullet
		&=-(\tensor{\nabla}{_0}\tensor{\mu}{_\alpha^\conjgamma})Z_\conjgamma.
	\end{align*}
	They immediately show that $\tensor*{h}{^\bullet_\alpha_\conjbeta}=0$ and
	$\tensor{N}{^\bullet_\alpha_\beta^\conjgamma}=\tensor{\nabla}{_\alpha}\tensor{\mu}{_\beta^\conjgamma}
	-\tensor{\nabla}{_\beta}\tensor{\mu}{_\alpha^\conjgamma}$.
	The first structure equation \eqref{eq:FirstStructureEquation2} implies
	\begin{equation*}
		\tensor{A}{^\bullet_\alpha^\conjbeta}
		=\theta^\conjbeta([\Hat{Z}_\alpha,T]^\bullet)=-\tensor{\nabla}{_0}\tensor{\mu}{_\alpha^\conjbeta}.
	\end{equation*}
	Similarly we have
	\begin{equation*}
		\tensor{\omega}{^\bullet_\alpha^\beta}(Z_\conjgamma)
		=-\tensor{\nabla}{_\alpha}\tensor{\mu}{_\conjgamma^\beta},\qquad
		\tensor{\omega}{^\bullet_\alpha^\beta}(T)=0,
	\end{equation*}
	and this together with
	$\tensor*{\omega}{^\bullet_\alpha_\conjbeta}+\tensor*{\omega}{^\bullet_\conjbeta_\alpha}
	=(d\tensor{h}{_\alpha_\conjbeta})^\bullet=0$ implies
	$\tensor{\omega}{^\bullet_\alpha^\beta}(Z_\gamma)=\tensor{\nabla}{^\beta}\tensor{\mu}{_\alpha_\gamma}$.
	From \eqref{eq:CurvatureForm} we have
	\begin{equation*}
		\tensor*{R}{^\bullet_\alpha_\conjbeta}
		=Z_\alpha\tensor{\omega}{^\bullet_\gamma^\gamma}(Z_\conjbeta)
		-Z_\conjbeta\tensor{\omega}{^\bullet_\gamma^\gamma}(Z_\alpha)
		-\tensor{\omega}{^\bullet_\gamma^\gamma}([Z_\alpha,Z_\conjbeta])
		=-\tensor{\nabla}{_\alpha}\tensor{\nabla}{^\conjgamma}\tensor{\mu}{_\conjbeta_\conjgamma}
		-\tensor{\nabla}{_\conjbeta}\tensor{\nabla}{^\gamma}\tensor{\mu}{_\alpha_\gamma}.
	\end{equation*}
	This completes the proof.
\end{proof}

Let $g$ be a normal-form ACH metric for $\theta$
satisfying the condition in Theorem \ref{thm:ApproximateACHEMetric}. Let
\begin{equation*}
	\tensor{g}{_0_0}=1+\tensor{\varphi}{_0_0},\qquad
	\tensor{g}{_0_\alpha}=\tensor{\varphi}{_0_\alpha},\qquad
	\tensor{g}{_\alpha_\conjbeta}=\tensor{h}{_\alpha_\conjbeta}+\tensor{\varphi}{_\alpha_\conjbeta},\qquad
	\tensor{g}{_\alpha_\beta}=\tensor{\varphi}{_\alpha_\beta}.
\end{equation*}
Then, as seen in Theorem \ref{thm:ApproximateACHEMetric},
\begin{equation*}
	\tensor{\varphi[m]}{_i_j}
	:=\dfrac{1}{m!}\left.\left(\partial_\rho^{m}\tensor{\varphi}{_i_j}\right)\right|_{\rho=0},\qquad
	m\le 2n+1+a(i,j)
\end{equation*}
are uniquely determined. For the standard CR structure they completely vanish.
We shall observe the differentials $\tensor*{\varphi[m]}{^\bullet_i_j}$ of $\tensor{\varphi[m]}{_i_j}$.
For notational convenience, we set $\tensor{\varphi[m]}{_i_j}:=0$ for $m\le 0$ and
\begin{equation*}
	\chi_k(m):=
	\begin{cases}
		1,& m=k,\\
		0,& \text{otherwise}.
	\end{cases}
\end{equation*}

\begin{lem}
	\label{lem:RecursiveFormulaForFirstVariation}
	The differentials $\tensor*{\varphi[m]}{^\bullet_i_j}$ of $\tensor{\varphi[m]}{_i_j}$ at the standard CR
	structure satisfy
	\allowdisplaybreaks
	\begin{align*}
		\begin{split}
			0
			&=-\tfrac{1}{8}\left(m^2-(2n+4)m-4n\right)\tensor*{\varphi[m]}{^\bullet_0_0}
			+\tfrac{1}{2}(m-2)\tensor{\varphi[m]}{^\bullet_\alpha^\alpha}\\
			&\phantom{=\;}
			+i(\tensor{\nabla}{^\alpha}\tensor*{\varphi[m-1]}{^\bullet_0_\alpha}
			-\tensor{\nabla}{^\conjalpha}\tensor*{\varphi[m-1]}{^\bullet_0_\conjalpha})
			+\tfrac{1}{2}\sublaplacian\tensor*{\varphi[m-2]}{^\bullet_0_0}\\
			&\phantom{=\;}
			+(\tensor{\nabla}{_0}\tensor{\nabla}{^\alpha}\tensor*{\varphi[m-3]}{^\bullet_0_\alpha}
			+\tensor{\nabla}{_0}\tensor{\nabla}{^\conjalpha}\tensor*{\varphi[m-3]}{^\bullet_0_\conjalpha})
			-\tensor{\nabla}{_0}\tensor{\nabla}{_0}\tensor{\varphi[m-4]}{^\bullet_\alpha^\alpha},
		\end{split}\\
		\begin{split}
			0
			&=-\chi_3(m)\tensor{\nabla}{_0}\tensor{\nabla}{^\beta}\tensor{\mu}{_\alpha_\beta}
			-\tfrac{1}{8}(m+1)(m-2n-3)\tensor*{\varphi[m]}{^\bullet_0_\alpha}\\
			&\phantom{=\;}
				+\tfrac{3i}{4}\tensor{\nabla}{_\alpha}\tensor*{\varphi[m-1]}{^\bullet_0_0}
				+\tfrac{i}{2}\tensor{\nabla}{_\alpha}\tensor{\varphi[m-1]}{^\bullet_\beta^\beta}
				-i\tensor{\nabla}{^\conjbeta}\tensor*{\varphi[m-1]}{^\bullet_\alpha_\conjbeta}\\
			&\phantom{=\;}
				+\tfrac{1}{2}\sublaplacian\tensor*{\varphi[m-2]}{^\bullet_0_\alpha}
				-\tfrac{i}{2}\tensor{\nabla}{_0}\tensor*{\varphi[m-2]}{^\bullet_0_\alpha}
				+\tfrac{1}{2}
					(\tensor{\nabla}{_\alpha}\tensor{\nabla}{^\beta}\tensor*{\varphi[m-2]}{^\bullet_0_\beta}
					+\tensor{\nabla}{_\alpha}\tensor{\nabla}{^\conjbeta}
					\tensor*{\varphi[m-2]}{^\bullet_0_\conjbeta})\\
			&\phantom{=\;}
				-\tensor{\nabla}{_0}\tensor{\nabla}{_\alpha}\tensor{\varphi[m-3]}{^\bullet_\beta^\beta}
				+\tfrac{1}{2}(\tensor{\nabla}{_0}\tensor{\nabla}{^\conjbeta}
					\tensor*{\varphi[m-3]}{^\bullet_\alpha_\conjbeta}
					+\tensor{\nabla}{_0}\tensor{\nabla}{^\beta}\tensor*{\varphi[m-3]}{^\bullet_\alpha_\beta}),
		\end{split}\\
		\begin{split}
			0
			&=-\chi_2(m)(\tensor{\nabla}{_\alpha}\tensor{\nabla}{^\conjgamma}\tensor{\mu}{_\conjbeta_\conjgamma}
				+\tensor{\nabla}{_\conjbeta}\tensor{\nabla}{^\gamma}\tensor{\mu}{_\alpha_\gamma})
			-\tfrac{1}{8}\left(m^2-(2n+2)m-8\right)\tensor*{\varphi[m]}{^\bullet_\alpha_\conjbeta}\\
			&\phantom{=\;}
			+\tfrac{1}{8}\tensor{h}{_\alpha_\conjbeta}(m-4)\tensor*{\varphi[m]}{^\bullet_0_0}
			+\tfrac{1}{4}\tensor{h}{_\alpha_\conjbeta}m\tensor{\varphi[m]}{^\bullet_\gamma^\gamma}\\
			&\phantom{=\;}
			+i(\tensor{\nabla}{_\alpha}\tensor*{\varphi[m-1]}{^\bullet_0_\conjbeta}
				-\tensor{\nabla}{_\conjbeta}\tensor*{\varphi[m-1]}{^\bullet_0_\alpha})
			-\tfrac{i}{4}\tensor{h}{_\alpha_\conjbeta}\tensor{\nabla}{_0}\tensor*{\varphi[m-2]}{^\bullet_0_0}
			-\tfrac{i}{2}\tensor{h}{_\alpha_\conjbeta}\tensor{\nabla}{_0}
				\tensor{\varphi[m-2]}{^\bullet_\gamma^\gamma}\\
			&\phantom{=\;}
			-\tfrac{1}{2}\tensor{\nabla}{_\alpha}\tensor{\nabla}{_\conjbeta}\tensor*{\varphi[m-2]}{^\bullet_0_0}
			-\tensor{\nabla}{_\alpha}\tensor{\nabla}{_\conjbeta}\tensor{\varphi[m-2]}{^\bullet_\gamma^\gamma}
			+\tfrac{1}{2}\sublaplacian\tensor*{\varphi[m-2]}{^\bullet_\alpha_\conjbeta}\\
			&\phantom{=\;}
			+\tfrac{1}{2}(
				\tensor{\nabla}{_\alpha}\tensor{\nabla}{^\gamma}\tensor*{\varphi[m-2]}{^\bullet_\conjbeta_\gamma}
				+\tensor{\nabla}{_\alpha}\tensor{\nabla}{^\conjgamma}
					\tensor*{\varphi[m-2]}{^\bullet_\conjbeta_\conjgamma}
				+\tensor{\nabla}{_\conjbeta}\tensor{\nabla}{^\conjgamma}
					\tensor*{\varphi[m-2]}{^\bullet_\alpha_\conjgamma}
				+\tensor{\nabla}{_\conjbeta}\tensor{\nabla}{^\gamma}
					\tensor*{\varphi[m-2]}{^\bullet_\alpha_\gamma})\\
			&\phantom{=\;}
			+\tfrac{1}{2}(\tensor{\nabla}{_0}\tensor{\nabla}{_\alpha}\tensor*{\varphi[m-3]}{^\bullet_0_\conjbeta}
				+\tensor{\nabla}{_0}\tensor{\nabla}{_\conjbeta}\tensor*{\varphi[m-3]}{^\bullet_0_\alpha})
			-\tfrac{1}{2}\tensor{\nabla}{_0}\tensor{\nabla}{_0}\tensor*{\varphi[m-4]}{^\bullet_\alpha_\conjbeta},
		\end{split}\\
		\begin{split}
			0
			&=-\chi_2(m)(\sublaplacian\tensor{\mu}{_\alpha_\beta}
				+\tensor{\nabla}{_\alpha}\tensor{\nabla}{^\gamma}\tensor{\mu}{_\beta_\gamma}
				+\tensor{\nabla}{_\beta}\tensor{\nabla}{^\gamma}\tensor{\mu}{_\alpha_\gamma}
				+2i\tensor{\nabla}{_0}\tensor{\mu}{_\alpha_\beta})
			+\chi_4(m)\tensor{\nabla}{_0}\tensor{\nabla}{_0}\tensor{\mu}{_\alpha_\beta}\\
			&\phantom{=\;}
			-\tfrac{1}{8}m(m-2n-2)\tensor*{\varphi[m]}{^\bullet_\alpha_\beta}
			-\tfrac{1}{2}\tensor{\nabla}{_\alpha}\tensor{\nabla}{_\beta}\tensor*{\varphi[m-2]}{^\bullet_0_0}
			-\tensor{\nabla}{_\alpha}\tensor{\nabla}{_\beta}\tensor{\varphi[m-2]}{^\bullet_\gamma^\gamma}
			+\tfrac{1}{2}\sublaplacian\tensor*{\varphi[m-2]}{^\bullet_\alpha_\beta}\\
			&\phantom{=\;}
			+\tfrac{1}{2}
				(\tensor{\nabla}{_\alpha}\tensor{\nabla}{^\conjgamma}
					\tensor*{\varphi[m-2]}{^\bullet_\beta_\conjgamma}
				+\tensor{\nabla}{_\alpha}\tensor{\nabla}{^\gamma}\tensor*{\varphi[m-2]}{^\bullet_\beta_\gamma}
				+\tensor{\nabla}{_\beta}\tensor{\nabla}{^\conjgamma}
					\tensor*{\varphi[m-2]}{^\bullet_\alpha_\conjgamma}
				+\tensor{\nabla}{_\beta}\tensor{\nabla}{^\gamma}\tensor*{\varphi[m-2]}{^\bullet_\alpha_\gamma})\\
			&\phantom{=\;}
			+i\tensor{\nabla}{_0}\tensor*{\varphi[m-2]}{^\bullet_\alpha_\beta}
			+\tfrac{1}{2}(\tensor{\nabla}{_0}\tensor{\nabla}{_\alpha}\tensor*{\varphi[m-3]}{^\bullet_0_\beta}
				+\tensor{\nabla}{_0}\tensor{\nabla}{_\beta}\tensor*{\varphi[m-3]}{^\bullet_0_\alpha})
			-\tfrac{1}{2}\tensor{\nabla}{_0}\tensor{\nabla}{_0}\tensor*{\varphi[m-4]}{^\bullet_\alpha_\beta},
		\end{split}
	\end{align*}
	\allowdisplaybreaks[0]%
	where in each equality $m$ takes any nonnegative integer and
	$\nabla$ denotes the Tanaka--Webster connection for the standard CR sphere with $\theta$.
\end{lem}

\begin{proof}
	This follows from Lemma \ref{lem:RicciTensorModuloHighOrderTerm}, because terms of type (N1)--(N3),
	which are neglected in the formulae recorded in that lemma, are at least quadratic in
	$\tensor{\mu}{_\alpha_\beta}$.
	By setting $\tensor{\Ein}{_I_J}=O(\rho^{2n+1+a(I,J)})$, the Taylor expansions of the last four equalities
	in Lemma \ref{lem:RicciTensorModuloHighOrderTerm} give the claimed formulae,
	thanks to Proposition \ref{prop:DerivativeOfPseudohermitianInvariants}.
\end{proof}

In principle we can calculate all $\tensor*{\varphi[m]}{^\bullet_i_j}$ using the recurrence formulae above.
It is easy to see that
$\tensor*{\varphi[m]}{^\bullet_0_0}=\tensor*{\varphi[m]}{^\bullet_\alpha_\conjbeta}
=\tensor*{\varphi[m]}{^\bullet_\alpha_\beta}=0$ for $m$ odd and $\tensor*{\varphi[m]}{^\bullet_0_\alpha}=0$ for
$m$ even, and each nonzero $\tensor*{\varphi[m]}{^\bullet_i_j}$ is a linear combination over $\bbC$ of
covariant derivatives of $\tensor{\mu}{_\alpha_\beta}$ which are given in
Table \ref{tbl:DerivativeOfApproximateMetric}.
As a result the differential $\tensor*{\caO}{^\bullet_\alpha_\beta}$ of the obstruction tensor is a linear
combination of
\begin{align*}
	&\sublaplacian^k\tensor*{\nabla}{^{n+1-k}_0}\tensor{\mu}{_\alpha_\beta},\qquad
	\sublaplacian^k\tensor*{\nabla}{^{n-k}_0}\tensor{\nabla}{_(_\alpha}\tensor{\nabla}{^\sigma}
	\tensor{\mu}{_\beta_)_\sigma},\\
	&\sublaplacian^k\tensor*{\nabla}{^{n-1-k}_0}
	\tensor{\nabla}{_\alpha}\tensor{\nabla}{_\beta}\tensor{\nabla}{^\sigma}\tensor{\nabla}{^\tau}
	\tensor{\mu}{_\sigma_\tau}\quad\text{and}\quad
	\sublaplacian^k\tensor*{\nabla}{^{n-1-k}_0}
	\tensor{\nabla}{_\alpha}\tensor{\nabla}{_\beta}\tensor{\nabla}{^\conjsigma}\tensor{\nabla}{^\conjtau}
	\tensor{\mu}{_\conjsigma_\conjtau},
\end{align*}
which are linearly independent if $n\ge 2$.

\begin{table}[htbp]
	\small
\begin{tabular}{cl}
	\toprule
	Type & \multicolumn{1}{c}{Terms} \\ \midrule
	\rule[-0.7em]{0pt}{1.8em}$\tensor*{\varphi[2l]}{^\bullet_0_0}$&
	$\sublaplacian^k\tensor*{\nabla}{^{l-1-k}_0}
	\tensor{\nabla}{^\alpha}\tensor{\nabla}{^\beta}\tensor{\mu}{_\alpha_\beta}$,\quad
	$\sublaplacian^k\tensor*{\nabla}{^{l-1-k}_0}
	\tensor{\nabla}{^\conjalpha}\tensor{\nabla}{^\conjbeta}\tensor{\mu}{_\conjalpha_\conjbeta}$\\ \midrule
	\rule[-0.7em]{0pt}{1.8em}$\tensor*{\varphi[2l+1]}{^\bullet_0_\alpha}$&
	$\sublaplacian^k\tensor*{\nabla}{^{l-k}_0}\tensor{\nabla}{^\beta}\tensor{\mu}{_\alpha_\beta}$,\quad
	$\sublaplacian^k\tensor*{\nabla}{^{l-1-k}_0}
	\tensor{\nabla}{_\alpha}\tensor{\nabla}{^\sigma}\tensor{\nabla}{^\tau}\tensor{\mu}{_\sigma_\tau}$,\quad
	$\sublaplacian^k\tensor*{\nabla}{^{l-1-k}_0}
	\tensor{\nabla}{_\alpha}\tensor{\nabla}{^\conjsigma}\tensor{\nabla}{^\conjtau}
	\tensor{\mu}{_\conjsigma_\conjtau}$\\ \midrule
	\rule[-0.6em]{0pt}{1.7em}$\tensor*{\varphi[2l]}{^\bullet_\alpha_\conjbeta}$&
	$\sublaplacian^k\tensor*{\nabla}{^{l-1-k}_0}
	\tensor{\nabla}{_\alpha}\tensor{\nabla}{^\conjsigma}\tensor{\mu}{_\conjbeta_\conjsigma}$,\quad
	$\sublaplacian^k\tensor*{\nabla}{^{l-1-k}_0}
	\tensor{\nabla}{_\conjbeta}\tensor{\nabla}{^\sigma}\tensor{\mu}{_\alpha_\sigma}$,\\
	\rule[-0.6em]{0pt}{1.6em}&
	$\sublaplacian^k\tensor*{\nabla}{^{l-2-k}_0}
	\tensor{\nabla}{_\alpha}\tensor{\nabla}{_\conjbeta}\tensor{\nabla}{^\sigma}\tensor{\nabla}{^\tau}
	\tensor{\mu}{_\sigma_\tau}$,\quad
	$\sublaplacian^k\tensor*{\nabla}{^{l-2-k}_0}
	\tensor{\nabla}{_\conjbeta}\tensor{\nabla}{_\alpha}\tensor{\nabla}{^\conjsigma}\tensor{\nabla}{^\conjtau}
	\tensor{\mu}{_\conjsigma_\conjtau}$,\\
	\rule[-0.7em]{0pt}{1.8em}& $\tensor{h}{_\alpha_\conjbeta}\sublaplacian^k\tensor*{\nabla}{^{l-1-k}_0}
	\tensor{\nabla}{^\sigma}\tensor{\nabla}{^\tau}\tensor{\mu}{_\sigma_\tau}$,\quad
	$\tensor{h}{_\alpha_\conjbeta}\sublaplacian^k\tensor*{\nabla}{^{l-1-k}_0}
	\tensor{\nabla}{^\conjsigma}\tensor{\nabla}{^\conjtau}\tensor{\mu}{_\conjsigma_\conjtau}$\\ \midrule
	\rule[-0.6em]{0pt}{1.7em}$\tensor*{\varphi[2l]}{^\bullet_\alpha_\beta}$&
	$\sublaplacian^k\tensor*{\nabla}{^{l-k}_0}\tensor{\mu}{_\alpha_\beta}$,\quad
	$\sublaplacian^k\tensor*{\nabla}{^{l-1-k}_0}
	\tensor{\nabla}{_(_\alpha}\tensor{\nabla}{^\sigma}\tensor{\mu}{_\beta_)_\sigma}$,\\
	\rule[-0.7em]{0pt}{1.7em}& $\sublaplacian^k\tensor*{\nabla}{^{l-2-k}_0}
	\tensor{\nabla}{_\alpha}\tensor{\nabla}{_\beta}\tensor{\nabla}{^\sigma}\tensor{\nabla}{^\tau}
	\tensor{\mu}{_\sigma_\tau}$,\quad
	$\sublaplacian^k\tensor*{\nabla}{^{l-2-k}_0}
	\tensor{\nabla}{_\alpha}\tensor{\nabla}{_\beta}\tensor{\nabla}{^\conjsigma}\tensor{\nabla}{^\conjtau}
	\tensor{\mu}{_\conjsigma_\conjtau}$\\ \bottomrule
\end{tabular}
	\vspace{0.5em}
	\caption{Terms appearing in the differentials $\tensor*{\varphi[m]}{^\bullet_i_j}$ of
	the coefficients of the approximate normal-form ACH-Einstein metric}
	\label{tbl:DerivativeOfApproximateMetric}
\end{table}

\begin{prop}
	Let $n\ge 2$ and
	\begin{equation*}
		\begin{split}
			\tensor*{\caO}{^\bullet_\alpha_\beta}
			&=\sum_{k=0}^{n+1}a_k\sublaplacian^k\tensor*{\nabla}{^{n+1-k}_0}\tensor{\mu}{_\alpha_\beta}
			+\sum_{k=0}^{n}b_k
			\sublaplacian^k\tensor*{\nabla}{^{n-k}_0}\tensor{\nabla}{_(_\alpha}\tensor{\nabla}{^\sigma}
			\tensor{\mu}{_\beta_)_\sigma}\\
			&\phantom{=\;}
			+\sum_{k=0}^{n-1}c_k\sublaplacian^k\tensor*{\nabla}{^{n-1-k}_0}
			\tensor{\nabla}{_\alpha}\tensor{\nabla}{_\beta}\tensor{\nabla}{^\sigma}\tensor{\nabla}{^\tau}
			\tensor{\mu}{_\sigma_\tau}
			+\sum_{k=0}^{n-1}d_k\sublaplacian^k\tensor*{\nabla}{^{n-1-k}_0}
			\tensor{\nabla}{_\alpha}\tensor{\nabla}{_\beta}\tensor{\nabla}{^\conjsigma}\tensor{\nabla}{^\conjtau}
			\tensor{\mu}{_\conjsigma_\conjtau}.
		\end{split}
	\end{equation*}
	Then $a_{n+1}=(-1)^{n}/(n!)^2$.
\end{prop}

\begin{proof}
	The last equality in Lemma \ref{lem:RecursiveFormulaForFirstVariation} and
	Table \ref{tbl:DerivativeOfApproximateMetric} show
	\begin{equation*}
		0\equiv -\chi_2(2l)\sublaplacian\tensor{\mu}{_\alpha_\beta}
		-\tfrac{1}{2}l(l-n-1)\tensor*{\varphi[2l]}{^\bullet_\alpha_\beta}
		+\tfrac{1}{2}\sublaplacian\tensor*{\varphi[2l-2]}{^\bullet_\alpha_\beta}
	\end{equation*}
	modulo $\sublaplacian^k\tensor*{\nabla}{^{l-k}_0}\tensor{\mu}{_\alpha_\beta}$, $k<l$, and
	\begin{equation*}
		\sublaplacian^k\tensor*{\nabla}{^{l-1-k}_0}\tensor{\nabla}{_(_\alpha}\tensor{\nabla}{^\sigma}
		\tensor{\mu}{_\beta_)_\sigma},\qquad
		\sublaplacian^k\tensor*{\nabla}{^{l-2-k}_0}
		\tensor{\nabla}{_\alpha}\tensor{\nabla}{_\beta}\tensor{\nabla}{^\sigma}\tensor{\nabla}{^\tau}
		\tensor{\mu}{_\sigma_\tau},\qquad
		\sublaplacian^k\tensor*{\nabla}{^{l-2-k}_0}
		\tensor{\nabla}{_\alpha}\tensor{\nabla}{_\beta}\tensor{\nabla}{^\conjsigma}\tensor{\nabla}{^\conjtau}
		\tensor{\mu}{_\conjsigma_\conjtau}.
	\end{equation*}
	Hence we have $\tensor*{\varphi[2]}{^\bullet_\alpha_\beta}\equiv(2/n)\sublaplacian\tensor{\mu}{_\alpha_\beta}$
	and
	\begin{equation*}
		\tensor*{\varphi[2l]}{^\bullet_\alpha_\beta}\equiv
		-\dfrac{1}{l(n+1-l)}\sublaplacian\tensor*{\varphi[2l-2]}{^\bullet_\alpha_\beta}.
	\end{equation*}
	This immediately shows that
	\begin{equation*}
		\tensor*{\varphi[2l]}{^\bullet_\alpha_\beta}\equiv
		\dfrac{2}{n}\cdot\dfrac{-1}{2(n-1)}\cdot\dfrac{-1}{3(n-2)}\cdot\dots\cdot
		\dfrac{-1}{l(n+1-l)}\sublaplacian^l\tensor{\mu}{_\alpha_\beta},\qquad
		\text{$l=1$, $2$, $\dots$, $n$}.
	\end{equation*}
	Then we use the last equality in Lemma \ref{lem:RicciTensorModuloHighOrderTerm} to see
	\begin{equation*}
		\tensor*{\caO}{^\bullet_\alpha_\beta}
		\equiv -\dfrac{1}{2}\sublaplacian\tensor*{\varphi[2n]}{^\bullet_\alpha_\beta}
		\equiv -\dfrac{1}{2}\cdot\dfrac{2}{n}\cdot\dfrac{-1}{2(n-1)}\cdot\dfrac{-1}{3(n-2)}\cdot\dots\cdot
			\dfrac{-1}{n\cdot 1}\sublaplacian^{n+1}\tensor{\mu}{_\alpha_\beta}
		\equiv \dfrac{(-1)^n}{(n!)^2}\sublaplacian^{n+1}\tensor{\mu}{_\alpha_\beta},
	\end{equation*}
	which implies the claim.
\end{proof}

\begin{cor}
	Let $n\ge 2$.
	Then there is a partially integrable almost CR structure on the $(2n+1)$-dimensional sphere,
	arbitrarily close to the standard one, for which the obstruction tensor does not vanish.
\end{cor}

\section{Formal solutions involving logarithmic singularities}
\label{sec:FormalSolution}

Let $X$ be a manifold-with-boundary and $\rho$ a boundary defining function.
We say that a function $f\in C^0(X)\cap C^\infty(\interior{X})$ belongs to $\caA(X)$, or simply $\caA$,
if it admits an asymptotic expansion of the form \eqref{eq:GeneralLogarithmicExpansion}.
By this we mean that for any $m\ge 0$,
\begin{equation*}
	r_N:=f-\sum_{q=0}^N f^{(q)}(\log\rho)^q\in C^m(X)\quad\text{and}\quad r_N=O(\rho^m)
\end{equation*}
holds for sufficiently large $N$.
The Taylor expansions of $f^{(q)}$ at $\bdry X$ are uniquely determined;
we write $f\in\caA^m$ if $f^{(q)}=O(\rho^m)$, $q\ge 0$, and $\caA^\infty:=\cap_{m=0}^\infty\caA^m$.
The usage of the symbol $\caA^m$ is similar to that of $O(\rho^m)$;
for example, $f=f_0+\caA^m$ means that $f-f_0\in\caA^m$.
One can show that $\caA$ is closed under multiplication,
and that if $f\in\caA$ and $f$ is nonzero everywhere then $f^{-1}\in\caA$.
Furthermore, $\caA$ is closed under the actions of totally characteristic linear differential operators, i.e.,
noncommutative polynomials of smooth vector fields tangent to the boundary.

As in \S\S\ref{sec:RicciTensorAndSomeLowOrderTerms}--\ref{sec:ApproximateSolution},
again in this section $X$ is an open neighborhood of $M$ in $M\times[0,\infty)$,
where $(M,T^{1,0})$ is a nondegenerate partially integrable almost CR manifold.
We fix a pseudohermitian structure $\theta$ and
consider (nonsmooth) $[\Theta]$-metrics of the form \eqref{eq:ProductDecompositionOfACHMetric}
with $\tensor{g}{_i_j}\in\caA$ satisfying \eqref{eq:NormalFormConditionOnACHMetric},
which we call \emph{singular normal-form ACH metrics} for $(M,T^{1,0})$ and $\theta$.

All the calculations regarding the Ricci tensor go in the same way as in
\S\ref{sec:RicciTensorAndSomeLowOrderTerms} and \S\ref{sec:HigherOrderPerturbation} except that,
while on the space of smooth $O(\rho^m)$ functions $\euler$ behaves as
a mere ``$m$ times'' operator modulo $O(\rho^{m+1})$,
it is no longer the case when $O(\rho^m)$ and $O(\rho^{m+1})$ are replaced by $\caA^m$ and $\caA^{m+1}$.
Nevertheless, since $\caA$ is closed under the actions of totally characteristic operators,
the Ricci tensors for singular normal-form ACH metrics have expansions of
the form \eqref{eq:GeneralLogarithmicExpansion}
with respect to the frame $\set{\euler,\rho^2T,\rho Z_\alpha,\rho Z_\conjalpha}$.

\begin{prop}
	\label{prop:ApproximateSolutionForLogACH}
	There exists a singular normal-form ACH metric $g$ satisfying
	\begin{equation}
		\label{eq:EinsteinTensorOfSmoothACHAtCriticalOrder}
		\tensor{\Ein}{_I_J}=\caA^{2n+1+a(I,J)},
	\end{equation}
	where $a(I,J)$ is defined by \eqref{eq:DefinitionOfOrderModification}.
	The components $\tensor{g}{_i_j}$ are uniquely determined, and do not contain logarithmic terms,
	modulo $\caA^{2n+1+a(i,j)}$.
\end{prop}

\begin{proof}
	This can be proved by following the argument in \S\ref{sec:RicciTensorAndSomeLowOrderTerms},
	\S\ref{sec:HigherOrderPerturbation} and the first half of \S\ref{sec:ApproximateSolution} again.
	We shall include here a detailed account of the following fact only, which is a version of Proposition
	\ref{prop:ACHEToThirdOrder}: $\tensor{\Ein}{_I_J}=\caA^3$ if and only if
	\begin{equation*}
		\tensor{g}{_0_0}=1+\caA^3,\qquad
		\tensor{g}{_0_\alpha}=\caA^3,\qquad
		\tensor{g}{_\alpha_\conjbeta}=\tensor{h}{_\alpha_\conjbeta}+\rho^2\tensor{\Phi}{_\alpha_\conjbeta}
			+\caA^3,\qquad
		\tensor{g}{_\alpha_\beta}=\rho^2\tensor{\Phi}{_\alpha_\beta}+\caA^3,
	\end{equation*}
	where $\tensor{\Phi}{_\alpha_\conjbeta}$ and $\tensor{\Phi}{_\alpha_\beta}$ are defined by
	\eqref{eq:PhiTensorInLowestTerms}.
	Then the rest of the proof goes similarly.

	Let $g$ be given. If we define $\tensor{\varphi}{_i_j}$ by \eqref{eq:NonconstantTermsOfACHMetric},
	then Lemma \ref{lem:RicciTensorModuloHighOrderTerm} is again valid.
	Take $N\ge 1$ large enough so that $\tensor{\varphi}{_i_j}$ and
	$\tensor{\Ein}{_I_J}$ for given $g$ are of the form
	\begin{equation*}
		\tensor{\varphi}{_i_j}=\sum_{q=0}^{N}\tensor*{\varphi}{^{(q)}_i_j}(\log\rho)^q+\caA^3,
		\qquad \tensor*{\varphi}{^{(q)}_i_j}\in C^\infty(X),
	\end{equation*}
	and
	\begin{equation*}
		\tensor{\Ein}{_I_J}=\sum_{q=0}^{N}\tensor*{\Ein}{^{(q)}_I_J}(\log\rho)^{q}+\caA^3,\qquad
		\tensor*{\Ein}{^{(q)}_I_J}\in C^\infty(X).
	\end{equation*}
	Then by Lemma \ref{lem:RicciTensorModuloHighOrderTerm}
	we have the same identities as \eqref{eq:FirstOrderTermOfACHMetric} between
	$\tensor*{\Ein}{^{(N)}_I_J}$ and $\tensor*{\varphi}{^{(N)}_i_j}$; namely,
	the following holds for $q=N$:
	\allowdisplaybreaks
	\begin{align*}
		\tensor*{\Ein}{^{(q)}_\infty_\infty}
		&=\tfrac{3}{2}\tensor*{\varphi}{^{(q)}_0_0}+\tensor{\varphi}{^{(q)}_\alpha^\alpha}+O(\rho^2),\\
		\tensor*{\Ein}{^{(q)}_\infty_0}&=O(\rho^2),\qquad
		\tensor*{\Ein}{^{(q)}_\infty_\alpha}=-i\tensor*{\varphi}{^{(q)}_0_\alpha}+O(\rho^2),\\
		\tensor*{\Ein}{^{(q)}_0_0}
		&=\tfrac{3}{8}(2n+1)\tensor*{\varphi}{^{(q)}_0_0}
		-\tfrac{1}{2}\tensor{\varphi}{^{(q)}_\alpha^\alpha}+O(\rho^2),\qquad
		\tensor*{\Ein}{^{(q)}_0_\alpha}=\tfrac{1}{2}(n+1)\tensor*{\varphi}{^{(q)}_0_\alpha}+O(\rho^2),\\
		\tensor*{\Ein}{^{(q)}_\alpha_\conjbeta}
		&=\tfrac{1}{8}(2n+9)\tensor*{\varphi}{^{(q)}_\alpha_\conjbeta}
		-\tfrac{3}{8}\tensor{h}{_\alpha_\conjbeta}\tensor*{\varphi}{^{(q)}_0_0}
		+\tfrac{1}{4}\tensor{h}{_\alpha_\conjbeta}\tensor{\varphi}{^{(q)}_\gamma^\gamma}+O(\rho^2),\\
		\tensor*{\Ein}{^{(q)}_\alpha_\beta}&=\tfrac{1}{8}(2n+1)\tensor*{\varphi}{^{(q)}_\alpha_\beta}+O(\rho^2).
	\end{align*}
	\allowdisplaybreaks[0]%
	Hence $\tensor*{\varphi}{^{(N)}_i_j}$ must be $O(\rho^2)$ so as to make $\tensor*{\Ein}{^{(N)}_I_J}=O(\rho^2)$.
	If $\tensor*{\varphi}{^{(q)}_i_j}=O(\rho^2)$, $q_0+1\le q\le N$, then the identities above hold for $q=q_0$,
	which shows that $\tensor*{\Ein}{^{(q_0)}_I_J}=O(\rho^2)$ is equivalent to
	$\tensor*{\varphi}{^{(q_0)}_i_j}=O(\rho^2)$.
	Hence we conclude that $\tensor{\Ein}{_I_J}=\caA^2$ if and only if $\tensor{\varphi}{_i_j}=\caA^2$.

	Next, again by Lemma \ref{lem:RicciTensorModuloHighOrderTerm} we see that the following is true for $q=N$:
	\begin{align*}
		\tensor*{\Ein}{^{(q)}_\infty_\infty}&=2\tensor*{\varphi}{^{(q)}_0_0}+O(\rho^3),\qquad
		\tensor*{\Ein}{^{(q)}_\infty_0}=O(\rho^3),\qquad
		\tensor*{\Ein}{^{(q)}_\infty_\alpha}=-\tfrac{3}{2}i\tensor*{\varphi}{^{(q)}_0_\alpha}+O(\rho^3),\\
		\tensor*{\Ein}{^{(q)}_0_0}&=\tfrac{1}{2}(2n+1)\tensor*{\varphi}{^{(q)}_0_0}+O(\rho^3),\qquad
		\tensor*{\Ein}{^{(q)}_0_\alpha}=\tfrac{3}{8}(2n+1)\tensor*{\varphi}{^{(q)}_0_\alpha}+O(\rho^3),\\
		\tensor*{\Ein}{^{(q)}_\alpha_\conjbeta}&=
			\tfrac{1}{2}(n+2)\tensor*{\varphi}{^{(q)}_\alpha_\conjbeta}
			-\tfrac{1}{4}\tensor{h}{_\alpha_\conjbeta}\tensor*{\varphi}{^{(q)}_0_0}
			+\tfrac{1}{2}\tensor{h}{_\alpha_\conjbeta}\tensor{\varphi}{^{(q)}_\gamma^\gamma}+O(\rho^3),\\
		\tensor*{\Ein}{^{(q)}_\alpha_\beta}&=\tfrac{1}{2}n\tensor*{\varphi}{^{(q)}_\alpha_\beta}+O(\rho^3).
	\end{align*}
	An inductive argument shows that $\tensor*{\Ein}{^{(q)}_I_J}=O(\rho^3)$, $1\le q\le N$, if and only if
	$\tensor*{\varphi}{^{(q)}_i_j}=O(\rho^3)$, $1\le q\le N$.
	Finally, the same identities as \eqref{eq:SecondOrderTermOfACHMetric} hold for
	$\tensor*{\Ein}{^{(0)}_I_J}$ and $\tensor*{\varphi}{^{(0)}_i_j}$, which imply that
	$\tensor*{\varphi}{^{(0)}_i_j}$ must satisfy
	$\tensor*{\varphi}{^{(0)}_0_0}=O(\rho^3)$, $\tensor*{\varphi}{^{(0)}_0_\alpha}=O(\rho^3)$,
	$\tensor*{\varphi}{^{(0)}_\alpha_\conjbeta}=\rho^2\tensor{\Phi}{_\alpha_\conjbeta}+O(\rho^3)$
	and $\tensor*{\varphi}{^{(0)}_\alpha_\beta}=\rho^2\tensor{\Phi}{_\alpha_\beta}+O(\rho^3)$ as desired.
\end{proof}

Let $\overline{g}$ be such a normal-form ACH metric, and for specificity, let its components
$\tensor{\overline{g}}{_i_j}$
be polynomials of degree $2n+a(i,j)$ in $\rho$, which are uniquely determined. We set
\begin{equation}
	\label{eq:DefinitionOfTensorE}
	\tensor{\overline\Ein}{_I_J}=\rho^{2n+1+a(I,J)}\tensor{E}{_I_J}+O(\rho^{2n+2+a(I,J)}),
\end{equation}
where $\tensor{E}{_I_J}$ is constant in the $\rho$-direction.
We already know that $\tensor{E}{_\alpha_\beta}=\tensor{\caO}{_\alpha_\beta}$ and
$\tensor{E}{_0_\alpha}=-i\tensor{\nabla}{^\beta}\tensor{\caO}{_\alpha_\beta}
-i\tensor{N}{_\alpha^\conjbeta^\conjgamma}\tensor{\caO}{_\conjbeta_\conjgamma}$.
Set
\begin{equation*}
	u:=-\frac{1}{n+1}(\tensor{E}{_\infty_0}-i\tensor{\nabla}{^\alpha}\tensor{E}{_\infty_\alpha}
	+i\tensor{\nabla}{^\conjalpha}\tensor{E}{_\infty_\conjalpha}).
\end{equation*}

\begin{thm}
	\label{thm:FormalSolution}
	Let $\kappa$ be any smooth function and $\tensor{\lambda}{_\alpha_\beta}$ a smooth tensor satisfying
	\begin{equation}
		\label{eq:PDEForTwoTensor}
		\tensor{D}{^\alpha^\beta}\tensor{\lambda}{_\alpha_\beta}
		-\tensor{D}{^\conjalpha^\conjbeta}\tensor{\lambda}{_\conjalpha_\conjbeta}=iu.
	\end{equation}
	Then there is a singular normal-form ACH metric $g$ satisfying $\tensor{\Ein}{_I_J}=\caA^\infty$ and
	\begin{equation}
		\label{eq:PrescriptionOfACHCoefficients}
		\dfrac{1}{(2n+4)!}\left.\left(\partial_\rho^{2n+4}\tensor*{g}{^{(0)}_0_0}\right)\right|_M=\kappa,\qquad
		\dfrac{1}{(2n+2)!}\left.\left(\partial_\rho^{2n+2}\tensor*{g}{^{(0)}_\alpha_\beta}\right)\right|_M
		=\tensor{\lambda}{_\alpha_\beta},
	\end{equation}
	where $\tensor{g}{_i_j}\sim\sum_{q=0}^{\infty}\tensor*{g}{^{(q)}_i_j}(\log\rho)^q$ is
	the asymptotic expansion of $\tensor{g}{_i_j}$.
	The components $\tensor{g}{_i_j}$ are uniquely determined modulo $\caA^\infty$ by the condition above.
\end{thm}

As is clear from the proof below, Theorem \ref{thm:FormalSolution} also holds in the following formal sense.
Let $p\in M$, $\kappa$ a smooth function and $\tensor{\lambda}{_\alpha_\beta}$ a tensor satisfying
\eqref{eq:PDEForTwoTensor} to the infinite order at $p$.
Then there exists a singular normal-form ACH metric $g$ satisfying \eqref{eq:PrescriptionOfACHCoefficients} and
$\tensor{\Ein}{_I_J}=\caA^\infty$ to the infinite order at $p$,
and the Taylor expansions of $\tensor*[]{g}{^{(q)}_i_j}$ at $p$ are uniquely determined by those of
$\kappa$ and $\tensor{\lambda}{_\alpha_\beta}$.
On the other hand, there is a formal power series solution to \eqref{eq:PDEForTwoTensor}
by the Cauchy--Kovalevskaya theorem.
Hence, by Borel's Lemma, we have $\tensor{\lambda}{_\alpha_\beta}$ solving \eqref{eq:PDEForTwoTensor} to the
infinite order at $p$ and prove the first statement of Theorem \ref{thm:SingularSolutionToInfiniteOrder}.
We do not know whether \eqref{eq:PDEForTwoTensor} is solvable in the category of smooth tensors.

The first step to prove Theorem \ref{thm:FormalSolution} is the following.

\begin{lem}
	\label{lem:FirstHalfOfSolvingAtCriticalStep}
	There exists a singular normal-form ACH metric $g$ satisfying
	\begin{align*}
		\tensor{\Ein}{_\infty_\infty}&=\caA^{2n+4},\qquad
		\tensor{\Ein}{_\infty_0}=\caA^{2n+4},\qquad
		\tensor{\Ein}{_\infty_\alpha}=\caA^{2n+3},\\
		\tensor{\Ein}{_0_0}&=\caA^{2n+4},\qquad
		\tensor{\Ein}{_0_\alpha}=\caA^{2n+4},\qquad
		\tensor{\Ein}{_\alpha_\conjbeta}=\caA^{2n+4},\qquad
		\tensor{\Ein}{_\alpha_\beta}=\caA^{2n+3}
	\end{align*}
	and the followings not containing logarithmic terms:
	\begin{align*}
		&\tensor{\Ein}{_\infty_0}\text{ mod $\caA^{2n+5}$},\qquad
		\tensor{\Ein}{_\infty_\alpha}\text{ mod $\caA^{2n+4}$},\\
		&\tensor{\Ein}{_0_0}\text{ mod $\caA^{2n+5}$},\qquad
		\tensor{\Ein}{_\alpha_\conjbeta}\text{ mod $\caA^{2n+5}$}.
	\end{align*}
	Any such a metric $g$ is of the form
	\begin{align*}
		\tensor{g}{_0_0}
		&=\tensor{\overline{g}}{_0_0}+\tensor*{\psi}{^{(0)}_0_0}+\tensor*{\psi}{^{(1)}_0_0}\log\rho
		+\tensor*{\psi}{^{(2)}_0_0}(\log\rho)^{2}+\caA^{2n+5},\\
		\tensor{g}{_0_\alpha}
		&=\tensor{\overline{g}}{_0_\alpha}+\tensor*{\psi}{^{(0)}_0_\alpha}+\tensor*{\psi}{^{(1)}_0_\alpha}\log\rho+\caA^{2n+4},\\
		\tensor{g}{_\alpha_\conjbeta}
		&=\tensor{\overline{g}}{_\alpha_\conjbeta}
		+\tensor*{\psi}{^{(0)}_\alpha_\conjbeta}+\tensor*{\psi}{^{(1)}_\alpha_\conjbeta}\log\rho
		+\tensor*{\psi}{^{(2)}_\alpha_\conjbeta}(\log\rho)^{2}+\caA^{2n+5},\\
		\tensor{g}{_\alpha_\beta}
		&=\tensor{\overline{g}}{_\alpha_\beta}
		+\tensor*{\psi}{^{(0)}_\alpha_\beta}+\tensor*{\psi}{^{(1)}_\alpha_\beta}\log\rho+\caA^{2n+3},
	\end{align*}
	where $\tensor*{\psi}{^{(q)}_i_j}=O(\rho^{2n+1+a(i,j)})$. Furthermore, among $\tensor*{\psi}{^{(q)}_i_j}$,
	\begin{equation*}
		\tensor*{\psi}{^{(2)}_0_0},\qquad
		\tensor*{\psi}{^{(2)}_\alpha_\conjbeta},\qquad
		\tensor*{\psi}{^{(1)}_0_\alpha},\qquad
		\tf(\tensor*{\psi}{^{(1)}_\alpha_\conjbeta}),\qquad
		\tensor*{\psi}{^{(1)}_\alpha_\beta}\qquad\text{and}\qquad
		\tfrac{1}{2}n\tensor*{\psi}{^{(1)}_0_0}+(n+1)\tensor{\psi}{^{(1)}_\alpha^\alpha}
	\end{equation*}
	are uniquely determined modulo $O(\rho^{2n+2+a(i,j)})$.
	In particular, if $\tensor{\caO}{_\alpha_\beta}=0$ then they are zero modulo $O(\rho^{2n+2+a(i,j)})$.
\end{lem}

\begin{proof}
	We shall determine when
	\begin{equation}
		\label{eq:OneStepFurtherApproximationForLogACH}
		\tensor{g}{_i_j}=\tensor{\overline{g}}{_i_j}+\sum_{q=0}^N\tensor*{\psi}{^{(q)}_i_j}(\log\rho)^q,\qquad
		\tensor*{\psi}{^{(q)}_i_j}=O(\rho^{2n+1+a(i,j)})
	\end{equation}
	enjoys the condition imposed.
	By \eqref{eq:EinPerturbationLowerOrder} and \eqref{eq:EinPerturbationHigherOrder}, which are also valid here
	if $O(\rho^{m'})$ is replaced by $\caA^{m'}$, the difference $\tensor{\delta\Ein}{_I_J}$ between the
	Einstein tensors of $\overline{g}$ and $g$ is of the form
	\begin{equation*}
		\tensor{\delta\Ein}{_I_J}=\sum_{q=0}^N\tensor*{\delta\Ein}{^{(q)}_I_J}(\log\rho)^q+\caA^{2n+2+a(I,J)}.
	\end{equation*}
	We may assume $N\ge 3$. Then, by \eqref{eq:EinPerturbationLowerOrder} we have
	$\tensor*{\delta\Ein}{^{(N)}_0_\alpha}=O(\rho^{2n+4})$ and
	$\tensor*{\delta\Ein}{^{(N)}_\alpha_\beta}=O(\rho^{2n+3})$, which imply that
	$\tensor*{\Ein}{^{(N)}_0_\alpha}=O(\rho^{2n+4})$, $\tensor*{\Ein}{^{(N)}_\alpha_\beta}=O(\rho^{2n+3})$
	already hold, and
	\begin{equation}
		\label{eq:VariationOfEinAtCriticalStep}
		\begin{split}
			\tensor*{\delta\Ein}{^{(q-1)}_0_\alpha}
			&=-\tfrac{1}{4}q(n+2)\tensor*{\psi}{^{(q)}_0_\alpha}+O(\rho^{2n+4}),\\
			\tensor*{\delta\Ein}{^{(q-1)}_\alpha_\beta}
			&=-\tfrac{1}{4}q(n+1)\tensor*{\psi}{^{(q)}_\alpha_\beta}+O(\rho^{2n+3})
		\end{split}
	\end{equation}
	for $q=N$. This shows that
	$\tensor*{\Ein}{^{(N-1)}_0_\alpha}=O(\rho^{2n+4})$, $\tensor*{\Ein}{^{(N-1)}_\alpha_\beta}=O(\rho^{2n+3})$
	if and only if $\tensor*{\psi}{^{(N)}_0_\alpha}=O(\rho^{2n+4})$,
	$\tensor*{\psi}{^{(N)}_\alpha_\beta}=O(\rho^{2n+3})$.
	Since \eqref{eq:VariationOfEinAtCriticalStep} holds for $q=q_0$ if
	$\tensor*{\psi}{^{(q)}_0_\alpha}=O(\rho^{2n+4})$ and $\tensor*{\psi}{^{(q)}_\alpha_\beta}=O(\rho^{2n+3})$
	for $q_0+1\le q\le N$,
	inductively we verify that
	$\tensor{\Ein}{_0_\alpha}=\caA^{2n+4}$, $\tensor{\Ein}{_\alpha_\beta}=\caA^{2n+3}$ if and only if
	$\tensor*{\psi}{^{(q)}_0_\alpha}=O(\rho^{2n+4})$, $\tensor*{\psi}{^{(q)}_\alpha_\beta}=O(\rho^{2n+3})$,
	$2\le q\le N$ and
	\begin{equation}
		\label{eq:FirstLogarithmicTerms}
		\tensor*{\psi}{^{(1)}_0_\alpha}=\dfrac{4}{n+2}\rho^{2n+3}\tensor{E}{_0_\alpha}+O(\rho^{2n+4}),\qquad
		\tensor*{\psi}{^{(1)}_\alpha_\beta}=\dfrac{4}{n+1}\rho^{2n+2}\tensor{E}{_\alpha_\beta}+O(\rho^{2n+3}).
	\end{equation}

	Next, from
	\eqref{eq:EinPerturbationHigherOrderZeroZero}--\eqref{eq:EinPerturbationHigherOrderHermitianTraceFree}
	we have $n\tensor*{\delta\Ein}{^{(N)}_0_0}-2\tensor{\delta\Ein}{^{(N)}_\alpha^\alpha}=O(\rho^{2n+5})$ and
	\begin{align*}
		\tensor*{\delta\Ein}{^{(q)}_0_0}
		&=\tfrac{1}{2}n\tensor*{\psi}{^{(q)}_0_0}+(n+1)\tensor{\psi}{^{(q)}_\alpha^\alpha}+O(\rho^{2n+5}),\\
		\tf(\tensor*{\delta\Ein}{^{(q)}_\alpha_\conjbeta})
		&=-\tfrac{1}{2}n\tf(\tensor*{\psi}{^{(q)}_\alpha_\conjbeta})+O(\rho^{2n+5}),\\
		n\tensor*{\delta\Ein}{^{(q-1)}_0_0}-2\tensor{\delta\Ein}{^{(q-1)}_\alpha^\alpha}
		&=-\tfrac{1}{4}q(n+3)(n\tensor*{\psi}{^{(q)}_0_0}-2\tensor{\psi}{^{(q)}_\alpha^\alpha})+O(\rho^{2n+5})
	\end{align*}
	for $q=N$.
	Hence both $\tensor*{\psi}{^{(N)}_0_0}$ and $\tensor*{\psi}{^{(N)}_\alpha_\conjbeta}$ must be $O(\rho^{2n+5})$.
	Inductively we show that, in order for us to have
	$\tensor*{\Ein}{^{(q)}_0_0}=O(\rho^{2n+5})$, $\tensor*{\Ein}{^{(q)}_\alpha_\conjbeta}=O(\rho^{2n+5})$,
	$2\le q\le N$,
	it is necessary and sufficient that $\tensor*{\psi}{^{(q)}_0_0}$,
	$\tensor*{\psi}{^{(q)}_\alpha_\conjbeta}$, $3\le q\le N$, and
	$\tfrac{1}{2}n\tensor*{\psi}{^{(2)}_0_0}+(n+1)\tensor{\psi}{^{(2)}_\alpha^\alpha}$,
	$\tf(\tensor*{\psi}{^{(2)}_\alpha_\conjbeta})$ are all $O(\rho^{2n+5})$.
	
	Again by
	\eqref{eq:EinPerturbationHigherOrderZeroZero}--\eqref{eq:EinPerturbationHigherOrderHermitianTraceFree},
	modulo $O(\rho^{2n+4})$ terms which linearly depend on $\tensor*{\psi}{^{(2)}_0_0}$,
	$\tensor*{\psi}{^{(1)}_0_\alpha}$ and $\tensor*{\psi}{^{(1)}_\alpha_\beta}$,
	\begin{align*}
		\tensor*{\delta\Ein}{^{(1)}_0_0}
		&\equiv\tfrac{1}{2}n\tensor*{\psi}{^{(1)}_0_0}+(n+1)\tensor{\psi}{^{(1)}_\alpha^\alpha}+O(\rho^{2n+5}),\\
		\tf(\tensor*{\delta\Ein}{^{(1)}_\alpha_\conjbeta})
		&\equiv -\tfrac{1}{2}n\tf(\tensor*{\psi}{^{(1)}_\alpha_\conjbeta})+O(\rho^{2n+5}),\\
		n\tensor*{\delta\Ein}{^{(1)}_0_0}-2\tensor{\delta\Ein}{^{(1)}_\alpha^\alpha}
		&\equiv -\tfrac{1}{2}(n+3)(n\tensor*{\psi}{^{(2)}_0_0}-2\tensor{\psi}{^{(2)}_\alpha^\alpha})
		+O(\rho^{2n+5}).
	\end{align*}
	Therefore $\tensor*{\psi}{^{(2)}_0_0}$, $\tensor{\psi}{^{(2)}_\alpha^\alpha}$,
	$\tf(\tensor*{\psi}{^{(1)}_\alpha_\conjbeta})$ and
	$\tfrac{1}{2}n\tensor*{\psi}{^{(1)}_0_0}+(n+1)\tensor{\psi}{^{(1)}_\alpha^\alpha}$ are uniquely determined
	modulo $O(\rho^{2n+5})$ by the requirement
	$\tensor*{\Ein}{^{(1)}_0_0}=O(\rho^{2n+5})$, $\tensor*{\Ein}{^{(1)}_\alpha_\conjbeta}=O(\rho^{2n+5})$.

	For $\tensor{g}{_i_j}$ satisfying all the restrictions we have found above,
	$\tensor{\Ein}{_\infty_0}$ and $\tensor{\Ein}{_\infty_\alpha}$ do not contain logarithmic terms
	modulo $\caA^{2n+5}$ and $\caA^{2n+4}$, respectively;
	one can show this fact by \eqref{eq:ContractedBianchiOriginal2} and
	\eqref{eq:ContractedBianchiOriginal3}, or by \eqref{eq:ImaginaryPartOfDoubleDivergenceOfObstruction}.
	If $\tensor{\caO}{_\alpha_\beta}=0$, \eqref{eq:FirstLogarithmicTerms} implies that
	$\tensor*{\psi}{^{(1)}_0_\alpha}$ and $\tensor*{\psi}{^{(1)}_\alpha_\beta}$ are zero,
	and hence $\tensor*{\psi}{^{(2)}_0_0}$, $\tensor{\psi}{^{(2)}_\alpha^\alpha}$,
	$\tf(\tensor*{\psi}{^{(1)}_\alpha_\conjbeta})$ and
	$\tfrac{1}{2}n\tensor*{\psi}{^{(1)}_0_0}+(n+1)\tensor{\psi}{^{(1)}_\alpha^\alpha}$ are also zero.
\end{proof}

The rest of the proof of Theorem \ref{thm:FormalSolution} consists of two parts,
in the first of which we finish constructing a singular normal-form ACH metric satisfying
$\tensor{\Ein}{_I_J}=\caA^{2n+2+a(I,J)}$,
and in the second we go through an inductive argument to achieve $\tensor{\Ein}{_I_J}=\caA^\infty$.

\begin{proof}[Proof of Theorem \ref{thm:FormalSolution}]
	Let $g$ be a singular normal-form ACH metric we have obtained in
	Lemma \ref{lem:FirstHalfOfSolvingAtCriticalStep}.
	By \eqref{eq:EinPerturbationHigherOrderInftyZero}, \eqref{eq:EinPerturbationLowerOrder} and
	\eqref{eq:FirstLogarithmicTerms} we have
	\begin{align*}
		\begin{split}
			\tensor*{\delta\Ein}{^{(0)}_\infty_0}&=
			(n+2)\rho
			(\tensor{\nabla}{^\alpha}\tensor*{\psi}{^{(0)}_0_\alpha}
			+\tensor{\nabla}{^\conjalpha}\tensor*{\psi}{^{(0)}_0_\conjalpha})
			-(n+1)\rho^2(\tensor{A}{^\alpha^\beta}\tensor*{\psi}{^{(0)}_\alpha_\beta}
			+\tensor{A}{^\conjalpha^\conjbeta}\tensor*{\psi}{^{(0)}_\conjalpha_\conjbeta})\\
			&\phantom{=\;}
			+\rho^{2n+4}
			\left(\tfrac{2}{n+2}(\tensor{\nabla}{^\alpha}\tensor{E}{_0_\alpha}
			+\tensor{\nabla}{^\conjalpha}\tensor{E}{_0_\conjalpha})
			-\tfrac{2}{n+1}(\tensor{A}{^\alpha^\beta}\tensor{E}{_\alpha_\beta}
			+\tensor{A}{^\conjalpha^\conjbeta}\tensor{E}{_\conjalpha_\conjbeta})\right)
			+O(\rho^{2n+5}),
		\end{split}\\
		\begin{split}
			\tensor*{\delta\Ein}{^{(0)}_\infty_\alpha}&=
			-i(n+2)\tensor*{\psi}{^{(0)}_0_\alpha}
			+(n+1)\rho\tensor{\nabla}{^\beta}\tensor*{\psi}{^{(0)}_\alpha_\beta}
			+(n+1)\rho\tensor{N}{_\alpha^\conjbeta^\conjgamma}\tensor*{\psi}{^{(0)}_\conjbeta_\conjgamma}\\
			&\phantom{=\;}
			-\rho^{2n+3}
			\left(\tfrac{2}{n+2}i\tensor{E}{_0_\alpha}
			-\tfrac{2}{n+1}(\tensor{\nabla}{^\beta}\tensor{E}{_\alpha_\beta}
			+\tensor{N}{_\alpha^\conjbeta^\conjgamma}\tensor{E}{_\conjbeta_\conjgamma})\right)
			+O(\rho^{2n+4}).
		\end{split}
	\end{align*}
	If we set $\tensor*{\psi}{^{(0)}_0_\alpha}=\rho^{2n+3}\tensor{\nu}{_\alpha}+O(\rho^{2n+4})$ and
	$\tensor*{\psi}{^{(0)}_\alpha_\beta}=\rho^{2n+2}\tensor{\mu}{_\alpha_\beta}+O(\rho^{2n+3})$,
	then attaining $\tensor*{\Ein}{^{(0)}_\infty_0}=O(\rho^{2n+5})$ and
	$\tensor*{\Ein}{^{(0)}_\infty_\alpha}=O(\rho^{2n+4})$
	is equivalent to solving the following system of PDEs:
	\begin{equation}
		\label{eq:PDEForOneAndTwoTensors}
		\begin{cases}
			&(n+2)(\tensor{\nabla}{^\alpha}\tensor{\nu}{_\alpha}
			+\tensor{\nabla}{^\conjalpha}\tensor{\nu}{_\conjalpha})
			-(n+1)(\tensor{A}{^\alpha^\beta}\tensor{\mu}{_\alpha_\beta}
			+\tensor{A}{^\conjalpha^\conjbeta}\tensor{\mu}{_\conjalpha_\conjbeta})\\
			&\quad =-\tensor{E}{_\infty_0}
			-\frac{2}{n+2}(\tensor{\nabla}{^\alpha}\tensor{E}{_0_\alpha}
			+\tensor{\nabla}{^\conjalpha}\tensor{E}{_0_\conjalpha})
			+\frac{2}{n+1}(\tensor{A}{^\alpha^\beta}\tensor{E}{_\alpha_\beta}
			+\tensor{A}{^\conjalpha^\conjbeta}\tensor{E}{_\conjalpha_\conjbeta}),\\
			&-i(n+2)\tensor{\nu}{_\alpha}
			+(n+1)\tensor{\nabla}{^\beta}\tensor{\mu}{_\alpha_\beta}
			+(n+1)\tensor{N}{_\alpha^\conjbeta^\conjgamma}\tensor{\mu}{_\conjbeta_\conjgamma}\\
			&\quad =-\tensor{E}{_\infty_\alpha}
			+\frac{2}{n+2}i\tensor{E}{_0_\alpha}
			-\frac{2}{n+1}(\tensor{\nabla}{^\beta}\tensor{E}{_\alpha_\beta}
			+\tensor{N}{_\alpha^\conjbeta^\conjgamma}\tensor{E}{_\conjbeta_\conjgamma}).
		\end{cases}
	\end{equation}
	If we substitute the second equation into the first one and use
	$\tensor{E}{_\alpha_\beta}=\tensor{\caO}{_\alpha_\beta}$ and
	\eqref{eq:ImaginaryPartOfDoubleDivergenceOfObstruction}, the system is reduced to
	$\tensor{D}{^\alpha^\beta}\tensor{\mu}{_\alpha_\beta}
	-\tensor{D}{^\conjalpha^\conjbeta}\tensor{\mu}{_\conjalpha_\conjbeta}=iu$.
	Hence, by setting $\tensor{\mu}{_\alpha_\beta}=\tensor{\lambda}{_\alpha_\beta}$ and
	determining $\tensor{\nu}{_\alpha}$ by \eqref{eq:PDEForOneAndTwoTensors} we achieve
	$\tensor*{\Ein}{^{(0)}_\infty_0}=O(\rho^{2n+5})$ and $\tensor*{\Ein}{^{(0)}_\infty_\alpha}=O(\rho^{2n+4})$.

	Having fixed $\tensor*{\psi}{^{(0)}_0_\alpha}$ and $\tensor*{\psi}{^{(0)}_\alpha_\beta}$,
	now we may determine $\tensor*{\psi}{^{(1)}_0_0}$, $\tensor{\psi}{^{(1)}_\alpha^\alpha}$,
	$\tf(\tensor*{\psi}{^{(0)}_\alpha_\conjbeta})$ and
	$\tfrac{1}{2}n\tensor*{\psi}{^{(0)}_0_0}+(n+1)\tensor{\psi}{^{(0)}_\alpha^\alpha}$ modulo
	$O(\rho^{2n+5})$ so that $\tensor*{\Ein}{^{(0)}_0_0}$, $\tensor*{\Ein}{^{(0)}_\alpha_\conjbeta}$
	are $O(\rho^{2n+5})$ by observing
	\eqref{eq:EinPerturbationHigherOrderZeroZero}--\eqref{eq:EinPerturbationHigherOrderHermitianTraceFree}.
	It automatically holds that $\tensor{\Ein}{_\infty_\infty}=\caA^{2n+5}$ by
	\eqref{eq:ContractedBianchiOriginal1}.
	Although $\tfrac{1}{2}n\tensor*{\psi}{^{(0)}_0_0}+(n+1)\tensor{\psi}{^{(0)}_\alpha^\alpha}$ is fixed,
	$\tensor*{\psi}{^{(0)}_0_0}$ (or $\tensor{\psi}{^{(0)}_\alpha^\alpha}$) is remaining to be free,
	so we prescribe it by $\tensor*{\psi}{^{(0)}_0_0}=\rho^{2n+4}\kappa+O(\rho^{2n+5})$.

	We have shown that there is a singular normal-form ACH metric satisfying
	$\tensor{\Ein}{_I_J}=\caA^{2n+2+a(I,J)}$,
	and that if we impose the condition \eqref{eq:PrescriptionOfACHCoefficients} then
	$\tensor{g}{_i_j}$ are unique modulo $\caA^{2n+2+a(i,j)}$.
	Let $m\ge 2n+3$ and suppose that $g$ is a singular normal-form ACH metric satisfying
	$\tensor{\Ein}{_I_J}=\caA^{m-1+a(I,J)}$. We set
	\begin{equation*}
		\tensor{ {g'} }{_i_j}=\tensor{g}{_i_j}+\sum_{q=0}^N\tensor*{\psi}{^{(q)}_i_j}(\log\rho)^q,
	\end{equation*}
	where $\tensor*{\psi}{^{(q)}_i_j}=O(\rho^{m-1+a(i,j)})$,
	and shall prove that $\tensor*{\psi}{^{(q)}_i_j}$ mod $O(\rho^{m+a(i,j)})$ may be uniquely determined
	so that $\tensor{ {\Ein'} }{_I_J}=\caA^{m+a(I,J)}$ holds.
	Then the induction works and we obtain the theorem.

	By replacing $N$ with a larger one if necessary, we express the difference $\delta\Ein=\Ein'-\Ein$
	between the Einstein tensors as
	\begin{equation*}
		\tensor{\delta\Ein}{_I_J}=\sum_{q=0}^N\tensor*{\delta\Ein}{^{(q)}_I_J}(\log\rho)^q+\caA^{m+a(I,J)}.
	\end{equation*}
	Then by \eqref{eq:EinPerturbationLowerOrder} and \eqref{eq:EinPerturbationHigherOrder} we have,
	modulo terms linearly depending on $\tensor*{\psi}{^{(q+2)}_i_j}$ or $\tensor*{\psi}{^{(q+1)}_i_j}$,
	\allowdisplaybreaks
	\begin{subequations}
	\begin{align}
		\begin{split}
			\tensor*{\delta\Ein}{^{(q)}_0_0}
			&\equiv -\tfrac{1}{8}(m^2-2nm-8n-4)\tensor*{\psi}{^{(q)}_0_0}
			+\tfrac{1}{2}m\tensor{\psi}{^{(q)}_\alpha^\alpha}\\
			&\phantom{=\;}
			+(\text{$O(\rho^{m+2})$ terms depending on
			$\tensor*{\psi}{^{(q)}_0_\alpha}$ and $\tensor*{\psi}{^{(q)}_\alpha_\beta}$})
			+O(\rho^{m+3})
		\end{split}\\
		\tensor*{\delta\Ein}{^{(q)}_0_\alpha}
		&\equiv -\tfrac{1}{8}(m+2)(m-2n-2)\tensor*{\psi}{^{(q)}_0_\alpha}+O(\rho^{m+2}),\\
		\begin{split}
			\label{eq:PerturbationOfEinsteinTraceForInduction}
			\tensor{\delta\Ein}{^{(q)}_\alpha^\alpha}
			&\equiv\tfrac{1}{8}n(m-2)\tensor*{\psi}{^{(q)}_0_0}
			-\tfrac{1}{8}\left( m^2-(4n-2)m-8n-8 \right)\tensor{\psi}{^{(q)}_\alpha^\alpha}\\
			&\phantom{=\;}
			+(\text{$O(\rho^{m+2})$ terms depending on
			$\tensor*{\psi}{^{(q)}_0_\alpha}$ and $\tensor*{\psi}{^{(q)}_\alpha_\beta}$})+O(\rho^{m+3}),
		\end{split}\\
		\begin{split}
			\tf(\tensor*{\delta\Ein}{^{(q)}_\alpha_\conjbeta})
			&\equiv -\tfrac{1}{8}(m^2-2nm-2n-9)\tf(\tensor*{\psi}{^{(q)}_\alpha_\conjbeta})\\
			&\phantom{=\;}
			+(\text{$O(\rho^{m+2})$ terms depending on
			$\tensor*{\psi}{^{(q)}_0_\alpha}$ and $\tensor*{\psi}{^{(q)}_\alpha_\beta}$})+O(\rho^{m+3}),
		\end{split}\\
		\tensor*{\delta\Ein}{^{(q)}_\alpha_\beta}
		&\equiv -\tfrac{1}{8}m(m-2n-2)\tensor*{\psi}{^{(q)}_\alpha_\beta}+O(\rho^{m+1}).
	\end{align}
	\end{subequations}
	\allowdisplaybreaks[0]%
	By \eqref{eq:DeterminantOfCoefficients},
	if $m\not=4n+2$, we may determine $\tensor*{\psi}{^{(N)}_i_j}$, $\tensor*{\psi}{^{(N-1)}_i_j}$, $\dots$,
	$\tensor*{\psi}{^{(0)}_i_j}$ inductively so that $\tensor{ {\Ein'} }{_i_j}=\caA^{m+a(i,j)}$ hold.
	Then by \eqref{eq:ContractedBianchiOriginal} it automatically holds that
	$\tensor{ {\Ein'} }{_\infty_\infty}=\caA^{m+3}$, $\tensor{ {\Ein'} }{_\infty_0}=\caA^{m+3}$
	and $\tensor{ {\Ein'} }{_\infty_\alpha}=\caA^{m+2}$.
	If $m=4n+2$, instead of \eqref{eq:PerturbationOfEinsteinTraceForInduction} we use
	\begin{equation*}
		\begin{split}
			\tensor*{\delta\Ein}{^{(q)}_\infty_\infty}
			&\equiv -8n(n+1)\tensor*{\psi}{^{(q)}_0_0}-8(n+1)(2n+1)\tensor{\psi}{^{(q)}_\alpha^\alpha}\\
			&\phantom{=\;}
			+(\text{$O(\rho^{4n+4})$ terms depending on $\tensor*{\psi}{^{(q)}_\alpha_\beta}$})+O(\rho^{4n+5}),
		\end{split}
	\end{equation*}
	which holds modulo $\tensor*{\psi}{^{(q+2)}_i_j}$ and $\tensor*{\psi}{^{(q+1)}_i_j}$.
	We may determine $\tensor*{\psi}{^{(N)}_i_j}$, $\tensor*{\psi}{^{(N-1)}_i_j}$, $\dots$,
	$\tensor*{\psi}{^{(0)}_i_j}$ inductively so that $\tensor{ {\Ein'} }{_\infty_\infty}=\caA^{4n+5}$,
	$\tensor{ {\Ein'} }{_0_0}=\caA^{4n+5}$, $\tensor{ {\Ein'} }{_0_\alpha}=\caA^{4n+4}$,
	$\tf(\tensor{ {\Ein'} }{_\alpha_\conjbeta})=\caA^{4n+5}$ and
	$\tensor{ {\Ein'} }{_\alpha_\beta}=\caA^{4n+4}$.
	By \eqref{eq:ContractedBianchiOriginal}, we obtain
	$\tensor{ {\Ein'} }{_\alpha^\alpha}=\caA^{4n+3}$, $\tensor{ {\Ein'} }{_\infty_0}=\caA^{4n+5}$ and
	$\tensor{ {\Ein'} }{_\infty_\alpha}=\caA^{4n+4}$.
\end{proof}

Finally we shall discuss constructing a completely log-free solution when $\tensor{\caO}{_\alpha_\beta}=0$.
We set
\begin{equation*}
	\begin{split}
		v&:=-\tensor{E}{_0_0}+\tfrac{2}{n}\tensor{E}{_\alpha^\alpha}
		-\tfrac{1}{n}(\tensor{\nabla}{^\alpha}\tensor{E}{_\infty_\alpha}
		+\tensor{\nabla}{^\conjalpha}\tensor{E}{_\infty_\conjalpha})
		+\tfrac{2}{n(n+2)}i(\tensor{\nabla}{^\alpha}\tensor{E}{_0_\alpha}
		-\tensor{\nabla}{^\conjalpha}\tensor{E}{_0_\conjalpha})\\
		&\phantom{:=\;}
		-\tfrac{2}{n(n+1)}(\tensor{\nabla}{^\alpha}\tensor{\nabla}{^\beta}\tensor{E}{_\alpha_\beta}
		+\tensor{\nabla}{^\conjalpha}\tensor{\nabla}{^\conjbeta}\tensor{E}{_\conjalpha_\conjbeta}
		+\tensor{N}{^\gamma^\alpha^\beta}\tensor{\nabla}{_\gamma}\tensor{E}{_\alpha_\beta}
		+\tensor{N}{^\conjgamma^\conjalpha^\conjbeta}
		\tensor{\nabla}{_\conjgamma}\tensor{E}{_\conjalpha_\conjbeta}\\
		&\phantom{:=\;-\tfrac{2}{n(n+1)}(}
		+\tensor{N}{^\gamma^\alpha^\beta_,_\gamma}\tensor{E}{_\alpha_\beta}
		+\tensor{N}{^\conjgamma^\conjalpha^\conjbeta_,_\conjgamma}\tensor{E}{_\conjalpha_\conjbeta}).
	\end{split}
\end{equation*}

\begin{thm}
	Suppose that $\tensor{\caO}{_\alpha_\beta}=0$.
	Let $\kappa$ be a smooth function and $\tensor{\lambda}{_\alpha_\beta}$ a smooth tensor satisfying
	\begin{equation}
		\label{eq:SystemOfPDEsForTwoTensor}
		\begin{cases}
		\tensor{D}{^\alpha^\beta}\tensor{\lambda}{_\alpha_\beta}
		-\tensor{D}{^\conjalpha^\conjbeta}\tensor{\lambda}{_\conjalpha_\conjbeta}=iu,\\
		\tensor*{D}{_{-2/n}^\alpha^\beta}\tensor{\lambda}{_\alpha_\beta}
		+\tensor*{D}{_{-2/n}^\conjalpha^\conjbeta}\tensor{\lambda}{_\conjalpha_\conjbeta}=v.
	\end{cases}
	\end{equation}
	Then there is a normal-form ACH metric $g$, which is free of logarithmic terms, satisfying
	$\tensor{\Ein}{_I_J}=\caA^\infty$ and
	\begin{equation}
		\label{eq:PrescriptionOfACHCoefficients2}
		\dfrac{1}{(2n+4)!}\left.\left(\partial_\rho^{2n+4}\tensor*{g}{_0_0}\right)\right|_M=\kappa,\qquad
		\dfrac{1}{(2n+2)!}\left.\left(\partial_\rho^{2n+2}\tensor*{g}{_\alpha_\beta}\right)\right|_M
		=\tensor{\lambda}{_\alpha_\beta}.
	\end{equation}
	The Taylor expansions of the components $\tensor{g}{_i_j}$ at $\bdry X$ are unique.
\end{thm}

Again this theorem also holds in the formal sense.
Since the principal parts of $\tensor{D}{^\alpha^\beta}$ and $\tensor*{D}{_{-2/n}^\alpha^\beta}$ agree,
the system \eqref{eq:SystemOfPDEsForTwoTensor} is formally solvable at any given point;
in fact, if one arbitrarily prescribes the components of $\tensor{\lambda}{_\alpha_\beta}$ except
$\tensor{\lambda}{_1_1}$, for example, and writes $\tensor{\lambda}{_1_1}=\mu+i\nu$ where $\mu$ and $\nu$ are
real-valued, then \eqref{eq:SystemOfPDEsForTwoTensor} can be regarded as a system of PDEs for $\mu$ and $\nu$
and the Cauchy--Kovalevskaya theorem can be applied to this system.
Thus we can show the second statement of Theorem \ref{thm:SingularSolutionToInfiniteOrder}.

\begin{proof}
	If $\tensor{\caO}{_\alpha_\beta}=0$, then a (potentially) singular normal-form ACH metric $g$
	satisfying the conditions in the statement of Lemma \ref{lem:FirstHalfOfSolvingAtCriticalStep} is of the form
	\begin{align*}
		\tensor{g}{_0_0}
		&=\tensor{\overline{g}}{_0_0}+\tensor*{\psi}{^{(0)}_0_0}+\tensor*{\psi}{^{(1)}_0_0}\log\rho+\caA^{2n+5},\\
		\tensor{g}{_0_\alpha}
		&=\tensor{\overline{g}}{_0_\alpha}+\tensor*{\psi}{^{(0)}_0_\alpha}+\caA^{2n+4},\\
		\tensor{g}{_\alpha_\conjbeta}
		&=\tensor{\overline{g}}{_\alpha_\conjbeta}
		+\tensor*{\psi}{^{(0)}_\alpha_\conjbeta}
		+\tfrac{1}{n}\tensor{h}{_\alpha_\conjbeta}\tensor{\psi}{^{(1)}_\gamma^\gamma}\log\rho+\caA^{2n+5},\\
		\tensor{g}{_\alpha_\beta}
		&=\tensor{\overline{g}}{_\alpha_\beta}
		+\tensor*{\psi}{^{(0)}_\alpha_\beta}+\caA^{2n+3}.
	\end{align*}
	Here $\tfrac{1}{2}n\tensor*{\psi}{^{(1)}_0_0}+(n+1)\tensor{\psi}{^{(1)}_\alpha^\alpha}=O(\rho^{2n+5})$
	should hold.
	After prescribing $\tensor*{\psi}{^{(0)}_0_\alpha}$ and $\tensor*{\psi}{^{(0)}_\alpha_\beta}$,
	the potential log-term coefficients
	$\tensor*{\psi}{^{(1)}_0_0}$ and $\tensor{\psi}{^{(1)}_\alpha^\alpha}$ are determined by requiring
	$n\tensor*{\Ein}{^{(0)}_0_0}-2\tensor{\Ein}{^{(0)}_\alpha^\alpha}=O(\rho^{2n+5})$.
	So let us look at the dependence of
	$n\tensor*{\Ein}{^{(0)}_0_0}-2\tensor {\Ein}{^{(0)}_\alpha^\alpha}$
	on $\tensor*{\psi}{^{(0)}_\alpha_\beta}$.
	Using \eqref{eq:EinPerturbationHigherOrderZeroZero} and \eqref{eq:EinPerturbationHigherOrderTrace} again,
	we obtain
	\begin{equation*}
		\begin{split}
			n\tensor*{\delta\Ein}{^{(0)}_0_0}-2\tensor{\delta\Ein}{^{(0)}_\alpha^\alpha}
			&=-\tfrac{1}{2}(n+2)(n\tensor*{\psi}{^{(1)}_0_0}-2\tensor{\psi}{^{(1)}_\alpha^\alpha})\\
			&\phantom{=\;}
			+i(n+2)\rho(\tensor{\nabla}{^\alpha}\tensor*{\psi}{^{(0)}_0_\alpha}
				-\tensor{\nabla}{^\conjalpha}\tensor*{\psi}{^{(0)}_0_\conjalpha})
			-\tfrac{1}{2}n\rho^2(\tensor{\Phi}{^\alpha^\beta}\tensor*{\psi}{^{(0)}_\alpha_\beta}
				+\tensor{\Phi}{^\conjalpha^\conjbeta}\tensor*{\psi}{^{(0)}_\conjalpha_\conjbeta})\\
			&\phantom{=\;}
			-\rho^2(\tensor{\nabla}{^\alpha}\tensor{\nabla}{^\beta}\tensor*{\psi}{^{(0)}_\alpha_\beta}
				+\tensor{\nabla}{^\conjalpha}\tensor{\nabla}{^\conjbeta}\tensor*{\psi}{^{(0)}_\conjalpha_\conjbeta}
				+\tensor{N}{^\gamma^\alpha^\beta}\tensor{\nabla}{_\gamma}\tensor*{\psi}{^{(0)}_\alpha_\beta}
				+\tensor{N}{^\conjgamma^\conjalpha^\conjbeta}
					\tensor{\nabla}{_\conjgamma}\tensor*{\psi}{^{(0)}_\conjalpha_\conjbeta}\\
			&\phantom{=-\rho^2(}
				+\tensor{N}{^\gamma^\alpha^\beta_,_\gamma}\tensor*{\psi}{^{(0)}_\alpha_\beta}
				+\tensor{N}{^\conjgamma^\conjalpha^\conjbeta_,_\conjgamma}
					\tensor*{\psi}{^{(0)}_\conjalpha_\conjbeta})
			+O(\rho^{2n+5}).
	\end{split}
	\end{equation*}
	Hence if we can set $\tensor*{\psi}{^{(0)}_0_\alpha}$ and $\tensor*{\psi}{^{(0)}_\alpha_\beta}$ appropriately,
	then $\tensor*{\psi}{^{(1)}_0_0}-2\tensor{\psi}{^{(1)}_\alpha^\alpha}$ must be $O(\rho^{2n+5})$,
	and hence both $\tensor*{\psi}{^{(1)}_0_0}$ and $\tensor{\psi}{^{(1)}_\alpha^\alpha}$ must be $O(\rho^{2n+5})$.
	Let $\tensor*{\psi}{^{(0)}_0_\alpha}=\rho^{2n+3}\tensor{\nu}{_\alpha}+O(\rho^{2n+4})$ and
	$\tensor*{\psi}{^{(0)}_\alpha_\beta}=\rho^{2n+2}\tensor{\mu}{_\alpha_\beta}+O(\rho^{2n+3})$.
	Combined with \eqref{eq:PDEForOneAndTwoTensors}, the equations to be solved are
	\begin{equation}
		\label{eq:LastSystemOfPDEsOnTwoTensor}
		\begin{cases}
			&(n+2)(\tensor{\nabla}{^\alpha}\tensor{\nu}{_\alpha}
			+\tensor{\nabla}{^\conjalpha}\tensor{\nu}{_\conjalpha})
			-(n+1)(\tensor{A}{^\alpha^\beta}\tensor{\mu}{_\alpha_\beta}
			+\tensor{A}{^\conjalpha^\conjbeta}\tensor{\mu}{_\conjalpha_\conjbeta})\\
			&\quad =-\tensor{E}{_\infty_0}
			-\frac{2}{n+2}(\tensor{\nabla}{^\alpha}\tensor{E}{_0_\alpha}
			+\tensor{\nabla}{^\conjalpha}\tensor{E}{_0_\conjalpha})
			+\frac{2}{n+1}(\tensor{A}{^\alpha^\beta}\tensor{E}{_\alpha_\beta}
			+\tensor{A}{^\conjalpha^\conjbeta}\tensor{E}{_\conjalpha_\conjbeta}),\\
			&-i(n+2)\tensor{\nu}{_\alpha}
			+(n+1)\tensor{\nabla}{^\beta}\tensor{\mu}{_\alpha_\beta}
			+(n+1)\tensor{N}{_\alpha^\conjbeta^\conjgamma}\tensor{\mu}{_\conjbeta_\conjgamma}\\
			&\quad =-\tensor{E}{_\infty_\alpha}
			+\frac{2}{n+2}i\tensor{E}{_0_\alpha}
			-\frac{2}{n+1}(\tensor{\nabla}{^\beta}\tensor{E}{_\alpha_\beta}
			+\tensor{N}{_\alpha^\conjbeta^\conjgamma}\tensor{E}{_\conjbeta_\conjgamma}),\\
			&i(n+2)(\tensor{\nabla}{^\alpha}\tensor{\nu}{_\alpha}
				-\tensor{\nabla}{^\conjalpha}\tensor{\nu}{_\conjalpha})
				-\tfrac{1}{2}n(\tensor{\Phi}{^\alpha^\beta}\tensor{\mu}{_\alpha_\beta}
				+\tensor{\Phi}{^\conjalpha^\conjbeta}\tensor{\mu}{_\conjalpha_\conjbeta})
				-\tensor{\nabla}{^\alpha}\tensor{\nabla}{^\beta}\tensor{\mu}{_\alpha_\beta}
				-\tensor{\nabla}{^\conjalpha}\tensor{\nabla}{^\conjbeta}\tensor{\mu}{_\conjalpha_\conjbeta}\\
			&-\tensor{N}{^\gamma^\alpha^\beta}\tensor{\nabla}{_\gamma}\tensor{\mu}{_\alpha_\beta}
				-\tensor{N}{^\conjgamma^\conjalpha^\conjbeta}
					\tensor{\nabla}{_\conjgamma}\tensor{\mu}{_\conjalpha_\conjbeta}
				-\tensor{N}{^\gamma^\alpha^\beta_,_\gamma}\tensor{\mu}{_\alpha_\beta}
				-\tensor{N}{^\conjgamma^\conjalpha^\conjbeta_,_\conjgamma}\tensor{\mu}{_\conjalpha_\conjbeta}\\
			&\quad=-n\tensor{E}{_0_0}+2\tensor{E}{_\alpha^\alpha}.
		\end{cases}
	\end{equation}
	By substituting the second equation into the other two and using \eqref{eq:PhiTensorInLowestTerms},
	the system is reduced to
	\begin{equation*}
		\begin{cases}
			\tensor{D}{^\alpha^\beta}\tensor{\mu}{_\alpha_\beta}
			-\tensor{D}{^\conjalpha^\conjbeta}\tensor{\mu}{_\conjalpha_\conjbeta}=iu,\\
			\tensor*{D}{_{-2/n}^\alpha^\beta}\tensor{\mu}{_\alpha_\beta}
			+\tensor*{D}{_{-2/n}^\conjalpha^\conjbeta}\tensor{\mu}{_\conjalpha_\conjbeta}
			=v.
		\end{cases}
	\end{equation*}
	So we set $\tensor{\mu}{_\alpha_\beta}=\tensor{\lambda}{_\alpha_\beta}$ and determine $\tensor{\nu}{_\alpha}$
	by the second equation of \eqref{eq:LastSystemOfPDEsOnTwoTensor}.
	Then $\tensor{\Ein}{_\infty_0}=\caA^{2n+5}$, $\tensor{\Ein}{_\infty_\alpha}=\caA^{2n+4}$ and
	$n\tensor{\Ein}{_0_0}-2\tensor{\Ein}{_\alpha^\alpha}=\caA^{2n+5}$ are solved by
	$\tensor*{\psi}{^{(1)}_0_0}=O(\rho^{2n+5})$, $\tensor{\psi}{^{(1)}_\alpha^\alpha}=O(\rho^{2n+5})$.
	As before, $\tf(\tensor*{\psi}{^{(0)}_\alpha_\conjbeta})$ mod $O(\rho^{2n+5})$ and
	$\tfrac{1}{2}n\tensor*{\psi}{^{(0)}_0_0}+(n+1)\tensor{\psi}{^{(0)}_\alpha^\alpha}$ mod $O(\rho^{2n+5})$
	are uniquely determined so that $\tensor*{\Ein}{_0_0}=\caA^{2n+5}$,
	$\tensor*{\Ein}{_\alpha_\conjbeta}=\caA^{2n+5}$.
	We set $\tensor*{\psi}{^{(0)}_0_0}=\rho^{2n+4}\kappa+O(\rho^{2n+5})$.
	By \eqref{eq:ContractedBianchiOriginal1} we have $\tensor{\Ein}{_\infty_\infty}=\caA^{2n+5}$.

	Now we have constructed a normal-form ACH metric $g$, which is log-free,
	satisfying $\tensor{\Ein}{_I_J}=\caA^{2n+2+a(I,J)}$ and \eqref{eq:PrescriptionOfACHCoefficients2}
	in a unique way.
	After that we once again follow the latter half of the proof of Theorem \ref{thm:FormalSolution} to
	determine all the higher-order terms of $\tensor{g}{_i_j}$. No logarithmic terms occur in this process.
\end{proof}


\begin{thebibliography}{EMM}
	\bibitem[BaD]{BarlettaDragomir} E. Barletta and S. Dragomir, \textit{Differential equations on contact Riemannian manifolds,} Ann. Scoula Norm. Sup. Pisa Cl. Sci. (4) \textbf{30} (2001), 63--95.
	\bibitem[BlD]{BlairDragomir} D. E. Blair and S. Dragomir, \textit{Pseudohermitian geometry on contact Riemannian manifolds,} Rend. Mat. Appl. (7) \textbf{22} (2002), 275--341.
	\bibitem[Bi]{Biquard} O. Biquard, \textit{M\'{e}triques d'Einstein asymptotiquement sym\'{e}triques,} Ast\'{e}risque 265, 2000, vi+109 pp.
	\bibitem[BH1]{BiquardHerzlichACHE} O. Biquard and M. Herzlich, \textit{A Burns-Epstein invariant for ACHE 4-manifolds,} Duke Math. J. \textbf{126} (2005), 53--100.
	\bibitem[BH2]{BiquardHerzlichAsymptotics} O. Biquard and M. Herzlich, \textit{Analyse sur un demi-espace hyperbolique et poly-homog\'en\'eit\'e locale,} preprint, \texttt{arXiv:1002.4106v1}.
	\bibitem[Br]{Branson} T. P. Branson, \textit{Sharp inequalities, the functional determinant, and the complementary series,} Trans. Amer. Math. Soc. \textbf{347} (1995), 3671--3742.
	\bibitem[\v{C}G]{CapGover2008} A. \v{C}ap and A. R. Gover, \textit{CR-tractors and the Fefferman space,} Indiana Univ. Math. J. \textbf{57} (2008), 2519--2570.
	\bibitem[\v{C}Sc]{CapSchichl} A. \v{C}ap and H. Schichl, \textit{Parabolic geometries and canonical Cartan connections,} Hokkaido Math. J. \textbf{29} (2000), 453--505.
	\bibitem[\v{C}Sl]{CapSlovak} A. \v{C}ap and J. Slov\'ak, \textit{Parabolic Geometries I: Background and General Theory,} Mathematical Surveys and Monographs 154, 2009, x+628 pp.
	\bibitem[CY]{ChengYau} S. Y. Cheng and S. T. Yau, \textit{On the existence of a complete K\"{a}hler metric on noncompact complex manifolds and the regularity of Fefferman's equation,} Comm. Pure Appl. Math. \textbf{33} (1980), 507--544.
	\bibitem[EMM]{EpsteinMelroseMendoza} C. L. Epstein, R. B. Melrose and G. A. Mendoza, \textit{Resolvent of the Laplacian on strictly pseudoconvex domains,} Acta Math. \textbf{167} (1991), 1--106.
\bibitem[Fa]{Farris} F. Farris, \textit{An intrinsic construction of Fefferman's CR metric,} Pacific J. Math. \textbf{123} (1986), 33--45.
	\bibitem[Fe]{Fefferman} C. L. Fefferman, \textit{Monge--Amp\`{e}re equations, the Bergman kernel, and geometry of pseudoconvex domains,} Ann. of Math. (2) \textbf{103} (1976), 395--416; erratum \textbf{104} (1976), 393--394.
	\bibitem[FG1]{FeffermanGrahamAsterisque} C. Fefferman and C. R. Graham, \textit{Conformal invariants,} in \textit{The mathematical heritage of \'{E}lie Cartan (Lyon, 1984),} Ast\'{e}risque, 1985, hors s\'{e}rie, 95--116.
	\bibitem[FG2]{FeffermanGrahamQ} C. Fefferman and C. R. Graham, \textit{$Q$-curvature and Poincar\'e metrics,} Math. Res. Lett. \textbf{9} (2002), 139-151.
	\bibitem[FG3]{FeffermanGrahamAmbient} C. Fefferman and C. R. Graham, \textit{The ambient metric,} preprint, \texttt{arXiv:0710.0919v2}.
	\bibitem[GG]{GoverGraham} A. R. Gover and C. R. Graham, \textit{CR invariant powers of the sub-Laplacian,} J. Reine Angew. Math. \textbf{583} (2005), 1--27.
	\bibitem[G]{Graham} C. R. Graham, \textit{Higher asymptotics of the complex Monge--Amp\`{e}re equation,} Compositio Math. \textbf{64} (1987), 133--155.
	\bibitem[GH]{GrahamHirachi} C. R. Graham and K. Hirachi, \textit{The ambient obstruction tensor and $Q$-curvature,} in \textit{AdS/CFT Correspondence: Einstein Metrics and Their Conformal Boundaries,} IRMA Lect. Math. Theo. Phys. 8, Eur. Math. Soc., Z\"urich, 2005, 59--71.
	\bibitem[GL]{GrahamLee} C. R. Graham and J. M. Lee, \textit{Einstein metrics with prescribed conformal infinity on the ball,} Adv. Math. \textbf{87} (1991), 186--225.
	\bibitem[GZ]{GrahamZworski} C. R. Graham and M. Zworski, \textit{Scattering matrix in conformal geometry,} Invent. Math. \textbf{152} (2003), 89--118.
	\bibitem[GS]{GuillarmouSaBarreto} C. Guillarmou and A. S\'a Barreto, \textit{Scattering and inverse scattering on ACH manifolds,} J. Reine Angew. Math. \textbf{622} (2008), 1--55.
	\bibitem[HPT]{HislopPerryTang} P. D. Hislop, P. A. Perry and S.-H. Tang, \textit{CR-invariants and the scattering operator for complex manifolds with boundary,} Anal. PDE \textbf{1} (2008), 197--227.
	\bibitem[K]{Kuranishi} M. Kuranishi, \textit{PDEs associated to the CR embedding theorem,} in \textit{Analysis and Geometry in Several Complex Variables (Katata, 1997),} Trends Math., Birkh\"auser, Boston, 1999, 129--157.
	\bibitem[L1]{LeeFeffermanMetric} J. M. Lee, \textit{The Fefferman metric and pseudohermitian invariants,} Trans. Amer. Math. Soc. \textbf{296} (1986), 411--429.
	\bibitem[L2]{LeePseudoEinstein} J. M. Lee, \textit{Pseudo-Einstein structures on CR manifolds,} Amer. J. Math. \textbf{110} (1988), 157--178.
	\bibitem[LM]{LeeMelrose} J. Lee and R. Melrose, \textit{Boundary behaviour of the complex Monge--Amp\`{e}re equation,} Acta Math. \textbf{148} (1982), 159--192.
	\bibitem[M]{Mizner} R. I. Mizner, \textit{Almost CR structures, $f$-structures, almost product structures and associated connections,} Rocky Mountain J. Math. \textbf{23} (1993), 1337--1359.
	\bibitem[S]{Seshadri} N. Seshadri, \textit{Approximately Einstein ACH metrics,} Bull. Soc. Math. France \textbf{137} (2009), 63--91.
	\bibitem[Tnk]{Tanaka} N. Tanaka, \textit{A differential geometric study on strongly pseudo-convex manifolds,} Lectures in Mathematics 9, Department of Mathematics, Kyoto Univ., Kinokuniya, Tokyo, 1975.
	\bibitem[Tno]{Tanno} S. Tanno, \textit{Variational problems on contact Riemannian manifolds,} Trans. Amer. Math. Soc. \textbf{314} (1989), 349--379.
	\bibitem[W]{Webster} S. M. Webster, \textit{Pseudo-hermitian structures on a real hypersurface,}  J. Differential Geom. \textbf{13} (1978), 25--41.
\end{thebibliography}
\end{document}